\documentclass{amsart}
\usepackage{amscd,amssymb,amsmath,amsthm,amsfonts}
\usepackage[a4paper, margin=2.5cm]{geometry} 
\date{April 4, 2023}

\numberwithin{equation}{section}

\newcommand{\degofz}{j}
\newcommand{\cyl}{cyl}

\newcommand{\prob}{\operatorname{P}}
\newcommand{\probg}{\operatorname{P}_{\!g}^+}
\newcommand{\probgn}{\operatorname{P}_{\!g,n}}

\newcommand{\relcontgn}{p_1(\cQ_{g,n})}

\newcommand{\MeandNumber}{\operatorname{M}}
\newcommand{\MulticurvesNumber}{\operatorname{G}}
\newcommand{\nbigons}{n}
\newcommand{\narcs}{k}

\newcommand{\cST}{\mathcal{ST}}
\newcommand{\cSTg}{\mathcal{ST}{\hspace*{-3pt}}_{g}}
\newcommand{\cSTgn}{\mathcal{ST}{\hspace*{-3pt}}_{g,n}}
\newcommand{\cSTAb}{\cST^{Ab}}
\newcommand{\cSTgAb}{\cSTg^{\;Ab}}
\newcommand{\Card}{\operatorname{card}}
\newcommand{\card}{\operatorname{card}}

\newcommand{\Vol}{\operatorname{Vol}}
\newcommand{\dVolMV}{d\!\Vol}
\newcommand{\Area}{\operatorname{Area}}

\newcommand{\CP}{{\mathbb C}\!\operatorname{P}^1}

\renewcommand{\epsilon}{\varepsilon}

\newcommand\N{\mathbb N}

\newcommand{\C}{\mathbb{C}}

\newcommand{\cH}{\mathcal{H}}
\newcommand{\cM}{\mathcal{M}}

\newcommand{\cQ}{{\mathcal Q}}

\newcommand{\cG}{\mathcal{G}}

\newcommand{\cZ}{\mathcal{Z}}

\newcommand{\biggglB}[1]{\left\{\!\parbox{0pt}{\rule{0pt}{#1}}\right.}
\newcommand{\bigggrB}[1]{\left.\parbox{0pt}{\rule{0pt}{#1}}\!\right\}}

\newlength{\halfbls}

\newtheorem{Theorem}{Theorem}[section]
\newtheorem*{NNTheorem}{Theorem}
\newtheorem{Proposition}[Theorem]{Proposition}
\newtheorem*{NNProposition}{Proposition}
\newtheorem{Lemma}[Theorem]{Lemma}
\newtheorem{Corollary}[Theorem]{Corollary}
\newtheorem{Conjecture}[Theorem]{Conjecture}

\theoremstyle{remark}

\newtheorem{Example}[Theorem]{Example}
\newtheorem{Remark}[Theorem]{Remark}
\newtheorem{Definition}[Theorem]{Definition}

\title{Higher genus meanders and Masur--Veech volumes}

\author[V.~Delecroix]{Vincent Delecroix}
\thanks{Research of the first two authors is partially supported by
the grant ANR-19-CE40-0003.}
\address{
LaBRI,
Domaine universitaire,
351 cours de la Lib\'eration, 33405 Talence, FRANCE
}
\email{20100.delecroix@gmail.com}

\author[\'E.~Goujard]{\'Elise Goujard}
\address{
IMB, Univ. de Bordeaux,
351 cours de la Lib\'eration, 33405 Talence, FRANCE
et Institut Universitaire de France
}
\email{elise.goujard@gmail.com}

\author[P.~G.~Zograf]{Peter~Zograf}
\thanks{
Research of the third author
is partially supported by Ministry of Science and Higher Education
of the Russian Federation, agreement \textnumero 075–15–2022–289.
The results of Section~\ref{sec:flat}
were obtained at SPbU
under support of
the RSF grant 19-71-30002.}
\address{
St.~Petersburg Department, Steklov Math. Institute, Fontanka 27,
St. Petersburg 191023, and Chebyshev Laboratory,
St. Petersburg State University, 14th
Line V.O. 29B, St.Petersburg 199178 Russia}
\email{zograf@pdmi.ras.ru}

\author[A.~Zorich]{Anton Zorich}
\address{
Institut de Math\'ematiques de Jussieu --
Paris Rive Gauche,
Case 7012,
8 Place Aur\'elie Nemours,
75205 PARIS Cedex 13, France}
\email{anton.zorich@gmail.com}

\begin{document}

\begin{abstract}
A classical meander is a pair consisting of a straight line in the
plane and of a smooth closed curve transversally intersecting the
line, where the pair is considered up to an isotopy preserving the
straight line. The number $\MeandNumber(N)$ of meanders with $2N$
intersections grows exponentially with $N$, but asymptotics still
remains conjectural.

A meander defines a pair of transversally intersecting simple closed
curves on a 2-sphere.  In this paper we consider pairs of
transversally intersecting simple closed curves on a closed oriented
surface of arbitrary genus $g$. The number of such higher genus
meanders still admits exponential upper and lower bounds as the
number of intersections grows. Fixing the number $\nbigons$ of bigons
in the complement to the union of the two curves, we compute the
precise asymptotics of genus $g$ meanders with $\nbigons$ bigons and
with at most $2N$ intersections and show that this asymptotics is
\textit{polynomial} in $N$ as $N\to\infty$. We obtain a similar
result for the number of positively intersecting pairs of oriented
simple closed curves on a surface of genus $g$. We also compute the
asymptotic probability of getting a meander from a random braid on a
surface of genus $g-1$ with two boundary components.

In order to effectively count meanders we identify them with integer
points represented by certain square-tiled surfaces in the moduli
spaces of Abelian and quadratic differentials and make use of recent
advances in the geometry of these moduli spaces combined with
asymptotic properties of Witten--Kontsevich $2$-correlators on moduli
spaces of complex curves.
\end{abstract}

\maketitle
\tableofcontents


\section{Introduction}

\subsection{Classical meanders}
A \textit{meander} is a topological configuration of an oriented
straight line in the plane and of a smooth  simple closed curve
intersecting transversally the straight line considered up to an
isotopy of the plane preserving the straight line. Meanders can be
traced back to H.~Poincar\'e~\cite{Poincare} and naturally appear in
various areas of mathematics, theoretical physics and computational
biology (in particular, they provide a model of polymer
folding~\cite{DiFrancesco:Golinelli:Guitter}).

\begin{figure}[hbt]
\includegraphics{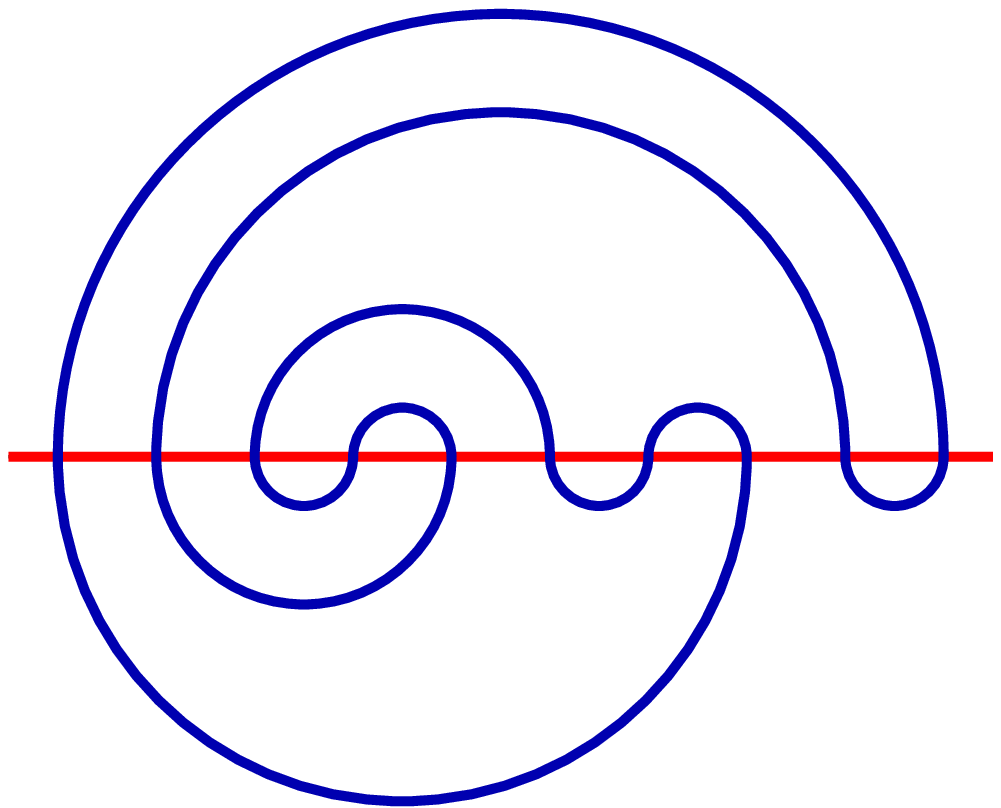}
\includegraphics{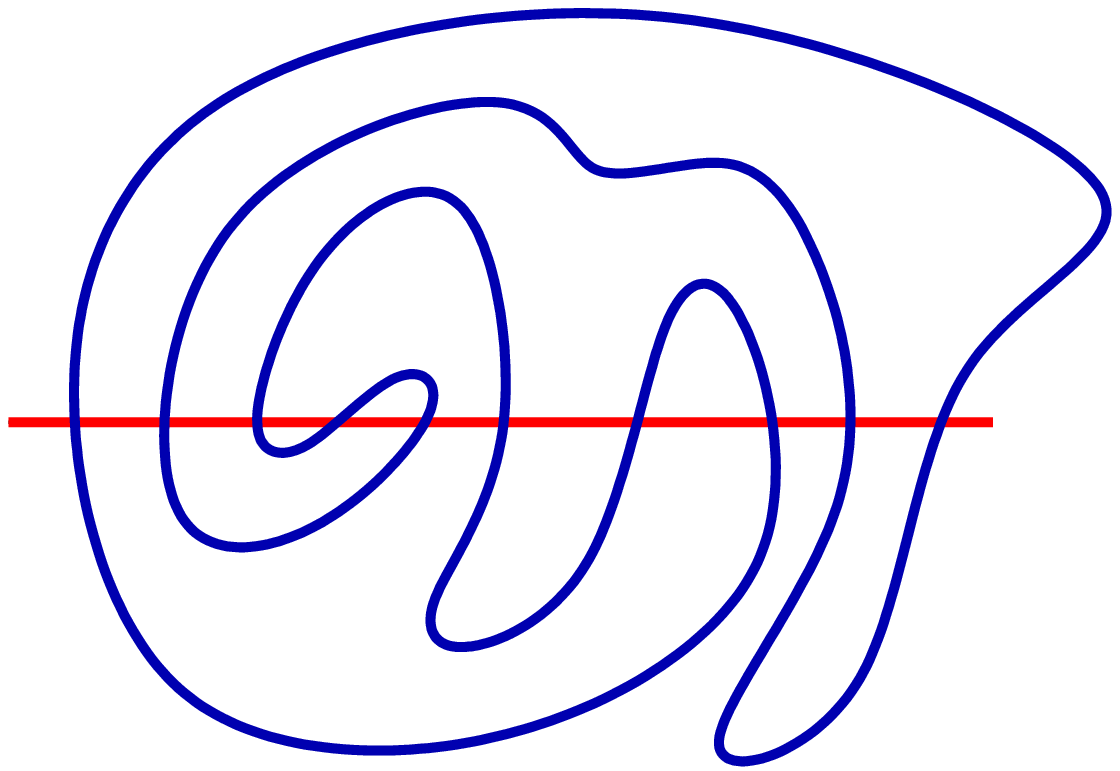}
\vspace{80pt}
\caption{Equivalent meanders with 10 crossings}
\end{figure}

The asymptotic count of the number $\MeandNumber(N)$ of meanders with
exactly $2N$ crossings as $N$ tends to infinity is one of the oldest
open questions in the study of meanders. The problem was popularized
by V.~I.~Arnold (see Problem~1986-7 in~\cite{Arnold} and later
comments by M.~Kontsevich and S.~Lando in the same book).  Exponential
upper and lower bounds for this number were obtained by S.~Lando and
A.~Zvonkin in~\cite{Lando:Zvonkin:92} and in~\cite{Lando:Zvonkin:93}.
They conjectured that there exist constants $const, R, \alpha$ such
that
\begin{equation}
\label{eq:menader:count:conjecture}
\MeandNumber(N)\overset{?}{\sim}
const\cdot R^{2N}\cdot N^{\alpha}
\quad\text{as}\quad N\to\infty\,.
\end{equation}
The conjecture was sharpened by P.~Di~Francesco, O.~Golinelli,
E.~Guitter~\cite{DiFrancesco:Golinelli:Guitter},
\cite{DiFrancesco:Golinelli:Guitter:2}, who
described the generating function of meandric numbers
$\MeandNumber(N)$. They suggested
in~\cite{DiFrancesco:Golinelli:Guitter:2} a conjectural exact value
$\alpha=-\frac{29+\sqrt{145}}{12}\approx-3.42$ interpreted as the
corresponding critical exponent $\alpha$ in a two-dimensional
conformal field theory with central charge $c=-4$ coupled to gravity.
The  conjectural approximate value $R^2\approx 12.26$ was suggested
by I.~Jensen~\cite{Jensen} through computer simulations. The best
known rigorous bounds for the constant $R^2$ are $11.380 \le R^2 \le
12.901$, as proved in~\cite{Albert:Paterson}. However, all elements
of this conjecture stated thirty years ago are still open.

Mathematical literature devoted to meanders is vast and varies from
representation theory, see~\cite{Dergachev:Kirillov},
\cite{Duflo:Yu}, \cite{Elashvili:Jibladze},
\cite{Delecroix:lianders},
and theory of PDEs~\cite{Fiedler:Rocha}
to theoretical physics~\cite{DiFrancesco:Duplantier:Golinelli:Guitter}
and more recently
to Schramm--Loewner evolution curves on a Liouville quantum gravity
surface~\cite{Borga:Gwynne:Sun}. Meanders are particular cases
of more general \textit{meandric systems}, recently
studied in~\cite{Curien:Kozma:Sidoraviciu:Tournier},
\cite{Feray:Thevenin} \cite{Fukuda:Nechita},
\cite{Goulden:Nica:Puder}, \cite{Kargin}. We recommend a beautiful
recent survey~\cite{Zvonkin} on meanders for further details and
references.

One can organize meanders into groups and count them
group by group. For example, one can fix the number $\nbigons$ of
minimal arcs (marked by black color in
Figure~\ref{fig:meander:types}) and count separately the number
$\MeandNumber^+_{0,\nbigons}(N)$ (respectively
$\MeandNumber^-_{0,\nbigons}(N)$) of meanders with at most $2N$
crossings, exactly $\nbigons$ minimal arcs and having (respectively
not having)  a maximal arc.

\begin{figure}[hbt]
\includegraphics{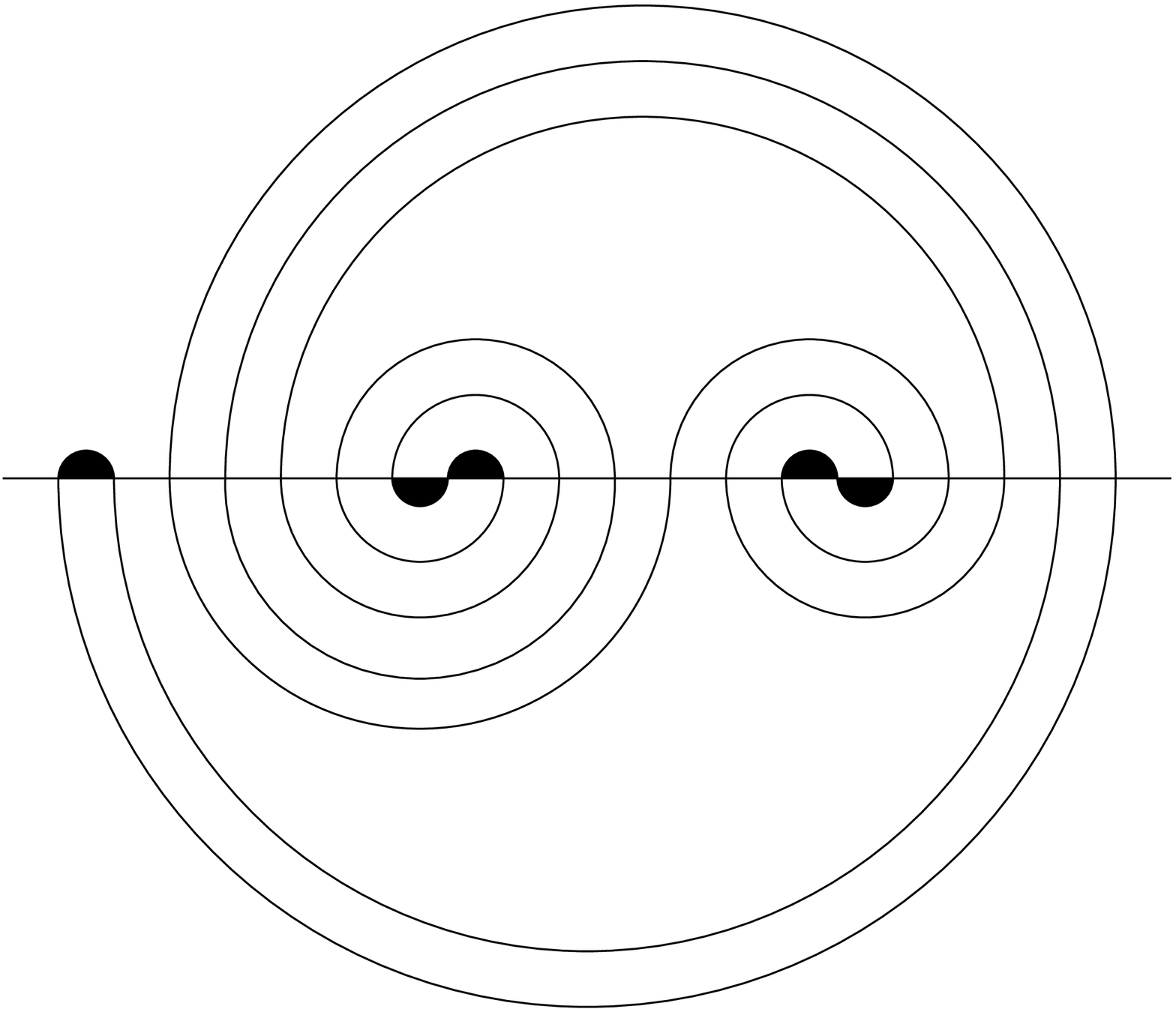}
\includegraphics{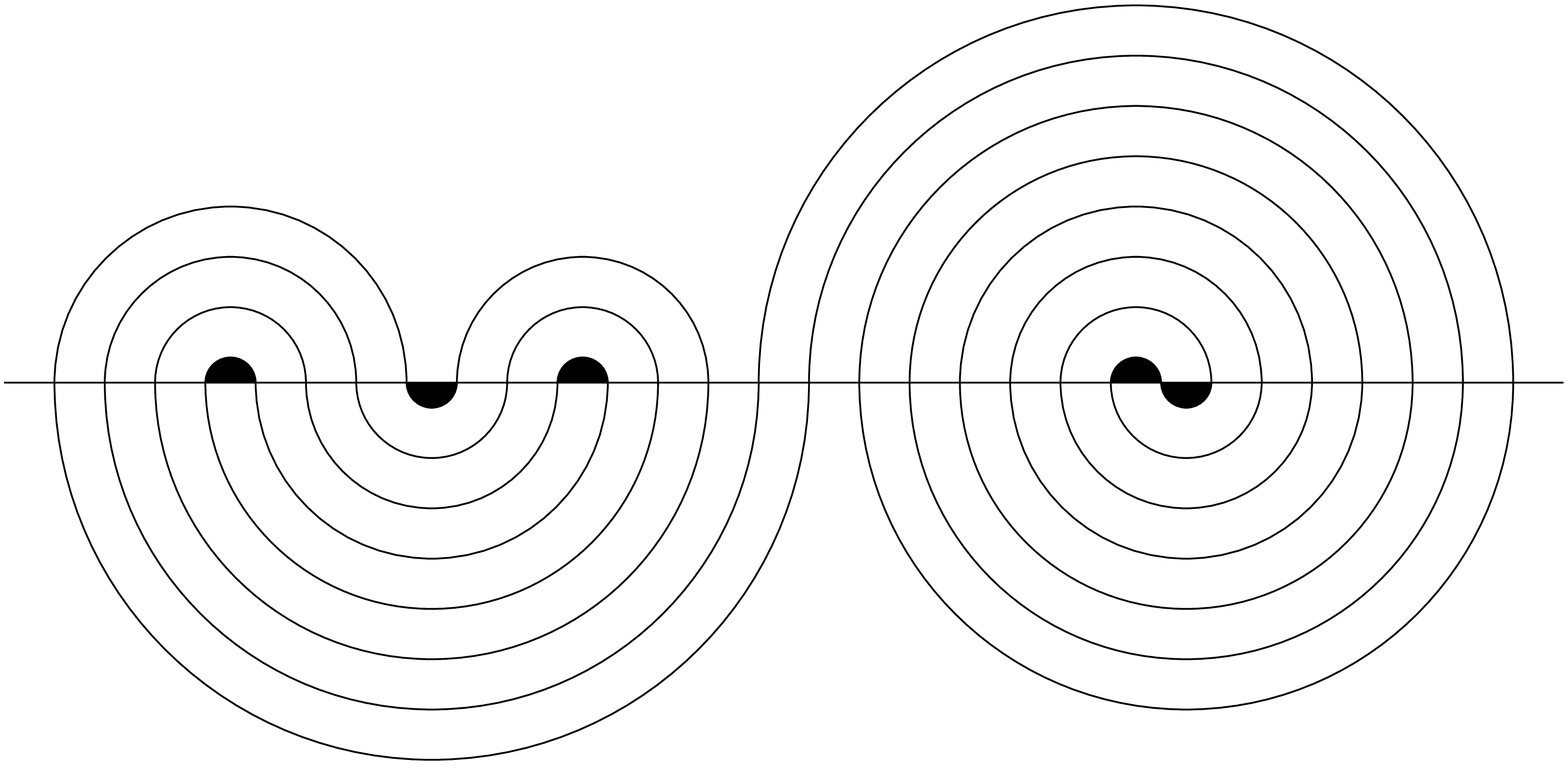}
\begin{picture}(260,20)(30,-15) 
\put(20,-100){Contributes to $\MeandNumber^+_{0,5}(N)$}
\put(200,-100){Contributes to $\MeandNumber^-_{0,5}(N)$}
\end{picture}
\vspace{80bp} 
\caption{
\label{fig:meander:types}
Meanders with and without maximal arcs.
Both meanders have $5$ minimal arcs.
}
\end{figure}

We proved in~\cite{DGZZ-meander} that the counting functions
$\MeandNumber^+_{0,\nbigons}(N)$ and $\MeandNumber^-_{0,\nbigons}(N)$
admit the following asymptotics as $N\to+\infty$:
\begin{align*}
\MeandNumber^+_{0,\nbigons}(N) &=
\frac{2}{ \nbigons!\, (\nbigons-3)!}
\left(\frac{2}{\pi^2}\right)^{\nbigons-2}\cdot
\binom{2\nbigons-2}{\nbigons-1}^2\cdot \frac{N^{2\nbigons-4}}{4\nbigons-8}\ +\ o(N^{2\nbigons-4})\,.
\\
\MeandNumber^-_{0,\nbigons}(N) &=
\frac{4}{ \nbigons!\, (\nbigons-4)!}
\left(\frac{2}{\pi^2}\right)^{\nbigons-3}\cdot
\binom{2\nbigons-4}{\nbigons-2}^2
\cdot \frac{N^{2\nbigons-5}}{4\nbigons-10}\ +\ o(N^{2\nbigons-5})\,.
\end{align*}

This restricted count giving \textit{polynomial} asymptotics for
$\MeandNumber^\pm_{0,\nbigons}(N)$ versus \textit{exponential}
asymptotics for $\MeandNumber(N)$ neither contradicts nor
corroborates conjecture~\eqref{eq:menader:count:conjecture}. A
meander with $2N$ crossings can have from $3$ to $2N-1$ minimal arcs,
so
$\MeandNumber(N)=\sum_{\nbigons=3}^{2N-1}\MeandNumber^+_{0,\nbigons}(N)
+\sum_{\nbigons=4}^{2N-2}\MeandNumber^-_{0,\nbigons}(N)$.
However, the sum of asymptotic expressions on the right-hand sides of
the above formulas for $\MeandNumber^\pm_{0,\nbigons}$ has no
relation to $\MeandNumber(N)$. The problem is that, conjecturally,
roughly half of the arcs of a typical meander with large number of
crossings are minimal, while in the asymptotic formulas for
$\MeandNumber^\pm_{0,\nbigons}(N)$ we fix $\nbigons$ and only then
let $N\to+\infty$, so our asymptotic formulas make sense only in the
regime  when $\nbigons\ll N$.

Meanders with a fixed number $\nbigons$ of minimal arcs are related
to simple closed geodesics on a hyperbolic sphere with $\nbigons$
cusps. When the number of intersections $2N$ is large, this number
gives a reasonable approximation of the length of a simple closed
geodesic in this correspondence. M.~Mirzakhani proved
in~\cite{Mirzakhani:growth:of:simple:geodesics} that the number of
simple closed hyperbolic geodesics of bounded length $L$ has exact
polynomial asymptotics with respect to $L$, while, by classical
results of Delsarte, Huber and Selberg, the total number of closed
geodesics of length bounded by $L$ grows exponentially as $e^L/L$.
\medskip

\subsection{Higher genus meanders}

Meanders can be considered as configurations of ordered pairs of
simple closed curves on a 2-sphere, where the first
curve, corresponding to the straight line, is endowed with a marked
point distinct from intersection points with the second curve.
Applying an appropriate diffeomorphism of the sphere we can send the
first curve to a large circle on a round sphere; postcomposing this
diffeomorphism with the stereographic projection from the marked
point to the plane we get a classical meander.

In this paper, we count \textit{higher genus meanders} represented by
ordered pairs of transversally intersecting smooth simple closed
curves on a higher genus surface. As before, two pairs are considered
as equivalent if there exists an orientation preserving
diffeomorphism of the surface (not necessarily homotopic to identity)
which sends one ordered pair of curves to another pair respecting the
ordering of curves. We do not distinguish any point of the first
curve in the higher genus case.

Denote by $\cG$ be the embedded graph defined by a transverse pair of
multicurves on a surface $S$. Vertices of $\cG$ are the intersection
points of the pair of multicurves. The boundary components of the
complement $S-\cG$ correspond to closed broken lines formed by edges
of $\cG$.

\begin{Definition}
\label{def:bigons}
The boundary components of the complement $S-\cG$ formed by two edges
of $\cG$ are called \textit{bigons}. A bigon is
called \textit{filling} when it bounds a topological disc and
\textit{non-filling} otherwise.
\end{Definition}

In the genus zero case, bigons correspond to minimal arcs and are
always filling. In higher genera a bigon might bound a connected
component of $S-\cG$ having nontrivial topology; it can also
represent just one of several boundary components of a connected
component of $S-\cG$. However, we will see in Section~\ref{sec:flat}
that when the number of bigons is fixed, while the number of
intersections grows, for all but a vanishing part (as $N\to+\infty$)
of meanders all bigons are filling, and, more generally, for most of
meanders all connected components of $S-\cG$ are topological discs.

We count higher genus meanders in two settings. In the first
setting we study asymptotics of the number
$\MeandNumber_{g,\nbigons}(N)$ of meanders with exactly $\nbigons$
\textit{bigons} (generalizing minimal arcs) on a surface of genus $g$
with at most $2N$ crossings, as the bound $2N$ for the number of
crossings tends to infinity.

In the second setting we study asymptotics of the number
$\MeandNumber^+_g(N)$ of \textit{oriented} meanders, for which the
curves are oriented and have only positive transverse intersections.
Oriented meanders do not exist on a sphere. In the second setting we
fix only the genus $g$ of the surface and let the bound $N$ for the
number of crossings tend to infinity.

\begin{Remark}
Note that a higher genus meander might have odd number of
intersections (unlike spherical meanders, which always have even
number of intersections). We will see that the number of meanders
with $\nbigons$ bigons on a surface of genus $g$ with at most $2N-1$
crossings has the same asymptotics as $\MeandNumber_{g,\nbigons}(N)$
and the number of meanders with at most $N$ crossings has asymptotics
$2^{-(6g-6+2\nbigons)}\MeandNumber_{g,\nbigons}(N)$. It is convenient
to keep notation $\MeandNumber_{g,\nbigons}(N)$ for the number of
meanders with at most $2N$ (and not $N$) crossings to include the
genus zero case and to have better correspondence with count of
square-tiled surfaces. However, in the count $\MeandNumber^+_g(N)$ of
\textit{oriented} meanders we assume that the bound for the number of
crossings is $N$ and not $2N$.
\end{Remark}

\begin{figure}
\includegraphics{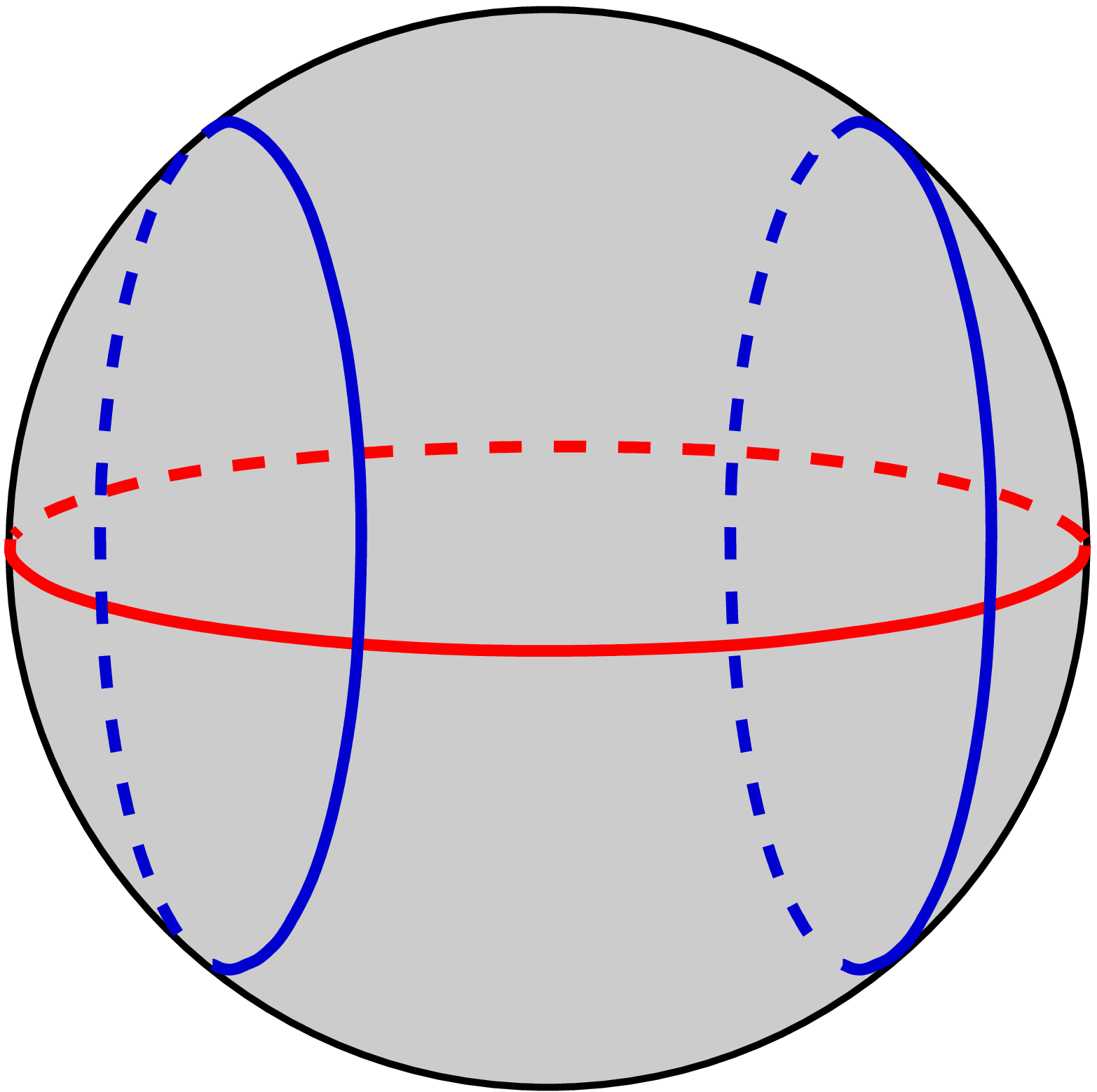}
\includegraphics{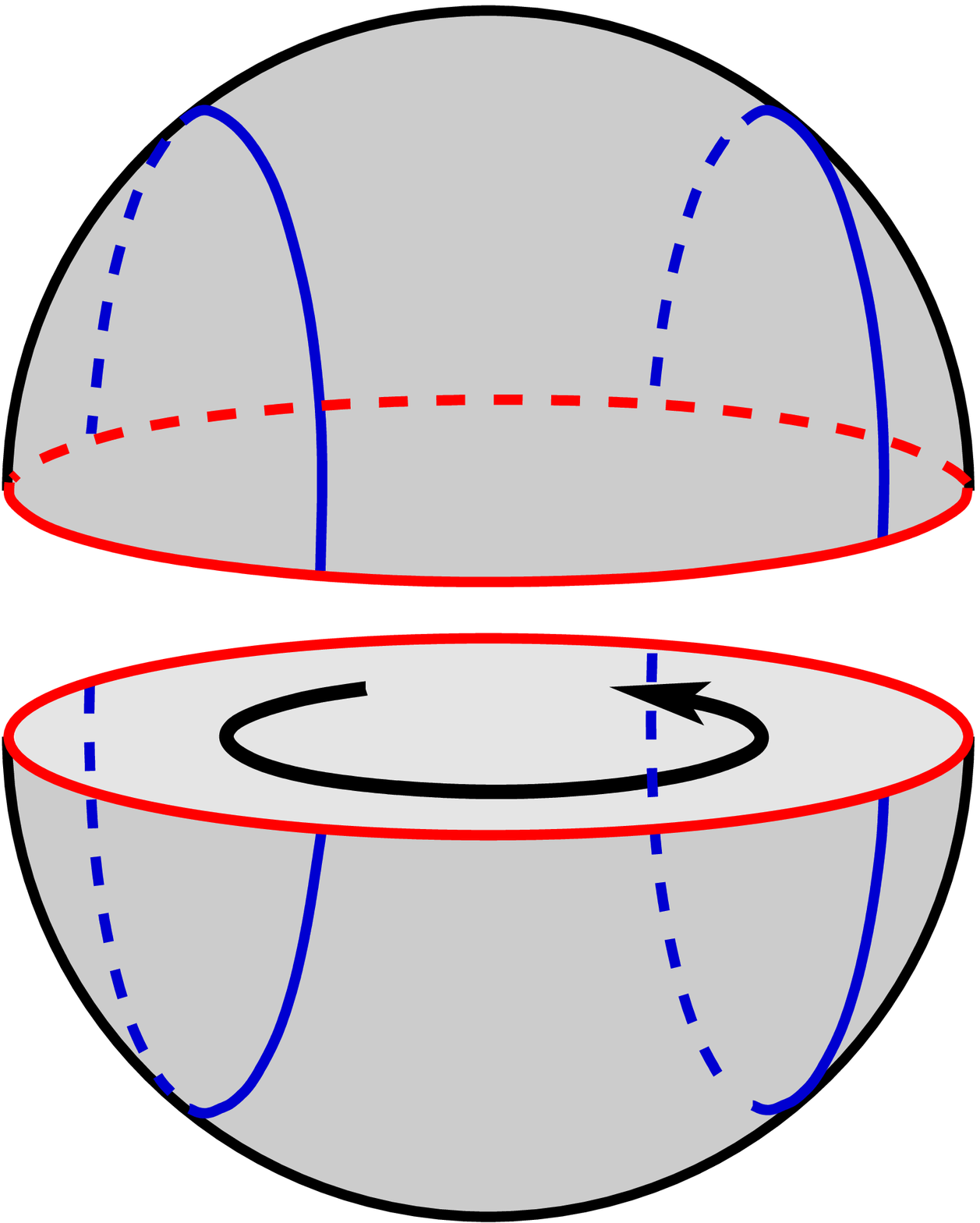}
\includegraphics{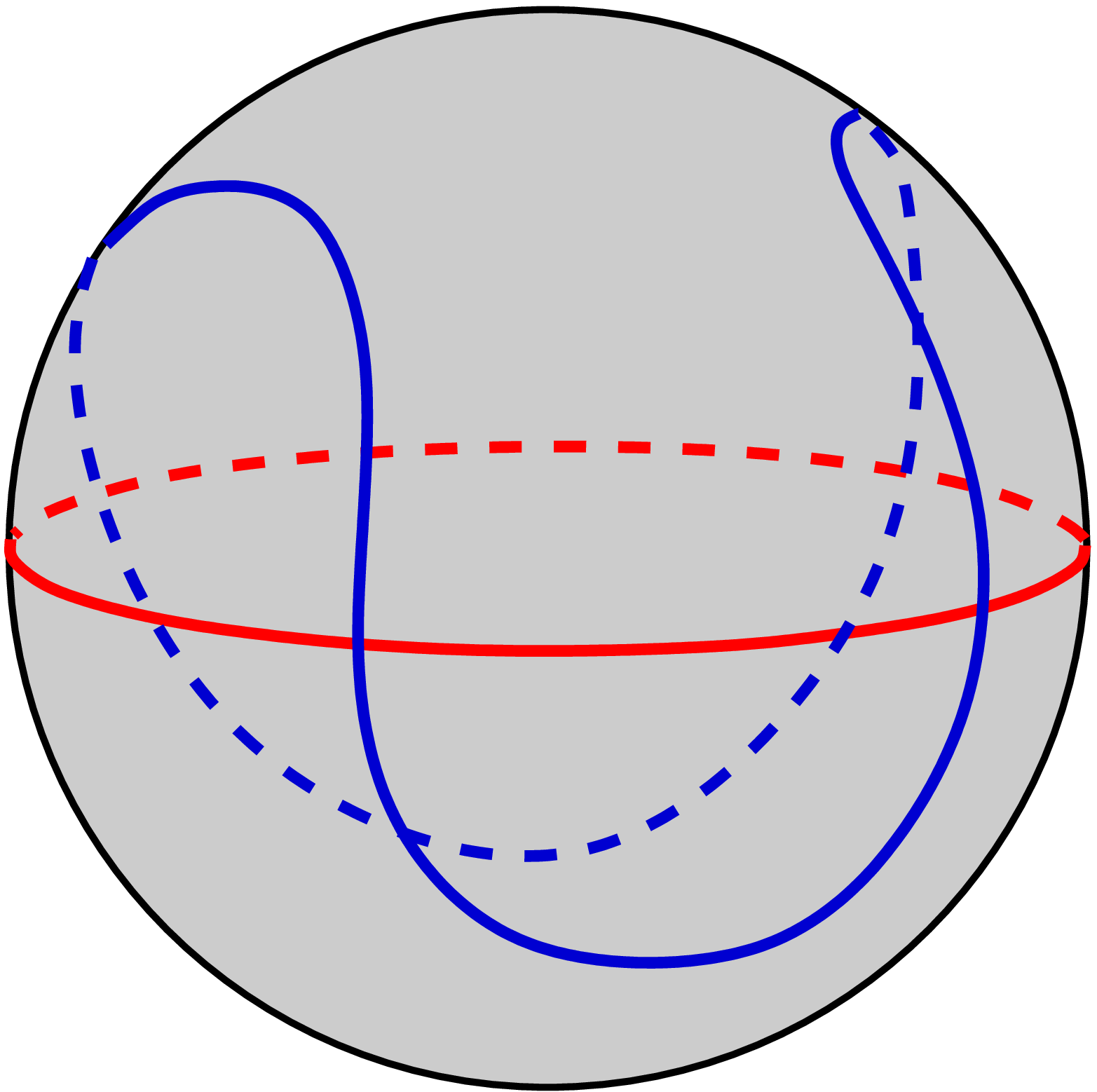}
\vspace{110bp}
\caption{
\label{fig:tennis:ball}
Identifying a pair of hemispheres, each endowed with $\narcs$
disjoint arcs, by a common equator we sometimes get a meander and
sometimes --- not.}
\end{figure}

Consider now a collection of $\narcs$ disjoint arcs on the northern
hemisphere and a collection of the same number of disjoint arcs on
the southern hemisphere. Assume that the endpoints of the arcs are
equidistant on the equators. Denote by $\nbigons$ the total number of
minimal arcs (the ones, for which the endpoints are neighbors on the
equator) on two hemispheres. We computed in~\cite{DGZZ-meander} the
asymptotic probability $\prob_{0,\nbigons}$ that a random gluing of a
random pair of arcs as above with $\narcs\le N$ gives a meander, see
Figure~\ref{fig:tennis:ball}. As in the other problems, the
asymptotics is computed for a fixed $\nbigons$ letting $N\to+\infty$.
In the current paper we derive  general formulas for probabilities to
get a meander under analogous identification of endpoints of
compatible random collections of disjoint arcs on a surface of any
genus $g$ with two boundary components. We consider this problem in
various settings and under various asymptotic regimes.

\begin{figure}[htb]
\includegraphics{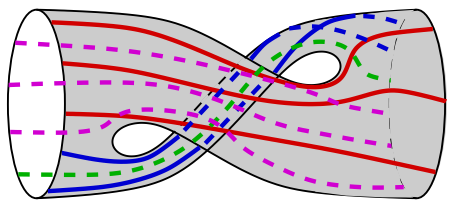}
\vspace{70bp} 
\caption{
\label{fig:oriented:arc:system}
Collection of disjoint strands joining two boundary components of a surface of genus $1$}
\end{figure}

We also consider random collections of disjoint strands on a
connected surface of genus $g-1$ with two boundary components.
Assuming that each strand goes from one component to another, as in
Figure~\ref{fig:oriented:arc:system}, we compute asymptotic
probability $\prob^+_g$ that upon a random gluing of two boundary
components matching the endpoints of strands one gets a single
connected closed curve, or, in other words, an oriented meander. For
genus $g=1$ the surface of genus $g-1$ with boundary is a cylinder
and the problems reduces to computation of asymptotic probability
that random positive integers $(k,m)$, such that $m\le k$, are
coprime. In this case the answer $\prob^+_1=\frac{6}{\pi^2}$ is
elementary. However, already for genus $g=2$ as in
Figure~\ref{fig:oriented:arc:system} we do not know any way to
compute $\prob^+_2=\frac{45}{2\pi^4}$ other than applying technique
involving the Masur--Veech volume $\Vol\cH_2$ of the moduli spaces
$\cH_2$ of Abelian differentials. This technique allows us to
produce a list of exact values of $\prob^+_g$ up to $g=1000$ in
several seconds. For large $g$ we prove the asymptotic formula
$\prob^+_g=\frac{1}{4g}\left(1+\frac{12+\pi^2}{24g}+O\left(\frac{1}{g^2}\right)\right)$.

One more open problem of Arnold, asking what is the probability that
the decomposition of a random ``interval exchange permutation'' into
disjoint cycles contains a single cycle, is closely related to
evaluation of $\prob^+_g$. This problem admits a solution by methods
developed below, but we will treat it separately to avoid overloading
the paper.

\subsection{Technique of the proofs}

We start with an idea from our work~\cite{DGZZ-meander} on meander
count in genus $0$, namely, we translate the problem into the
language of count of square-tiled surfaces closely related to
evaluation of the Masur--Veech volumes of moduli spaces of
meromorphic quadratic differentials. However, while
in~\cite{DGZZ-meander} this translation, basically, completes the
solution of the problem, in the current paper it serves as a starting
point. Combinatorics of graphs on a sphere is in certain aspects much
simpler than on higher genus surfaces. In particular, the count of
single-band square-tiled surfaces is  simple only in the case of
genus $0$. Furthermore, by work of J.~Athreya, A.~Eskin and
A.~Zorich~\cite{AEZ:genus:0} the Masur--Veech volume of any stratum
in genus $0$ is given by a simple closed formula, while in higher
genera an efficient algorithm providing explicit volumes of general
strata, different from the principal one, is not known yet beyond
strata of dimension $12$. (Volumes of all low-dimensional strata of
quadratic differentials were evaluated by
E.~Goujard~\cite{Goujard:volumes} based on the general algorithm
developed by A.~Eskin and
A.~Okounkov~\cite{Eskin:Okounkov:pillowcase}.) By these reasons the
results of the current paper were out of reach at the time
when~\cite{DGZZ-meander} was written.

In order to study higher genus meanders, we apply recent technique of
evaluation of Masur--Veech volumes of moduli spaces of Abelian and
quadratic differentials. More concretely, we use most of spectacular
advances in the study of Masur--Veech volumes obtained by A.~Aggarwal
in~\cite{Aggarwal:Volumes}, \cite{Agg:vol:quad}, and by D.~Chen,
M.~M\"oller, A.~Sauvaget, D.~Zagier in~\cite{CMSZ}, \cite{CMS:quad},
\cite{Sauvaget:minimal}, \cite{Sauvaget:asymptotic:expansion}. We
also use developments of these results by M.~Kazarian~\cite{Kazarian}
and by D.~Yang, D.~Zagier, Y.~Zhang~\cite{YZZ}; see also a related
work of J.~Guo and of D.~Yang~\cite{Guo:Yang}. Finally, we also
actively use our own recent results on Masur--Veech volumes of the
principal strata of quadratic differentials~\cite{DGZZ-volumes},
\cite{DGZZ-distrib} and the count of square-tiled
surfaces~\cite{DGZZ-Yoccoz}, developed, in particular, in view of
applications to meanders count. One of the main technical tools of
the current paper consists in the count of square-tiled surfaces
through Witten--Kontsevich correlators (intersection numbers of
$\psi$-classes) obtained in~\cite{DGZZ-volumes}. In the context of
meanders we need only $1$- and $2$-correlators. The $1$-correlators
admit a closed formula~\cite{Witten}, and the $2$-correlators admit a linear
recursion~\cite{Zograf:2-correlators} and precise
estimates~\cite{DGZZ-volumes}.

While for general genus $g$ meanders we obtain only a restricted
count corresponding to a fixed number $\nbigons$ of bigons, for
\textit{oriented} meanders we solve the enumeration problem
completely. We also obtain a precise large genus asymptotic count of
oriented meanders.

\begin{Remark}
\label{rm:large:n:large:g}
Count of square-tiled surfaces through Witten--Kontsevich
correlators, performed in our paper~\cite{DGZZ-volumes}, is closely
related to the count of asymptotic frequencies of simple closed
geodesics on a hyperbolic surface performed by
M.~Mirzakhani~\cite{Mirzakhani:growth:of:simple:geodesics}. As a
byproduct of the count of genus $g$ meanders, we compare in the
current paper the frequencies of separating versus non separating
simple closed geodesics on surfaces of a large genus $g$ with $n$
cusps. Numerous quantities responsible for geometry of a hyperbolic
surface of a large genus $g$ with $n$ cusps are very sensitive to the
growth rate of the number of cusps compared to the growth rate of the
genus, see results~\cite{Agg:vol:quad} of A.~Aggarwal on
Witten--Kontsevich correlators or results by
W.~Hide~\cite{Hide}, W.~Hide and M. Magee~\cite{Hide:Magee},
W.~Hide and J.~Thomas~\cite{Hide:Thomas}, of
N.~Anantharaman and
L.~Monk~\cite{Anantharaman:Monk}, T.~Budzinski, N.~Curien,
B.~Petri~\cite{Budzinski:Curien:Petri}, M.~Lipnowski and
A.~Wright~\cite{Lipnowski:Wright}, Yunhui~Wu and
Yuhao~Xue~\cite{Wu:Xue}, Yang~Shen and Yunhui~Wu~\cite{Shen:Wu}, and
of P.~Zograf~\cite{Zograf:spectral:gap} on the spectral gap and on
the Cheeger constant of hyperbolic surfaces of large genus with
cusps. We show in this paper that the frequencies of separating
versus non separating simple closed geodesics have very limited
dependence on number of cusps in large genus.
\end{Remark}

\subsection{Structure of the paper}

Section~\ref{sec:def} provides accurate definitions of higher genus
meanders and arc systems, and presents the main counting results. All
the results follow from the correspondence between meanders and
square-tiled surfaces. The latter represent integer points in moduli
spaces of quadratic or Abelian differentials.
Section~\ref{s:Square:tiled:surfaces:and:MV:volumes} recalls
necessary facts on the count of square-tiled surfaces;
Section~\ref{sec:flat} describes the above mentioned correspondence.

Having established this count we complete the proofs of those results
stated in Section~\ref{sec:def} which concern fixed genus $g$ and
fixed number of bigons $\nbigons$.

Section~\ref{sec:vol} is devoted to analysis of the resulting count
in two complementary asymptotic regimes: when the number $\nbigons$
of bigons is fixed and the genus $g$ grows and when the genus is
fixed and the number $\nbigons$ of bigons grows. Here we transpose
recent advances in asymptotics of the Masur--Veech volumes of moduli
spaces of Abelian and quadratic differentials mentioned above to
asymptotic count of meanders. As an application we compute the
asymptotic probability that a random simple closed geodesic on a
hyperbolic surface of genus $g$ with $n$ cusps is separating in the
regime when the number of cusps $\nbigons$ is fixed and the genus $g$
tends to infinity and in the regime when the genus $g$ is fixed while
the number of cusps $\nbigons$ tends to infinity stated in
Section~\ref{ssec: nonsep}.

We complete the paper with two Appendices. Consider the graph $\cG$
formed by a meander on a surface of genus $g$. For a fixed number
$\nbigons$ of bigons, the total number of boundary components of
$S-\cG$ formed by $6$ and more edges is bounded above by $4g-4+\nbigons$
independently of the number $2N$ of intersections of a meander. All
the remaining boundary components are bounded by exactly $4$ edges.
In appendix~\ref{app:combi} we perform a restricted count of the
asymptotic number of meanders with at most $2N$ intersections on a
surface of genus $g$ imposing to a meander not only a fixed number
$\nbigons$ of bigons in $S-\cG$, but also fixing the numbers of
$6$-gons, $8$-gons, etc.

Appendix~\ref{app:sum:of:binomials} might represent an independent
interest with no relation to meanders. Namely, we compute asymptotics
of the sum of the form
$
\sum_{k}
\cfrac{
\binom{a_1 n+b_1}{c_1 k+d_1}
\cdots
\binom{a_l n+b_l}{c_l k+d_l}
}{
\binom{s_1 n+t_1}{u_1 k+v_1}
\cdots
\binom{s_m n+t_m}{u_1 k+v_m}
}
$
under assumptions that $\frac{a_1}{c_1}=\dots=\frac{a_l}{c_1}=
\frac{s_1}{u_1}=\dots=\frac{s_m}{u_m}$ and that
$a_1+\dots+a_l>s_1+\dots+s_m$. Particular cases of this computation
are used in Section~\ref{sec:vol}. Another particular case allows us
to derive large genus asymptotics of the sums
$\sum_{k=0}^{3g-1}\langle\tau_k\tau_{3g-1-k}\rangle_{g}$ of
Witten--Kontsevich 2-correlators used in Section~\ref{sec:vol}.
\medskip

\noindent
\textbf{Acknowledgments.}
We are grateful to A.~Aggarwal, C.~McMullen, F.~Petrov, A.~Zvonkin
for helpful questions and comments. We thank L.~Monk and to B.~Petri
for clarifying to us dependence of properties of hyperbolic surfaces
in different regimes when either genus or the number of cusps are
growing, see Remark~\ref{rm:large:n:large:g}. We thank I.~Ren for
indicating us several typos in the manuscript.


\section{Transverse multicurves, meanders and arc systems}
\label{sec:def}

\subsection{Count of higher genus meanders}
In this paper we consider only simple curves on oriented
smooth surfaces, where \textit{simple} means that the curve is
smoothly embedded into the surface, i.e. does not have
self-intersections, self-tangencies or cusps. In this paper we do not
exclude closed curves homotopic to a point. We reserve the notions
\textit{arc} and \textit{strand} for topological segments. In all our
considerations closed curves live on closed surfaces, while
non-closed curves (i.e. \textit{arcs} and \textit{strands}) live on
surfaces with boundaries, have endpoints at boundary components and
are transverse to the boundaries. By convention, endpoints of an
\textit{arc} might belong to the same or to different boundary
components, while the endpoints of a \textit{strand} necessarily
belong to distinct boundary components.  A \textit{multicurve} is a
finite collection of pairwise disjoint simple closed curves.
In particular, in this paper we do \textit{not} use integral weights
to encode freely homotopic connected components of a multicurve.

\begin{Definition}
\label{def:pair:of:multicurves}
We say that an ordered pair of multicurves forms a
\textit{connected transverse pair} if both of the
following conditions are satisfied: all pairs of components of
multicurves intersect transversally and the graph $\cG$ obtained
as a union of two multicurves is connected. The pair is
\textit{filling} if it cuts the surface into topological disks, or,
equivalently, if the graph $\cG$ is a \textit{map}.
\end{Definition}

By convention, the first multicurve in a connected
transverse pair of multicurves is called \textit{horizontal} and
the second one --- \textit{vertical}. We consider natural
equivalence classes of pairs of transverse multicurves up to
diffeomorphisms preserving orientation of the surface and
the  horizontal and vertical multicurves.

\begin{Remark}
\label{rm:multicurves}
Note that speaking of a ``multicurve'' one usually assumes that the
components of a multicurve are neither contractible nor peripheral
(i.e. not freely homotopic to a boundary component). Given a
connected filling transverse pair of multicurves in the sense of
Definition~\ref{def:pair:of:multicurves} make a single puncture at
every bigon. We will see in Section~\ref{sec:flat} that each
component of the horizontal (respectively vertical) multicurve on the
resulting punctured surface is neither contractible nor peripheral,
so we get a multicurve in the usual sense.
\end{Remark}

\begin{figure}[htb]
\includegraphics{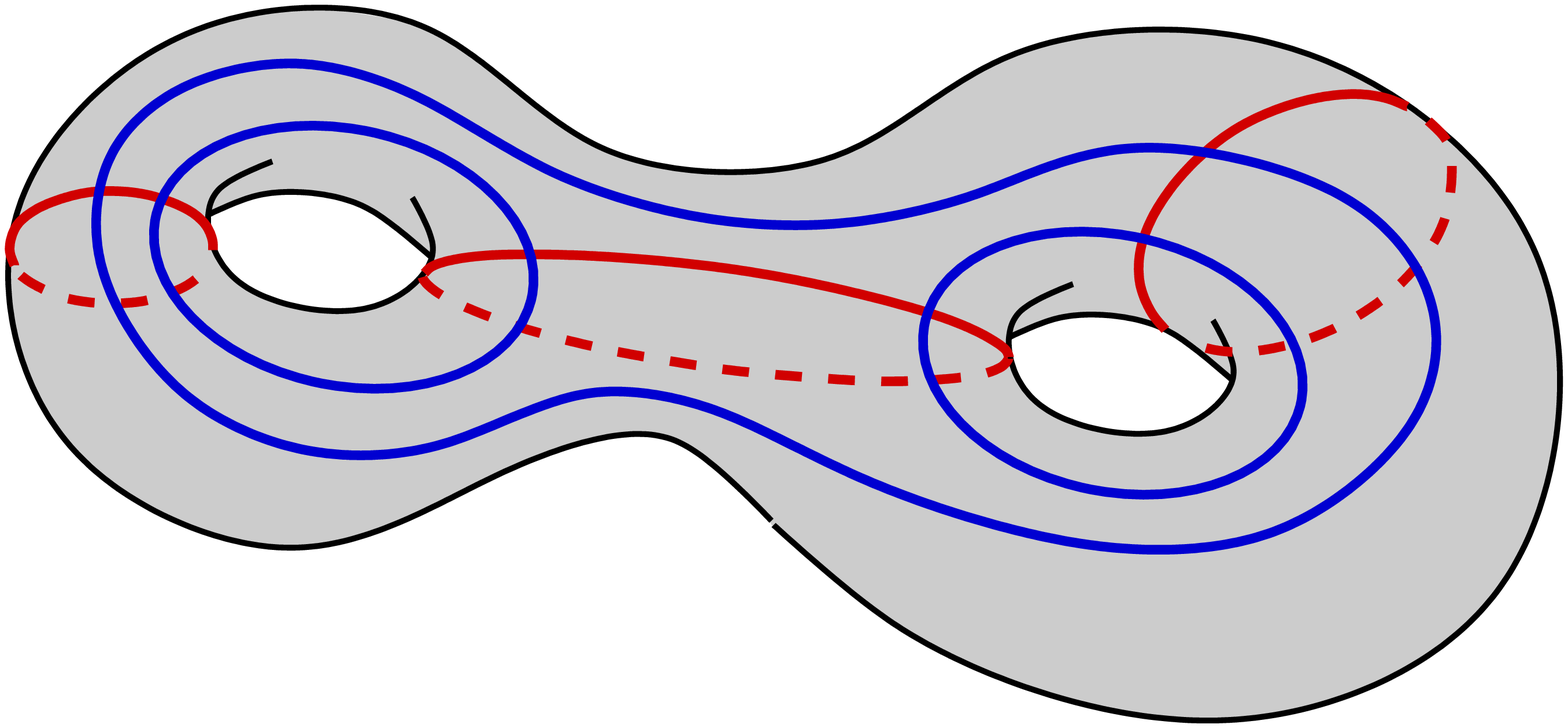}
\includegraphics{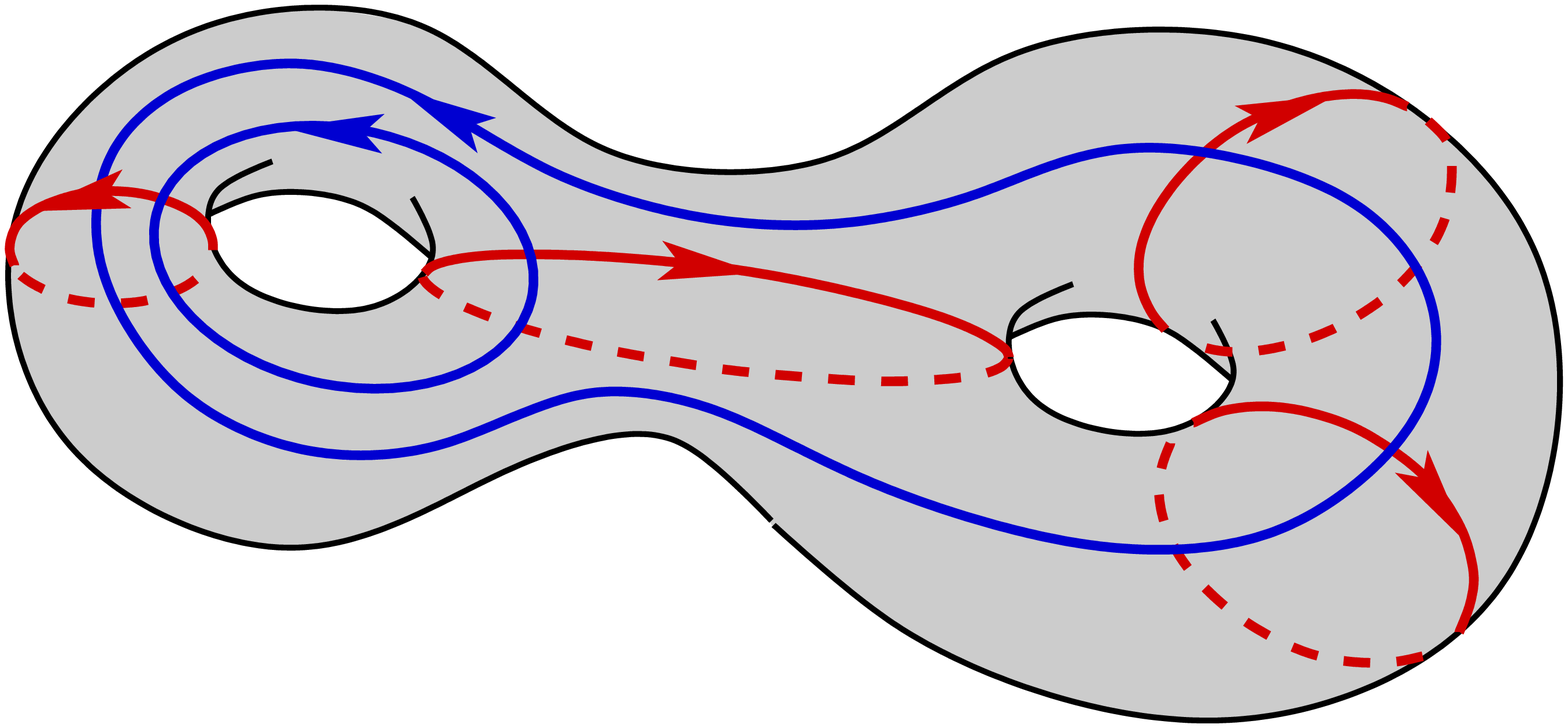}
\vspace{60bp}
\caption{
\label{fig:transverse:multicurves}
Both connected pairs of transverse multicurves are
filling, but the pair on the left is nonorientable, while the pair on
the right is oriented.}
\end{figure}

\begin{Definition}
A transverse pair of multicurves is called \textit{positively
oriented} (or just \textit{oriented} for brevity) if each connected
component of each multicurve is oriented in such way that any
individual intersection of any connected component of the horizontal
multicurve with any connected component of the vertical multicurve
matches the orientation of the surface. In other words, the
intersection number of any connected component of the horizontal
multicurve with any connected component of the vertical multicurve
coincides with the naive number of intersections. A transverse pair
of multicurves is called \textit{orientable} if it admits the above
structure and \textit{nonorientable} otherwise, see
Figure~\ref{fig:transverse:multicurves} for an illustration.
\end{Definition}

\begin{Definition}
A \textit{meander of genus $g$} is a connected pair of transverse
simple closed curves on a surface of genus $g$. Similarly, an
\textit{orientable meander of genus $g$} is an orientable connected
pair of transverse simple closed curves on a surface of genus $g$
(see Figure~\ref{fig:transverse:multicurves} for an illustration).
Fixing an orientation of an orientable meander, we get an
\textit{orientable} meander. Meanders of genus $g>0$ are called
\textit{higher genus meanders}.
\end{Definition}

An orientable meander admits two distinct orientations unless there
exists a diffeomorphism of the surface sending the meander to itself
and reversing the orientation of each of the two simple closed curves.
There are no orientable meanders in genus 0.


Following the notation for count of classical meanders on a sphere
(which fits traditional conventions on count of square-tiled surfaces
in strata of meromorphic quadratic differentials) we denote by
$\MeandNumber_{g,\nbigons}(N)$ the number of genus $g$ meanders with
at most $2N$ crossings and exactly $\nbigons$ bigons, see
Definition~\eqref{def:bigons}. It would be convenient to denote by
$\MeandNumber_{g}^+(N)$ the number of orientable genus $g$ meanders
with at most $N$ (and not $2N$) crossings, see
Figure~\eqref{fig:meanders:g:123} for an illustration.

\begin{figure}[htb]
\includegraphics{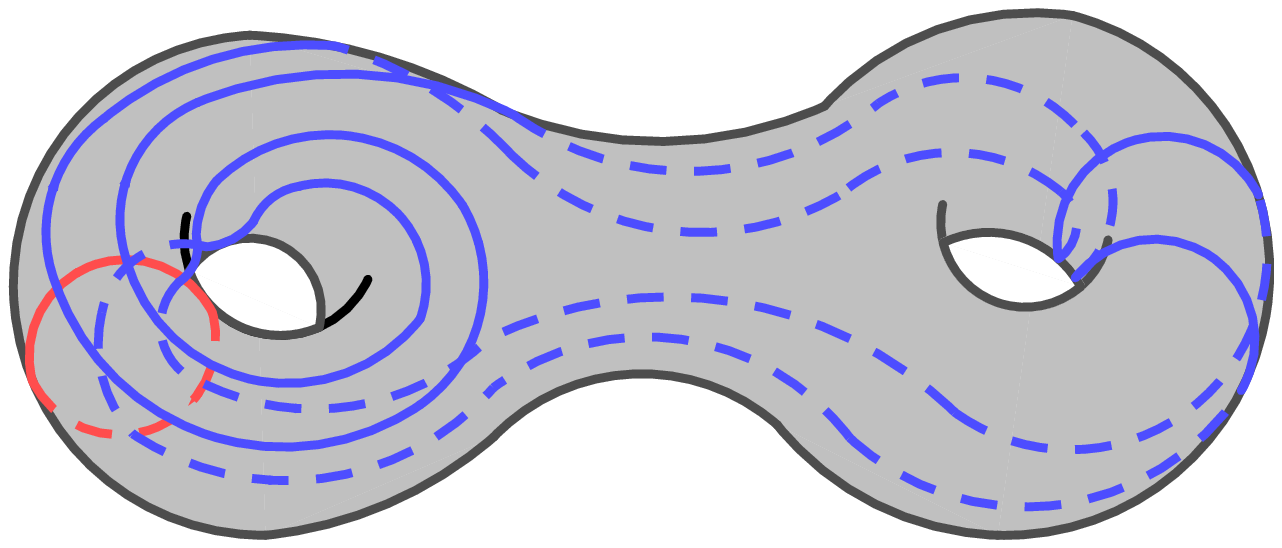}
\includegraphics{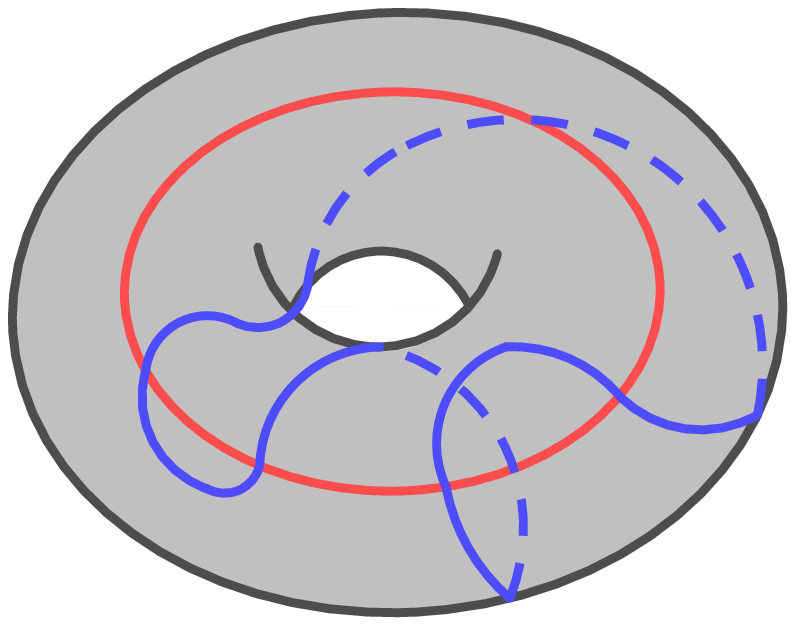}
\includegraphics{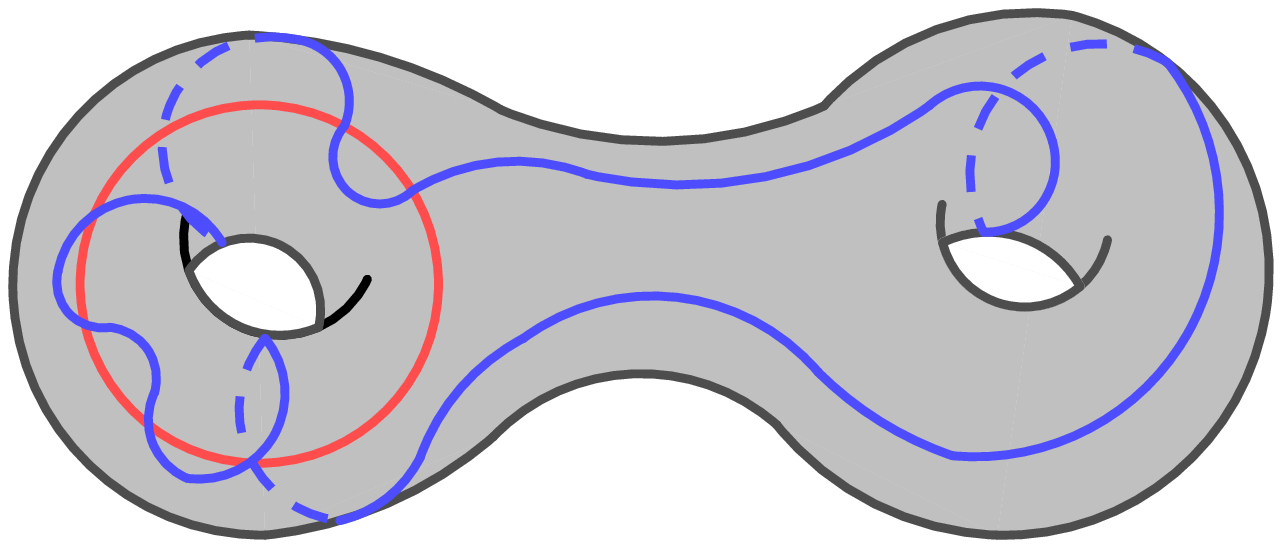}
   %
   %
\vspace{70bp} 
\caption{
\label{fig:meanders:g:123}
Meander on the left contributes to $\MeandNumber^+_2(4)$;
in the middle --- to $\MeandNumber_{1,2}(2)$; on the right ---
to $\MeandNumber_{2,4}(3)$.
}
\end{figure}

\begin{Theorem}
\label{th:meander}
For any nonnegative numbers $g$ and $n$ satisfying $2g+\nbigons\ge
4$, the number $\MeandNumber_{g,\nbigons}(N)$ of meanders with
exactly $\nbigons$ bigons on a surface of genus $g$ and with at most
$2N$ crossings satisfies the following asymptotics:
\begin{equation}
\label{eq:asymptotics:Mgp}
\MeandNumber_{g,\nbigons}(N)
=C_{g,\nbigons} N^{6g-6+2\nbigons}
+ o(N^{6g-6+2\nbigons})\quad\text{as }N\to \infty\,,
\end{equation}
with
\begin{equation}
\label{eq:constant:Cgp}
C_{g,\nbigons}=\frac{\cyl_{1,1}(\cQ_{g,\nbigons})}
{(4g-4+\nbigons)!\,\nbigons!\,(12g-12+4\nbigons)
}\,.
\end{equation}
Here
\begin{equation}
\label{eq:cyl:11:definition}
\cyl_{1,1}(\cQ_{g,\nbigons})
=\frac{\big(\cyl_1(\cQ_{g,\nbigons})\big)^2}{\Vol\cQ_{g,\nbigons}}
\end{equation}
is a rational multiple of $\pi^{-6g+6-2n}$, where
$\Vol\cQ_{g,\nbigons}$ denotes the Masur--Veech volume of the moduli
space of quadratic differentials, and $\cyl_1(\cQ_{g,\nbigons})$
denotes the contribution of single-band square-tiled surfaces to this
volume. The quantities $\Vol\cQ_{g,\nbigons}$ and
$\cyl_1(\cQ_{g,\nbigons})$ are expressed in terms of intersection
numbers of $\psi$-classes by formulas~\eqref{eq:Vol:Qgn} and
\eqref{cyl:1:Gamma:1:new}--\eqref{cyl1Q:simplified} respectively.

All but negligible (as $N\to+\infty$) part of meanders as above are
nonorientable, filling, and have only bigonal, quadrangular and
hexagonal faces.

For any fixed value of $g$ we have
\begin{equation}
\label{eq:Cgp:p:to:infty}
C_{g,\nbigons}\sim
\frac{1}{8\pi^2}\cdot
\frac{a_g^2}{\kappa_g}
\cdot\frac{1}{\nbigons^{\frac{5}{2}g-1}}
\cdot\left(\frac{32\cdot e^2}{\pi^2\cdot\nbigons^2}\right)^\nbigons
\ \text{as }\nbigons\to\infty\,;
\end{equation}
where $a_g$ and $\kappa_g$ are given by Equations~\eqref{eq:a:g}
and~\eqref{eq:kappa:g} respectively.

For any fixed value of $\nbigons$ we have:
\begin{equation}
\label{eq:Cgp:g:to:infty}
C_{g,\nbigons}\sim
\frac{1}{32}
\sqrt{\frac{3}{2\pi}}
\cdot\frac{1}{\nbigons!}
\cdot \left(\frac{4}{3g}\right)^{\nbigons-\frac{3}{2}}
\cdot\left(\frac{2e}{3g}\right)^{4g}
\ \text{as }g\to\infty\,.
\end{equation}
\end{Theorem}

The polynomial asymptotics~\eqref{eq:asymptotics:Mgp} and the fact
that the coefficient of the leading term is given
by expression~\eqref{eq:constant:Cgp} is proved in
Section~\ref{s:proofs:fixed:g:and:n}. Asymptotic
relation~\eqref{eq:Cgp:p:to:infty} is proved at the end of
Section~\ref{ss:large:number:of:poles}. Asymptotic
relation~\eqref{eq:Cgp:g:to:infty} is proved at the end of
Section~\ref{ss:large:genus}.

This result can be compared to the following
natural exponential bounds for the number of meanders on surfaces of
genus $g$, with no constraints on the number of bigons:

\begin{Lemma}
\label{lem:estim:meander}
The number $M_g^{=N}$ of meanders with \textit{exactly} $2N$
crossings on a surface of genus $g$ satisfies the following
bounds:
\[
C_N\leq M_0^{=N}\leq M_g^{=N}\leq
2N\cdot p_g(2N+1)\cdot C_{2N+1}
\left(1+o(1)\right)\ \text{ as }N\to+\infty\,,
\]
where $C_N=\frac{1}{N+1}\binom{2N}{N}$ is the $N$-th Catalan number,
and $p_g$ is an explicit polynomial of degree $3g$.
\end{Lemma}

Lemma~\ref{lem:estim:meander} is proven in
Section~\ref{s:proofs:fixed:g:and:n}.

\begin{Theorem}
\label{th:oriented_meander}
For any genus $g\geq 1$, the number $\MeandNumber_{g}^+(N)$ of
oriented meanders with at most $N$ crossings
satisfies the following asymptotics:
\begin{equation}
\label{eq:M:plus}
\MeandNumber_{g}^+(N)=C_{g}^+ N^{4g-3} + o(N^{4g-3})
\quad\text{as }N\to \infty
\,,
\end{equation}
where
\begin{equation}
\label{eq:Cg:plus}
C_{g}^+=\frac{cyl_{1,1}(\cH_{g})}{(2g-2)!(8g-6)}\,.
\end{equation}
Here
$cyl_{1,1}(\cH_g)=\frac{\left(\cyl_1(\cH_g)\right)^2}{\Vol\cH_g}$ is
a rational multiple of $\frac{1}{\pi^{2g}}$, where $\Vol\cH_g$ denotes
the Masur--Veech volume of the moduli space of Abelian differentials
$\cH_g$ and
$$
\cyl_1(\cH_g)=\frac{1}{2^{2g-4}(4g-2)}
$$
is the contribution of single-band square-tiled surfaces to this volume.

Moreover, $C_{g}^+$ satisfies the following asymptotics
\begin{equation}
\label{eq:Cg:plus:asymptotics}
\frac{1}{4\sqrt{\pi}}
\cdot\frac{1}{g^{\frac{3}{2}}}
\left(\frac{e}{4g}\right)^{2g}
\left(1+\frac{29+\pi^2}{24 g}
+ O\left(\frac{1}{g^2}\right)\right)
\quad\text{as }g\to\infty\,.
\end{equation}
\end{Theorem}

The polynomial asymptotics~\eqref{eq:M:plus} and the fact that the
coefficient of the leading term is given by
expression~\eqref{eq:Cg:plus} is proved in
Section~\ref{s:proofs:fixed:g:and:n}. Asymptotic
relation~\eqref{eq:Cg:plus:asymptotics} is
proved in Corollary~\ref{cor:vol:ab} in
Section~\ref{ss:large:genus:Abelian}.

Using the correspondence between meanders and flat surfaces we perform a
more detailed count for higher genus meanders fixing not only the
number $\nbigons$ of bigons but also the number of $6$-gons, $8$-gons
etc. This detailed count is presented in Appendix~\ref{app:combi}.
However, in the most general non-orientable case, at the current
stage of knowledge of Masur--Veech volumes of general strata of
quadratic differentials, one can transform our formulas into actual
rational numbers only for small values of $g$ and $\nbigons$.
\medskip

\noindent
\textbf{Meandric systems.}
Traditionally, one represents a multicurve as a weighted
sum $\gamma=h_1\gamma_1+\dots+h_m\gamma_m$ of the primitive
components $\gamma_1,\dots,\gamma_m$, which are already not
pairwise freely homotopic on the punctured surface (see Remark~\ref{rm:multicurves}), and
where the positive integer weight $h_i$ encodes the number
of components of the multicurve $\gamma$ freely homotopic
to the primitive component $\gamma_i$ for $i=1,\dots,m$.

The \textit{number of components} of a multicurve
$\gamma=h_1\gamma_1+\dots+h_m\gamma_m$ is given by the sum
$h=h_1+\dots+h_m$ of the weights, where $m$ is
the number of \textit{primitive} components.

One can define \textit{meandric systems} by allowing the vertical
multicurve to have several components. Meandric systems were recently
studied in~\cite{Curien:Kozma:Sidoraviciu:Tournier},
\cite{Feray:Thevenin} \cite{Fukuda:Nechita},
\cite{Goulden:Nica:Puder}, \cite{Kargin} (see these papers for
further references). The results of~\cite{DGZZ-distrib} provide the
large genus asymptotic distribution of the number $m$ of primitive
components of those meandric systems, which do not have any bigons.
The genus zero case with a fixed number of bigons was studied in
slightly different terms in~\cite{AEZ:Dedicata}.

The question of the distribution of the actual number $h$ of
components in large genus is still open, some conjectures about this
distribution will be presented in a subsequent paper.

\subsection{Asymptotic probability of getting a meander from an arc system}

Consider a transverse pair of multicurves such that the horizontal
multicurve is just a single simple closed curve. Cutting the
surface by this horizontal curve we get an  \textit{arc system} as in
Figure~\ref{fig:meanders:in:genus:1}. The cut surface has one or two
connected components depending on whether the simple closed curve is
separating or not.

\begin{Definition}\label{def:arc_system}
A \textit{balanced arc system of genus $g$} is a
finite collection of smooth pairwise nonintersecting segments (called
\textit{arcs}) on a smooth oriented surface with two boundary
components (a single connected surface of genus $g-1$ with two
boundary components or two connected surfaces of genera $g_1, g_2$
satisfying $g_1+g_2=g$, each with a single boundary component),
satisfying all of the following conditions:
\begin{itemize}
\item the endpoints of all segments are located at the boundary;
\item  each segment approaches the boundary transversally;
\item  the numbers of endpoints of the segments on one boundary
components is the same as on the other boundary component (and,
hence, equals the number of arcs).
\end{itemize}
The arc system is \textit{filling} if the segments cut the surface
into a collection of topological disks.
\end{Definition}

Throughout this paper we consider only balanced arc
systems, even when it is not stated explicitly.

\begin{figure}[htb]
\includegraphics{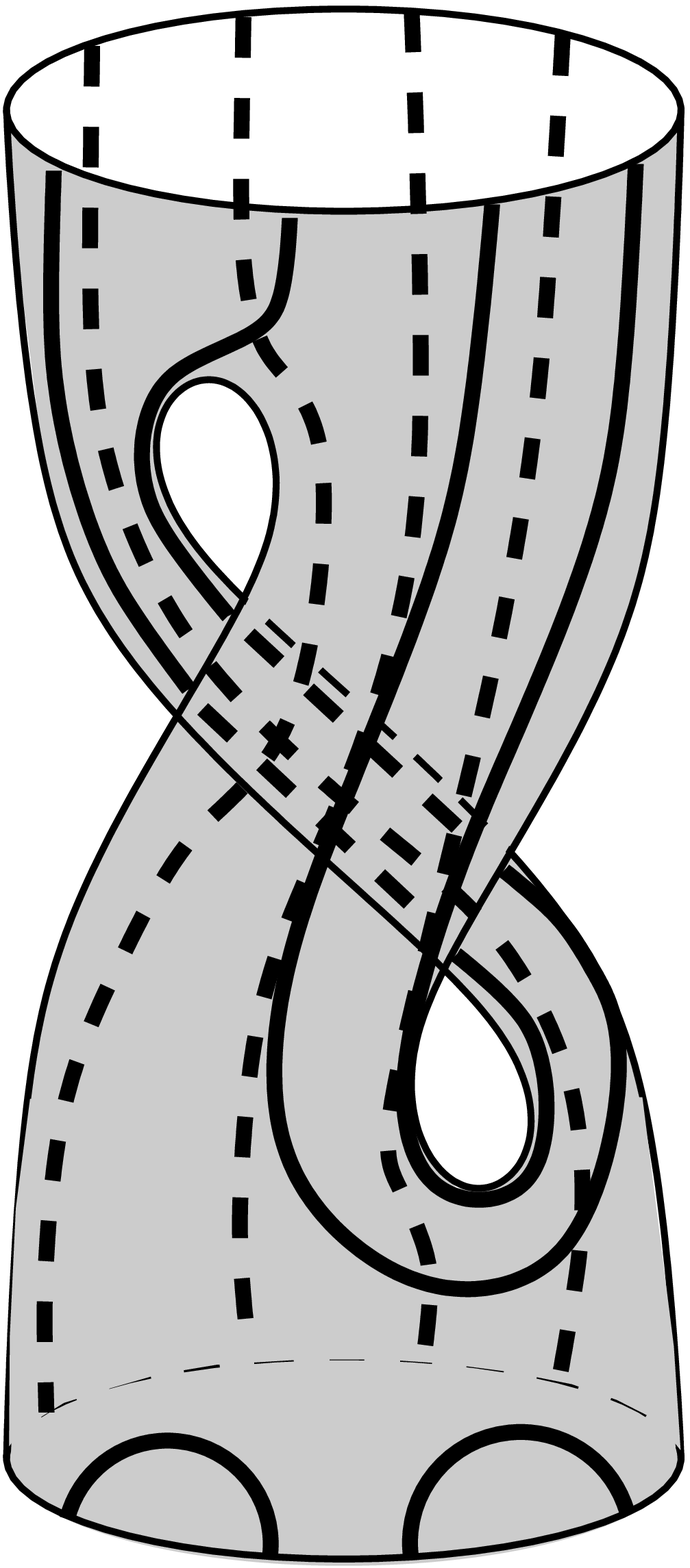}
\includegraphics{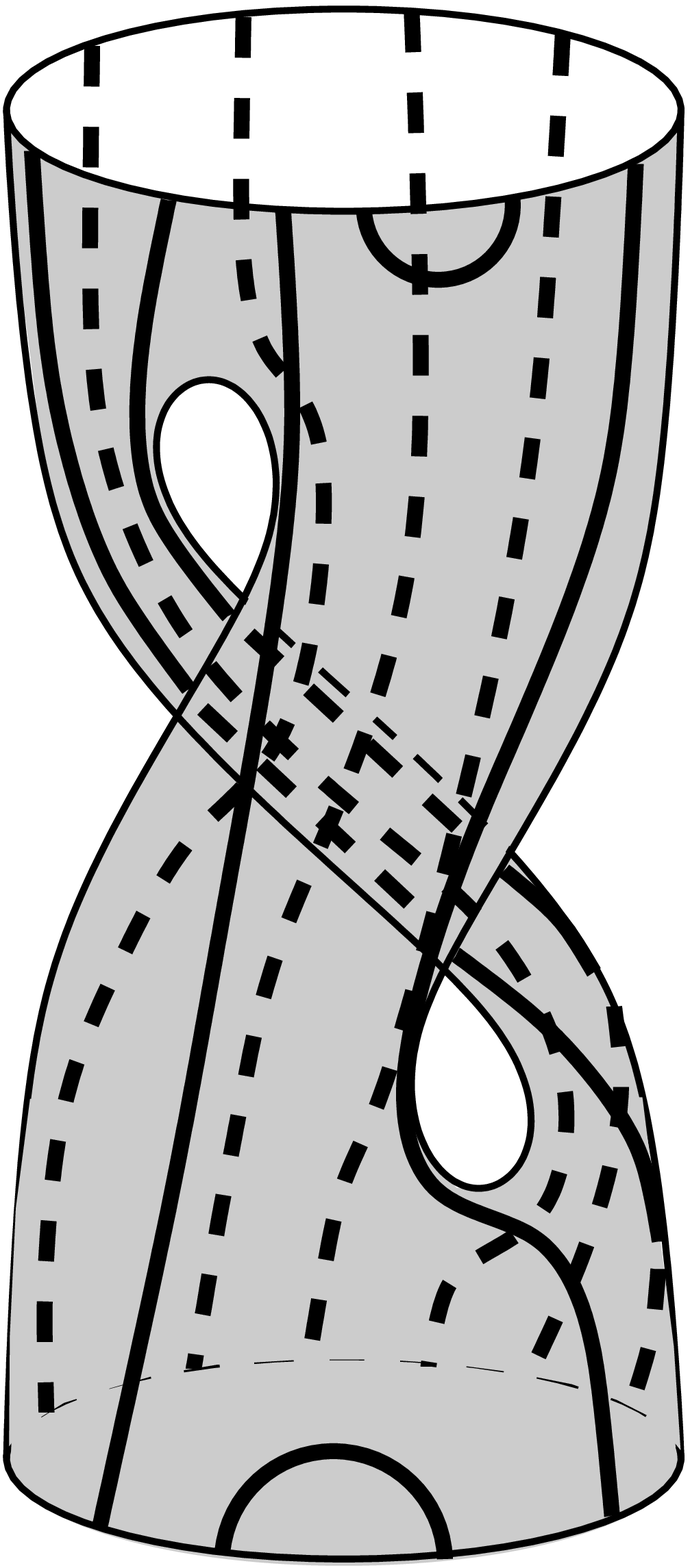}
\includegraphics{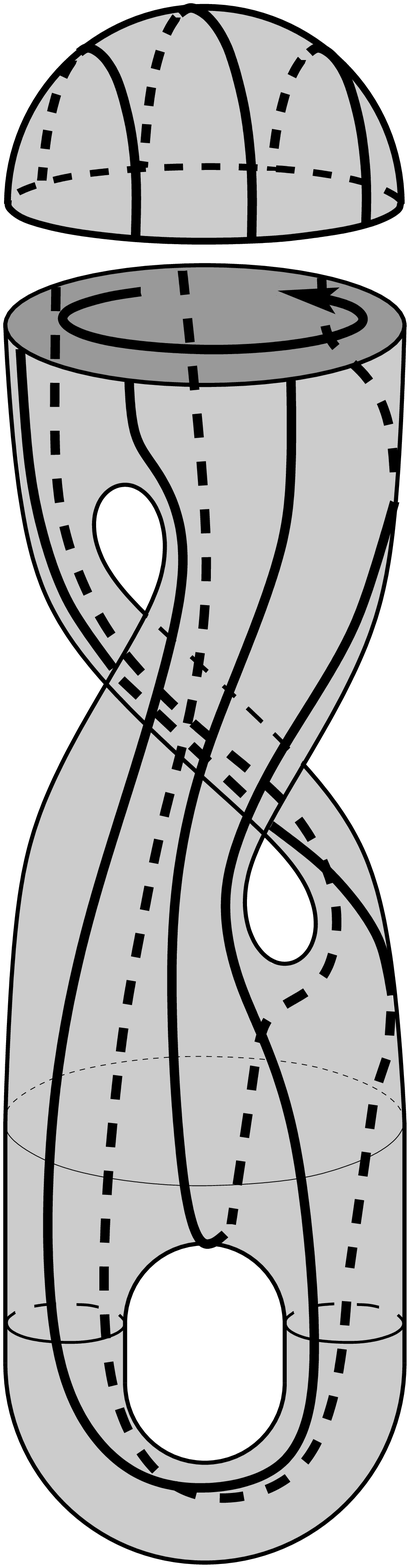}
\includegraphics{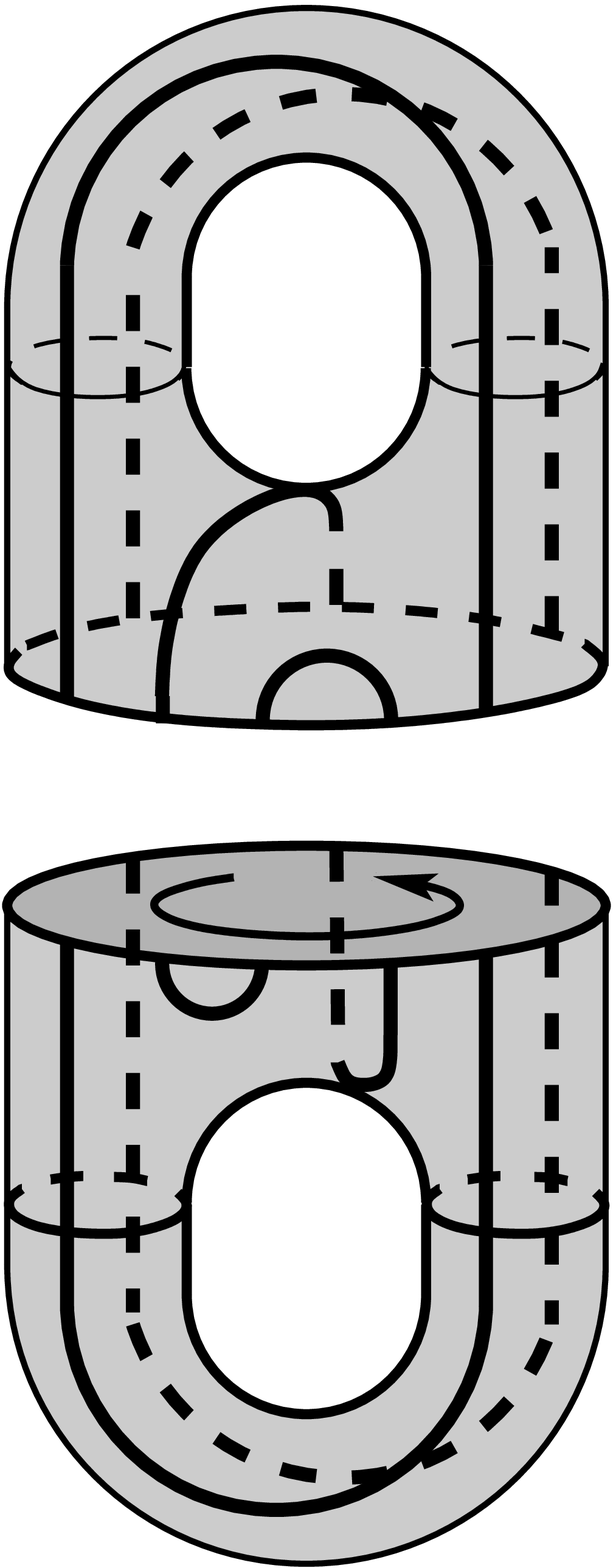}
\vspace{120bp} 
\caption{
\label{fig:meanders:in:genus:1}
A random identification of a random balanced arc system with large
number of arcs gives a meander with one and the same asymptotic
probability exceeding $1/4$ for each of the four types of arc
systems as in the picture, see Example~\ref{ex:as:probability}.
}
\end{figure}

Identifying the boundary components of a surface endowed with a
balanced arc system by a diffeomorphism which matches the endpoints
of the arcs and arranges them into smooth curves we get a transverse
pair of multicurves where the horizontal one is connected. We use
only those identifications of the boundary which lead to an oriented
surface. The following question seems to us natural to address: what
is the probability to obtain a meander by this construction? We have
an answer to this question, \textit{on average} in the following sense.

Up to a Dehn twist along the boundary component, there are (at most)
$\narcs$ distinct identifications of the two boundary components of a
surface endowed with a balanced arc system
containing $\narcs$ arcs, matching the endpoints of the arcs. (The
number of distinct identifications is less than $\narcs$ when the arc
system admits symmetries.)

Fix the genus $g$ and the number $\nbigons$ of bigons. We always
assume that $2g+\nbigons\ge 4$. Fix the upper bound $N$ for the
number of arcs. Denote by $\operatorname{AS}_{g,n}(N)$ the number of
all possible couples (balanced arc system of genus $g$ with $n$
bigons with $\narcs\le N$ arcs; identification) considered up to a
natural equivalence. Denote by $\operatorname{MAS}_{g,n}(N)$ the
number of those couples, which give rise to a meander. Define
\[
\prob_{g,\nbigons}(N)
=\frac{\operatorname{MAS}_{g,n}(N)}{\operatorname{AS}_{g,n}(N)}\,.
\]
Recall that by convention $g$ denotes the genus of the surface
obtained \textit{after} identification of the two boundary
components.

\begin{Theorem}
\label{th:arc_system}
The proportion of balanced arc systems of genus $g$
with $\nbigons$ bigons giving rise to meanders among all such arc
systems has a limiting value
\begin{equation}
\label{eq:Pgn:equals:p1gn}
\lim_{N\to\infty} \prob_{g,\nbigons}(N)=\probgn\,,
\end{equation}
which coincides with the relative contribution
$p_1(\cQ_{g,\nbigons})$ of single-band square-tiled surfaces to
the Masur--Veech volume $\Vol\cQ_{g,\nbigons}$. The quantity
\begin{equation}
\label{eq:p1:definition}
\probgn=p_1(\cQ_{g,\nbigons})
=\frac{cyl_1(\cQ_{g,\nbigons})}{\Vol\cQ_{g,\nbigons}}
\end{equation}
is a rational multiple of $\pi^{-6g+6-2n}$, where
$\Vol\cQ_{g,\nbigons}$ denotes the Masur--Veech volume of the moduli
space of quadratic differentials, and $\cyl_1(\cQ_{g,\nbigons})$
denotes the contribution of single-band square-tiled surfaces to this
volume. The quantities $\Vol\cQ_{g,\nbigons}$ and
$\cyl_1(\cQ_{g,\nbigons})$ are expressed in terms of intersection
numbers of $\psi$-classes by Formulas~\eqref{eq:Vol:Qgn} and
\eqref{cyl:1:Gamma:1:new}--\eqref{cyl1Q:simplified} respectively.

Moreover, for any fixed value of $g$ we have
\begin{equation}
\label{eq:p1:g:to:infty:promis}
\relcontgn \sim
\frac{1}{\sqrt{\pi}}\cdot
\frac{a_g}{\kappa_g}\cdot
n^{\frac{g-1}{2}}\left(\frac{8}{\pi^2}\right)^n
\quad\text{as } n\to \infty\,,
\end{equation}
where $a_g$ and $\kappa_g$ are given by Equations~\eqref{eq:a:g}
and~\eqref{eq:kappa:g} respectively.

For any fixed $\nbigons$:
\begin{equation}
\label{eq:p1:n:to:infty:promis}
\relcontgn\sim \frac{\sqrt{6\pi}}{12}\cdot \frac{1}{\sqrt g}\,,
\quad\text{as } g\to\infty\,.
\end{equation}
\end{Theorem}

The existence of the limit $\lim_{N\to\infty} \probgn(N)$
and expression~\eqref{eq:p1:definition} for its value are proved at
the end of Section~\ref{s:proofs:fixed:g:and:n}. Asymptotic
relations~\eqref{eq:p1:g:to:infty:promis}
and~\eqref{eq:p1:n:to:infty:promis} are proved in
Corollaries~\ref{cor:cyl11_quad_poles}
and~\ref{cor:cyl11_quad_genus} respectively.

Formulas~\eqref{eq:Vol:Qgn} and
\eqref{cyl:1:Gamma:1:new}--\eqref{cyl1Q:simplified} for
$\Vol\cQ_{g,\nbigons}$ and $\cyl_1(\cQ_{g,\nbigons})$ lead to closed
expressions for any fixed genus $g$ as a function of $n$. For
example,
\begin{align*}
p_1(\cQ_{0,\nbigons})
& =\frac{1}{2}\left(\frac{2}{\pi^2}\right)^{\nbigons-3}
\binom{2\nbigons-4}{\nbigons-2}\,,
\\
p_1(\cQ_{1,\nbigons})
& = \cfrac{2}{\pi^{2\nbigons}}
\cdot\cfrac{4\cdot\binom{2\nbigons-2}{\nbigons-1}
   +\frac{1}{48}\cdot\binom{2\nbigons-1}{\nbigons-2}
   -\frac{\nbigons^2-\nbigons+2}{96}}
{\frac{\nbigons!}{(2\nbigons-1)!!}
   +\frac{2\nbigons}{(2\nbigons-1)2^\nbigons}}\,.\\
\end{align*}

We can choose a setting, in which we consider subsets of arc systems
as above sharing a more restricted topology. Namely, when $g$ is
strictly positive, we can separately consider arc systems on a
surface of genus $g-1$ with two boundary components, as on the left
two pictures in Figure~\ref{fig:meanders:in:genus:1}. We can also fix
a partition of $\nbigons$ minimal arcs into $n_1\ge 0$ arcs adjacent
to the first boundary component and $n_2\ge 0$ arcs adjacent to the
second boundary component, where $n_1+n_2=n$. If $g=1$, we assume
that $\nbigons\ge 2$ and that $n_i\ge 1$ for $i=1,2$. Fixing
$g,n_1,n_2$ as above and the bound $N$ for the number of arcs, denote
by $\operatorname{AS}_{g,n_1,n_2}(N)$ the number of all possible
couples (balanced arc system with $n_i$ bigons at the $i$-th boundary
component, where $i=1,2$, with at most $\narcs\le N$ arcs on a
connected surface of genus $g-1$ with two boundary components;
identification) considered up to a natural equivalence. Denote by
$\operatorname{MAS}_{g,n_1,n_2}(N)$ the number of those couples,
which give rise to a meander. Define
$$
\prob_{g,n_1,n_2}(N)
=\frac{\operatorname{MAS}_{g,n_1,n_2}(N)}{\operatorname{AS}_{g,n_1,n_2}(N)}\,.
$$

Alternatively, we can chose any nonnegative integers $g_1,g_2$ such
that $g_1+g_2=g$ and consider two connected surfaces of genera $g_1$
and $g_2$ respectively, each having a single boundary component, as
on the right two pictures in Figure~\ref{fig:meanders:in:genus:1}. We
can also consider any partition $n_1+n_2=n$ of $n$ into an ordered
sum of nonnegative integers satisfying the following condition: if
$g_i=0$, for any of $i=1,2$, then $n_i\ge 2$. Denote by
$\operatorname{AS}_{g,n_1}^{g_2,n_2}(N)$ the number of all possible
couples (balanced arc system on a surface having two connected
components of genera $g_1$ and $g_2$, each with a single boundary
component, with $n_1$ bigons on the first component, with $n_2$
bigons on the second component and with $\narcs\le N$ arcs;
identification). considered up to a natural equivalence. Denote by
$\operatorname{MAS}_{g,n_1}^{g_2,n_2}(N)$ the number of those
couples, which give rise to a meander. Define
$$
\prob^{g_2,n_2}_{g_1,n_1}(N)
=\frac{\operatorname{MAS}_{g,n_1}^{g_2,n_2}(N)}{\operatorname{AS}_{g,n_1}^{g_2,n_2}(N)}\,.
$$

\begin{Theorem}
\label{th:all:the:same}
For any $g,n,g_1,g_2,n_1,n_2$ satisfying the above requirements one
has
\begin{equation}
\label{eq:Pgn:same:for:all}
\lim_{N\to\infty} \prob_{g,n_1,n_2}(N)
= \lim_{N\to\infty} \prob^{g_2,n_2}_{g_1,n_1}(N)
=\probgn\,,
\end{equation}
\end{Theorem}

Theorem~\ref{th:all:the:same} is proved at the end of
Section~\ref{s:proofs:fixed:g:and:n}.

\begin{Example}
\label{ex:as:probability}
Consider a random balanced arc system of any of the four types
schematically presented in Figure~\ref{fig:meanders:in:genus:1}. One
should, actually, take much more arcs than in the picture,
maintaining, however, a location of the only two minimal arcs
in each of the four cases.
Theorem~\ref{th:all:the:same} affirms, in particular, that a random
identification of boundary components of such an arc system we obtain
a meander with asymptotic probability
   %
$$
\lim_{N\to\infty} \prob_{2,2,0}(N)
=\lim_{N\to\infty} \prob_{2,1,1}(N)
= \lim_{N\to\infty} \prob^{1,1}_{1,1}(N)
= \lim_{N\to\infty} \prob^{2,0}_{0,2}(N)
\prob_{2,2}=\frac{9\,230\,760}{337\cdot\pi^{10}}
\approx 0.292489
$$
   %
for each of the four types of arc systems as in
Figure~\ref{fig:meanders:in:genus:1}. The numerical value of
$\prob_{2,2}$ given by Formula~\eqref{eq:p1:definition}, uses the
value $\Vol\cQ_{2,2}=\frac{337}{18144}\pi^{10}$ evaluated by means of
Formula~\eqref{eq:Vol:Qgn} and the value
$\cyl_1(\cQ_{2,2})=\frac{2035}{4}$ evaluated by means of
Formula~\eqref{cyl1Q:simplified}.
\end{Example}

\noindent\textbf{Oriented arc systems (collections of strands).}
Consider a closed oriented surface endowed with a connected transverse
pair of multicurves, such that the horizontal multicurve is a
single simple closed curve. Cutting the surface by the horizontal
curve we get an \textit{oriented arc system}, or equivalently, a
collection of disjoint strands, each strands joining one boundary
component to another, as in Figure~\ref{fig:oriented:arc:system}.
Reciprocally, having a connected oriented surface with two boundary
components, and a system of disjoint strands joining the boundary
components and approaching them transversally, we can identify
boundary components in a way which matches the endpoints of the
strands, and get a connected transverse pair of multicurves, where
the horizontal multicurve is a single simple closed curve. By
construction, the resulting transverse pair of multicurves is always
orientable. Choosing the orientation of the horizontal curve or of
any strand, we uniquely determine the orientation of the resulting
pair of multicurves. As before, when there are $\narcs$ strands,
there are (at most) $\narcs$ distinct identifications of the two
boundary components, matching the endpoints of the arcs, up to a Dehn
twist along the boundary component. The number of distinct
identification is less than $\narcs$ when the arc system admits
symmetries.

Fix the genus $g-1$ of the connected oriented surface with two
boundary components and the upper bound $N$ for the number of
strands. Denote by $\operatorname{AS}_g^+(N)$ the number of all
possible couples (oriented arc system with at most $N$ strands on a
surface of genus $g-1$; identification) and by
$\operatorname{MAS}_g^+(N)$ the number of couples as above which give
ride to a meander. Define
\[
\probg(N)
=\frac{\operatorname{MAS}_g^+(N)}{\operatorname{AS}_g^+(N)}\,.
\]

\begin{Theorem}
\label{th:oriented_arc_system}
The proportion of oriented arc systems of genus $g$ giving rise to
oriented meanders among all such arc systems satisfies
\begin{equation}
\label{eq:oriented:limit}
\lim_{N\to\infty} \probg(N)=\probg\,.
\end{equation}
Here
\begin{equation}
\label{eq:p1:Hg:def}
\probg=p_1(\cH_g)=\frac{cyl_1(\cH_g)}{\Vol\cH_g}
\end{equation}
is a rational multiple of $\pi^{-2g}$, where $\Vol\cH_g$ denotes the
Masur--Veech volume of the moduli space of Abelian differentials, and
$\cyl_1(\cH_g)=\frac{1}{(2g-1)\cdot 2^{2g-3}}$ denotes the
contribution of single-band square-tiled surfaces to this volume.

Moreover,
\begin{equation}
\label{eq:p1plus:large:g}
p_1(\cH_{g})
=\frac{1}{4g}\left(1+\frac{12+\pi^2}{24g}
+O\left(\frac{1}{g^2}\right)\right)
\quad\text{as } g\to+\infty
\,.
\end{equation}
\end{Theorem}

The existence of the limit $\lim_{N\to\infty} \probg(N)$ and
expression~\eqref{eq:p1:Hg:def} for its value are proved at the end
of Section~\ref{s:proofs:fixed:g:and:n}. Asymptotic
relation~\eqref{eq:p1plus:large:g} is proved in
Corollary~\ref{cor:vol:ab} in Section~\ref{ss:large:genus:Abelian}.

\subsection{Count of simple closed
geodesics on hyperbolic surfaces with cusps}
\label{ssec: nonsep}

We pass now to a different problem concerning closed geodesics on
hyperbolic surfaces, that we are able to solve using the techniques
developed in this paper. We refer to \cite{DGZZ-volumes} for more
details about the relation between this problem and the evaluation of
Masur--Veech volumes.

M.~Mirzakhani has counted
in~\cite{Mirzakhani:growth:of:simple:geodesics} asymptotic
frequencies of simple closed geodesics (and, more generally, of
simple closed geodesic multi-curves) on a hyperbolic surface of genus
$g$ with $n$ cusps. We distinguish the \textit{non-separating} simple
closed geodesics and the \textit{separating} ones. In the latter case
we count all separating simple closed geodesics, no matter the
resulting decomposition of the surface of genus $g$ with $n$ cusps
into two subsurfaces of genera $g_1+g_2=g$ and no matter how the
cusps are partitioned between the two subsurfaces. Denote by
$c_{g,n,\mathrm{sep}}$ and by $c_{g,n,\mathrm{nonsep}}$ the
corresponding Mirzakhani's frequencies.

Our asymptotic count of meanders in the regime when one of the two
parameters $g,\nbigons$ is fixed and the other tends to infinity
implies the following two results (namely,
Theorems~\ref{th:ratio:sep:nonsep:large:p}
and~\ref{th:ratio:sep:nonsep:large:g}) on asymptotic count of simple
closed hyperbolic geodesics.

\begin{Theorem}
\label{th:ratio:sep:nonsep:large:p}
The ratio of frequencies of separating over nonseparating simple
closed geodesics on a closed hyperbolic surface of genus $g=1$ with
$n$ cusps has the following asymptotics:
\begin{equation}
\label{eq:sep:over:nonsep:g:1:large:p}
\lim_{n\to\infty}
\frac{c_{1,n,\mathrm{sep}}}{c_{1,n,\mathrm{nonsep}}}
=\frac{1}{6}\,.
\end{equation}
The ratio of frequencies of separating over nonseparating simple
closed geodesics on a closed hyperbolic surface of genus $g\ge 2$
with $n$ cusps has the following asymptotics:
\begin{equation}
\label{eq:sep:over:nonsep:large:p}
\lim_{n\to\infty}
\frac{c_{g,n,\mathrm{sep}}}{c_{g,n,\mathrm{nonsep}}}
=
\frac{1}{12^g\cdot g!\cdot
\sum_{k=0}^{3g-4}
\langle\tau_k\tau_{3g-4-k}\rangle_{g-1}}\,,
\end{equation}
where $\langle\tau_k\tau_{3g-4-k}\rangle_{g-1}$ are the
Witten--Kontsevich correlators which can be computed recursively by
formulas~\eqref{eq:tau:g:k:difference:1}--\eqref{eq:tau:g:k:difference:3}.
\end{Theorem}

\begin{Remark}
M.~Mirzakhani proved that the ratio
$\frac{c_{g,n,\mathrm{sep}}}{c_{g,n,\mathrm{nonsep}}}$
is shared by all hyperbolic surfaces of genus $g$
with $n$ cusps, see~\cite[Corollary 1.4]{Mirzakhani:growth:of:simple:geodesics}.
In particular, taking a very symmetric hairy torus,
which has marked points at $n^2$ torsion points of your
favorite elliptic curve as in the example
from~\cite[Section~5.3]{Buser} or any
other randomly chosen hairy torus with a very large number of cusps
a random simple closed geodesics happens to be separating
with the same asymptotic probability $\frac{1}{7}$.
\end{Remark}
\begin{table}[hbt]
$$
\begin{array}{|c|c|c|c|c|c|c|c|c|c|}
\hline
g&1&2&3&4&5&6&7&8&9
\\ \hline &&&&&&&&&
\\[-\halfbls]
&\frac{1}{6}
&\frac{1}{36}
&\frac{5}{882}
&\frac{35}{28344}
&\frac{7}{25218}
&\frac{77}{1210716}
&\frac{143}{9686190}
&\frac{715}{206641008}
&\frac{12155}{14878191186}
\\ [-\halfbls] &&&&&&&&&
\\ \hline
\end{array}
$$
\caption{
\label{tab:lim:cgnsep:cgnnonsep}
Values of $\lim_{n\to\infty}
\frac{c_{g,n,\mathrm{sep}}}{c_{g,n,\mathrm{nonsep}}}
$ for $g=1,\dots,9$.
}
\end{table}

Note that though for any fixed $g$ we get a nonzero limit, the
right-hand side of~\eqref{eq:sep:over:nonsep:large:p} rapidly
decreases when $g$ grows. Table~\ref{tab:lim:cgnsep:cgnnonsep} provides the exact values
of the expression in the right-hand side
of~\eqref{eq:sep:over:nonsep:large:p} for small genera $g$.

The asymptotic value of the
sum of $2$-correlators in genus $g$
computed in~\eqref{eq:sum:of:ratios:of:binomials}
in Appendix~\ref{ss:asymptotics:2:correlators}
implies the following asymptotics for the expression in
the right-hand side of~\eqref{eq:sep:over:nonsep:large:p}:
\begin{equation}
\label{eq:sep:nonsep:first:n:then:g:to:infty}
\frac{1}{
12^g\cdot g!\cdot
\sum_{k=0}^{3g-4}
\langle\tau_k\tau_{3g-4-k}\rangle_{g-1}}
\sim
\frac{2}{\sqrt{3\pi g}}\cdot\frac{1}{4^g}
\quad\text{as }g\to+\infty\,.
\end{equation}

Theorem below describes the asymptotics of the ratio
$\frac{c_{g,n,\mathrm{sep}}}{c_{g,n,\mathrm{nonsep}}}$
in the complementary regime, when the number $n$ of cusps
is fixed and $g\to+\infty$.

\begin{Theorem}
\label{th:ratio:sep:nonsep:large:g}
For any fixed number $n\ge 0$ of cusps,
the ratio of frequencies of separating over nonseparating
simple closed geodesics on a closed
hyperbolic surface of genus $g$
has the following asymptotics:
\begin{equation}
\label{eq:sep:over:nonsep:large:g}
\frac{c_{g,n,\mathrm{sep}}}{c_{g,n,\mathrm{nonsep}}}
\sim
\sqrt{\frac{2}{3\pi g}}\cdot\frac{1}{4^g}
\quad\text{as }g\to+\infty\,.
\end{equation}
In particular, it does not depend on $n$, as soon as $n$ is fixed.
\end{Theorem}

Theorem~\ref{th:ratio:sep:nonsep:large:g} is proved in
Section~\ref{ss:large:genus}.

Morally, the
asymptotics~\eqref{eq:sep:nonsep:first:n:then:g:to:infty} represents
the ratio $\frac{c_{g,n,\mathrm{sep}}}{c_{g,n,\mathrm{nonsep}}}$ in
the regime $1\ll g\ll n$, while the
asymptotics~\eqref{eq:sep:over:nonsep:large:g} represents the
ratio
$\frac{c_{g,n,\mathrm{sep}}}{c_{g,n,\mathrm{nonsep}}}$
in the regime $1\ll n\ll g$. The resulting asymptotics
differ by a factor $\sqrt{2}$. Numerical simulations suggest the
following conjectural uniform asymptotics:
\begin{Conjecture}
\!\!\!\footnote{The conjecture was proved by
I.~Ren~\cite{Ren} during the last stage
of preparation of the current paper. The function $f$ found independently
by I.~Ren and by the authors has the form $f(t)=\sqrt{\frac{6+2t}{6+t}}$.}
The ratio of frequencies of separating over nonseparating simple
closed geodesics on a closed hyperbolic surface of genus $g$ with $n$
cusps admits the following uniform asymptotics:
\begin{equation}
\label{eq:sep:over:nonsep:large:g:uniform}
\frac{c_{g,n,\mathrm{sep}}}{c_{g,n,\mathrm{nonsep}}}
=
\sqrt{\frac{2}{3\pi g}}\cdot\frac{1}{4^g}
\cdot f\left(\frac{n}{g}\right)
\big(1+\varepsilon(g,n)\big)\,,
\end{equation}
where the function $f:[0;+\infty]\to\mathbb{R}$
is continuous and increases monotonously from $f(0)=1$ to
$f(\infty)=\sqrt{2}$ and the error term $\varepsilon(g,n)$
tends to $0$ as $g\to+\infty$
uniformly in $n$.
\end{Conjecture}

\begin{Remark}
The above conjecture claims that the dependence of
$\frac{c_{g,n,\mathrm{sep}}}{c_{g,n,\mathrm{nonsep}}}$ on the ratio
$\frac{n}{g}$ is moderate for \textit{any} hyperbolic surface of
large genus. Such a behavior of asymptotic frequencies of simple
closed geodesics is yet another manifestation of its topological
nature in a contrast with geometric quantities, for which the regimes
$n^2\ll g$ and $n^2\gg g$ are drastically different. For example, by
the results of A.~Aggarwal~\cite[Theorem 1.5 and Remark
1.6]{Agg:vol:quad}, the normalized Witten--Kontsevich correlators are
uniformly close to $1$ in the regime $n^2\ll g$ and might explode
exponentially in the complementary regime. Similarly, by the results
of Yang Shen and Yunhui Wu~\cite{Shen:Wu} the spectral gap vanishes
for Weil--Petersson random hyperbolic surfaces in the regime $n^2\gg
g$ (at least under an extra technical assumption). See results cited
in Remark~\ref{rm:large:n:large:g} for more details.
\end{Remark}

\section{Square-tiled surfaces and Masur--Veech volumes}
\label{s:Square:tiled:surfaces:and:MV:volumes}

In the current Section we recall the relevant information on the
count of square-tiled surfaces. In the next Section~\ref{sec:flat} we
express the count of higher genus meanders in terms of the count of
certain special square-tiled surfaces and derive those results
announced Section~\ref{sec:def}, which concern any fixed $g$ and $n$
from the results of the current Section.

\subsection{Masur--Veech volume of the moduli space of
quadratic differentials}
\label{ss:MV:volume}
Consider the moduli space $\cM_{g,n}$ of complex curves of genus $g$
with $n$ distinct labeled marked points. The total space $\cQ_{g,n}$
of the cotangent bundle over $\cM_{g,n}$ can be identified with the
moduli space of pairs $(C,q)$, where $C\in\cM_{g,n}$ is a smooth
complex curve with $n$ (labeled) marked points and $q$ is a
meromorphic quadratic differential on $C$ with at most simple poles
at the marked points and no other poles. In the case $n=0$ the
quadratic differential $q$ is holomorphic. Thus, the \textit{moduli
space of quadratic differentials} $\cQ_{g,n}$ is endowed with the
canonical real symplectic structure. The induced volume element
$\dVolMV$ on $\cQ_{g,n}$ is called the \textit{Masur--Veech volume
element}.

A non-zero quadratic differential $q$ in $\cQ_{g,n}$
defines a flat metric $|q|$ on the complex curve $C$. The
resulting metric has conical singularities at zeroes and
simple poles of $q$. The total area of $(C,q)$
$$
\Area(C,q)=\int_C |q|
$$
is positive and finite. For any real $a > 0$, consider the following
subset in $\cQ_{g,n}$:
$$
\cQ^{\Area\le a}_{g,n} := \left\{(C,q)\in\cQ_{g,n}\,|\, \Area(C,q) \le a\right\}\,.
$$
Since $\Area(C,q)$ is a norm in each fiber of the bundle
$\cQ_{g,n} \to \cM_{g,n}$, the set $\cQ^{\Area \le a}_{g,n}$
is a ball bundle over $\cM_{g,n}$. In particular, it is non-compact.
However, by the independent results of H.~Masur~\cite{Masur:82} and
W.~Veech~\cite{Veech:Gauss:measures}, the total mass of $\cQ^{\Area\le
a}_{g,n}$ with respect to the Masur--Veech volume element is finite.

\subsection{Square-tiled surfaces}
\label{ss:Square:tiled:surfaces}

One can construct a discrete collection of quadratic differentials by
assembling together identical flat squares in the following way. Take
a finite set of copies of the oriented $1/2 \times 1/2$-square for
which two opposite sides are chosen to be horizontal and the
remaining two sides are declared to be vertical. Identify pairs of
sides of the squares by isometries in such way that horizontal sides
are glued to horizontal ones and vertical sides to vertical ones. We
get a topological surface $S$ without boundary. We consider only
those surfaces obtained in this way which are connected and oriented.
The total area $\Area(S,q)$ is $\frac{1}{4}$ times the number of
squares. We call such surface a \textit{square-tiled surface}.

Consider the complex coordinate $z$ in each square and a quadratic
differential $(dz)^2$. It is easy to check that the resulting
square-tiled surface inherits the complex structure and globally
defined meromorphic quadratic differential $q$ having simple poles at
all conical singularities of angle $\pi$ and no other poles. Thus,
any square-tiled surface of genus $g$ having $n$ conical
singularities of angle $\pi$ canonically defines a point
$(C,q)\in\cQ_{g,n}$ (after labeling the conical singularities).
Fixing the size of the square once and forever and considering all
resulting square-tiled surfaces in $\cQ_{g,n}$ we get a discrete
subset $\cSTgn$ in $\cQ_{g,n}$.

Given a sequence of integers $\mu = [\mu_1 \ldots
\mu_m,\mu_{m+1}\dots\mu_{m+n}]$, where $[\mu_1\dots\mu_m]$ is a
partition of $4g-4+n$ and $\mu_{m+1}=\dots=\mu_{n+m}=-1$, the
corresponding \textit{stratum of quadratic differentials} $\cQ(\mu)$
is the space of equivalence classes of pairs consisting of a complex
curve $C$ with $m+n$ distinct marked points $z_1,\dots, z_m, p_1$,
\ldots, $p_n$ and a quadratic differential $q$ with the divisor
$\sum_{i=1}^m \mu_i z_i - \sum_{j=1}^n p_j$ (both zeroes and poles of
$q$ are considered to be labeled).

For any pair of nonnegative integers $(g,n)$ satisfying $2g+n > 3$,
the \textit{principal stratum} of meromorphic quadratic differentials
with at most simple poles is $\cQ(1^{4g-4+n}, -1^n)$ (that is,
$\mu=[1^{4g-4+n},-1^n]$). The natural morphism $\cQ(1^{4g-4+n}, -1^n)
\to \cQ_{g,n}$ that forgets the labeling of zeroes of $q$ is a
$(4g-4+n)!$-sheeted ramified cover of its image in $\cQ_{g,n}$.
Moreover, this image is open and dense in $\cQ_{g,n}$.

Denote by $\cSTgn(N)\subset\cSTgn$ the subset of square-tiled
surfaces in $\cQ_{g,n}$ made of at most $N$ identical squares. The
strata have a natural locally linear structure given by period
coordinates. The square-tiled surfaces form a lattice in period
coordinates in every stratum. This lattice defines a natural volume
element in the stratum by normalizing the volume of the fundamental
domain of the lattice to $1$. In the case of the principal stratum
the resulting volume element differs from the volume element induced
from the natural symplectic structure on $\cQ_{g,n}$ by a constant
factor depending only on $g$ and $n$. This justifies the following
conventional definition of the Masur--Veech volume of $\cQ_{g,n}$
(for $2g+n\ge 4$):
\begin{equation}
\label{eq:VolQ:N:d}
\Vol\cQ_{g,n}
:= \Vol \cQ(1^{4g-4+n}, -1^n)
= 2 d \cdot
\lim_{N\to+\infty}
\frac{\card\left(\cSTgn(2N)\right)}{N^{d}}\,,
\end{equation}
where
\begin{equation}
\label{eq:dim}
d=\dim_{\mathbb{C}}\cQ_{g,\nbigons}
=\dim_\C \cQ(1^{4g-4+n}, -1^n) =6g-6+2\nbigons\,.
\end{equation}
We emphasize that in the above formula we assume that all
conical singularities of square-tiled surfaces are labeled.

The cardinality of the subset of square-tiled surfaces in
$\cSTgn(2N)$ which belong to strata different from the
principal one is negligible
as $N\to+\infty$, so restricting the count to
square-tiled surfaces in the principal stratum
$\cQ(1^{4g-4+n}, -1^n)$
does not change the above limit.
We denote by $\cST(\cQ(\mu),N)\subset\cSTgn(N)$ the subset of
square-tiled surfaces in $\cQ(\mu)\subset\cQ_{g,n}$
tiled with at most $N$ identical squares.

We admit that certain conventions used in the
definition~\eqref{eq:VolQ:N:d} might seem unexpected. For example,
the square-tiled surfaces in $\cSTgn(2N)$ are made of at most $2N$
squares, while we normalize the cardinality of this set by $N^d$.
Also, as we already mentioned, the principal stratum $\cQ(1^{4g-4+n},
-1^n)$ is a $(4g-4+n)!$-sheeted cover over an open and dense subspace
in $\cQ_{g,n}$. However, the normalization in~\eqref{eq:VolQ:N:d}
follows the one used in the literature including~\cite{Agg:vol:quad},
\cite{ADGZZ}, \cite{AEZ:genus:0}, \cite{CMS:quad},
\cite{DGZZ-volumes}, \cite{Goujard:volumes}. Natural normalizations
are compared in~\cite[Appendix A]{DGZZ-meander}.

\subsection{Abelian square-tiled surfaces}

The total space $\cH_g$ of the Hodge bundle over $\cM_{g}$ can be
identified with the moduli space of pairs $(C,\omega)$, where
$C\in\cM_{g}$ is a smooth complex curve of genus $g$ and $\omega$ is
a holomorphic 1-form (Abelian differential of the first kind) on $C$.
As in the case of quadratic differentials, the moduli space $\cH_g$
is stratified and each stratum $\cH(\mu)$, where
$\mu$ is a partition of $2g-2$, admits a locally linear structure
defined by period coordinates.

One can also construct \textit{Abelian} square-tiled surfaces
$\cSTgAb$ living in $\cH_g$. This time we consider copies of the unit
square $0\le x,y \le 1$ from the positive quadrant of the standard
Euclidean plane. To get an \textit{Abelian} square-tiled surface we
not only identify horizontal sides of squares to horizontal sides
and vertical to vertical ones, but also respect the orientation of these
sides inherited from the axes $(Ox)$ and $(Oy)$. We denote by
$\cSTgAb(N)\subset\cSTgAb$ the subset of Abelian square-tiled
surfaces in $\cH_g$ tiled with at most $N$ identical squares, and by
$\cSTAb(\cH(\mu),N)$ the subset of Abelian square-tiled surfaces in
the stratum $\cH(\mu)$ tiled with at most $N$ identical unit squares.
By convention, we always label all zeroes of Abelian differentials
(and, thus, all conical points of square-tiled surfaces).

As in the case of quadratic
differentials, the only stratum of dimension $d=\dim\cH_g$
(called the \textit{principal} stratum)
is the stratum of Abelian differentials with only simple zeros.
We have
$$
\card(\cSTgAb(N))-\card(\cSTAb(\cH(1^{2g-2}),N))
= o(N^d)
\ \text{ as }\ N\to+\infty\,,
$$
where $d=\dim_{\mathbb{C}}\cH_g=\dim_{\mathbb{C}}\cH(1^{2g-2})=4g-3$.

As in the case of quadratic differentials, square-tiled surfaces form
a lattice in period coordinates. This lattice provides a canonical
normalization of the Masur--Veech volume element. The Masur--Veech
volume $\Vol\cH_g$ is defined as
\begin{equation}
\label{eq:VolH:N:d}
\Vol\cH_g= 2d\cdot\lim_{N\to+\infty}\frac{\card(\cSTgAb(N))}{N^d}\,.
\end{equation}
The fact that for each $g\in\N$ a finite limit in~\eqref{eq:VolH:N:d}
exists and is strictly positive is a nontrivial statement which
follows from independent results of H.~Masur~\cite{Masur:82} and
W.~Veech~\cite{Veech:Gauss:measures}. The Masur--Veech
volumes of several low-dimensional strata of Abelian differentials
were computed in~\cite{Zorich:square:tiled} by counting of
square-tiled surfaces. The first efficient algorithm for evaluation of the
Masur--Veech volumes of strata of Abelian differentials was
elaborated by A.Eskin and A.~Okounkov
in~\cite{Eskin:Okounkov:Inventiones} using quasimodularity of
the associated generating function.

\subsection{Count of single-band square-tiled surfaces}
For the purposes of the current paper we distinguish square-tiled
surfaces of the following types. We say that a square-tiled surface
has a \textit{single horizontal cylinder} if the complement
to the union of singular horizontal leaves is connected. Clearly,
this complement is a flat cylinder tiled with squares. We distinguish
the particular case when, moreover, this single horizontal cylinder is composed of a \textit{single horizontal band of squares}.

We performed in~\cite{DGZZ-Yoccoz}--\cite{DGZZ-volumes} the count of
$k$-cylinder square-tiled surfaces for $k=1,2,\dots$. The count of
one-cylinder square-tiled surfaces is treated in detail
in~\cite{DGZZ-Yoccoz} in the Abelian case and in~\cite{DGZZ-volumes}
in the quadratic case. We summarize below the relevant results.

\begin{Theorem}[\cite{DGZZ-Yoccoz}--\cite{DGZZ-volumes}]
\label{th:c1}
The number $\card\big(\cST_{\!1}(\cQ_{g,\nbigons},2N)\big)$ of
square-tiled surfaces in the moduli space $\cQ_{g,\nbigons}$ with
labeled zeros and poles tiled with at most $2N$ squares organized
into a single horizontal band of squares has asymptotics
\begin{equation}
\label{eq:c1:N:d}
\card\big(
\cST_{\!1}(\cQ_{g,\nbigons},2N)\big)=
\cyl_1(\cQ_{g,\nbigons}) \cdot \frac{N^d}{2d} + O(N^{d-1})
\ \text{ as }\ N\to+\infty\,,
\end{equation}
where $d=\dim_{\mathbb{C}}\cQ_{g,\nbigons}=6g-6+2\nbigons$ and the
coefficient $\cyl_1(\cQ_{g,\nbigons})$ is a positive rational number
expressed in terms of Witten--Kontsevich 2-correlators by
formula~\eqref{cyl1Q}.

The number of square-tiled surfaces in the moduli space
$\cQ_{g,\nbigons}$ with labeled zeros and poles   tiled with at most
$2N$ squares organized into a single horizontal cylinder has asymptotics
$$
c_1(\cQ_{g,\nbigons}) \cdot \frac{N^d}{2d} + O(N^{d-1})
\ \text{ as }\ N\to+\infty\,,
$$
where the coefficient $c_1(\cQ_{g,\nbigons})$ satisfies the relation
\begin{equation}
\label{eq:c:cyl:quadratic}
c_1(\cQ_{g,\nbigons})=\zeta(6g-6+2\nbigons)\cdot \cyl_1(\cQ_{g,\nbigons})\,.
\end{equation}
\end{Theorem}
The existence of polynomial asymptotics~\eqref{eq:c1:N:d} is proved
in~\cite[Corollary 4.25]{DGZZ-meander}. The fact, that the
coefficient $\cyl_1(\cQ_{g,\nbigons})$ is a positive rational number
given by formula~\eqref{cyl1Q} is contained in the proofs
of Theorem~4.2 and of Proposition~4.4 in~\cite{DGZZ-volumes}. For the
sake of completeness, we present an explicit proof of this formula in
Lemma~\ref{prop:cyl1} in
Section~\ref{ss:contribution:of:single:band}. Finally,
relation~\eqref{eq:c:cyl:quadratic} is an immediate corollary
of~\cite[Formula (1.14) and Lemma 1.32]{DGZZ-volumes}.

   %
\begin{Theorem}[\cite{DGZZ-Yoccoz}, \cite{DGZZ-meander}]
\label{th:c1+}
The number $\card\big(\cSTAb_{\!1}(\cH_g,N)\big)$ of square-tiled
surfaces with labeled zeros in the moduli space $\cH_g$ tiled with at
most $N$ squares organized into a single horizontal band of squares
has asymptotics
\begin{equation}
\label{eq:c1+:N:d}
\card\big(
\cSTAb_{\!1}(\cH_g,N) \big)
=\cyl_1\left(\cH_g\right) \cdot \frac{N^d}{2d} + O(N^{d-1})
\ \text{ as }\ N\to+\infty\,,
\end{equation}
where $d=\dim_{\mathbb{C}}\cH_g=4g-3$ and
\begin{equation}
\label{eq:c1:Hg}
\cyl_1\left(\cH_g\right)=\frac{1}{(2g-1)\cdot 2^{2g-3}}\,.
\end{equation}

The number of square-tiled surfaces in the moduli space $\cH_g$ with
labeled zeros and tiled with at most $N$ squares organized into a
single horizontal cylinder has asymptotics
$$
c_1(\cH_g) \cdot \frac{N^d}{2d} + O(N^{d-1})
\ \text{ as }\ N\to+\infty\,,
$$
where the coefficient $c_1(\cH_g)$ satisfies the relation
\begin{equation}
\label{eq:c:cyl:Abelian}
c_1(\cH_g)=\zeta(4g-3)\cdot \cyl_1(\cH_g)\,.
\end{equation}
\end{Theorem}
The existence of polynomial asymptotics~\eqref{eq:c1+:N:d} is proved
in~\cite[Corollary 4.25]{DGZZ-meander}. The fact, that the
coefficient $\cyl_1(\cH_g)$ is a positive rational number given by
formula~\eqref{eq:c1:Hg} and the relation~\eqref{eq:c:cyl:Abelian}
is the combination of
Equation (2.4) from~\cite[Corollary 2.6]{DGZZ-Yoccoz} with the two
paragraphs preceding~\cite[Corollary 2.6]{DGZZ-Yoccoz}.

As in the case of quadratic differentials, the condition that all
squares of an Abelian square-tiled surface are organized into a
single horizontal cylinder is equivalent to the condition that the
complement to the singular horizontal leaf is connected. By symmetry
arguments, we get the same asymptotics~\eqref{eq:c1:N:d}
and~\eqref{eq:c1+:N:d} with the same constants for the number of
square-tiled surfaces with a single vertical (instead of horizontal)
band of squares.

The next statement describes the count of square-tiled surfaces
having single horizontal and single vertical band of squares.

\begin{Theorem}[\cite{DGZZ-meander}]
\label{th:c11}
Consider square-tiled surfaces $\cSTgn$ with labeled zeroes and poles
in the moduli space $\cQ_{g,\nbigons}$. Consider the subset
$\cST_{\!1,1}(\cQ_{g,\nbigons},2N)\subset\cSTgn$ of those square
tiled-surfaces that are tiled with at most $2N$ squares and that are,
moreover, simultaneously organized into a single horizontal and a
single vertical band of squares. Then
\begin{equation}
\label{eq:c11:Q:nu}
\card(\cST_{\!1,1}(\cQ_{g,\nbigons},2N))=
\cyl_{1,1}\left(\cQ_{g,\nbigons}\right)\cdot\frac{N^d}{2d} +
o\left(N^{d}\right)\text{ as } N\to+\infty\,,
\end{equation}
where the constant $\cyl_{1,1}\left(\cQ_{g,\nbigons}\right)$
satisfies the following relation:
\begin{equation}
\label{eq:c11:as:c1:squared:over:Vol}
\cyl_{1,1}\left(\cQ_{g,\nbigons}\right)
=
\frac{\big(\cyl_{1}\left(\cQ_{g,\nbigons}\right)\big)^2}
{\Vol \cQ_{g,\nbigons}}\,,
\end{equation}
and the constants $\cyl_{1}\left(\cQ_{g,\nbigons}\right)$ and $d$ are the ones
from Theorem~\ref{th:c1}.
\end{Theorem}

We present now the count for Abelian square-tiled surfaces.

\begin{Theorem}[\cite{DGZZ-meander}]
\label{th:c11+}
Consider Abelian square-tiled surfaces $\cSTgAb$ with labeled zeroes
in the moduli space $\cH_g$. Consider the subset
$\cSTAb_{\!1,1}(\cH_g,N)\subset\cSTgn$ of those Abelian square-tiled
surfaces, that are tiled with at most $N$ squares and that are,
moreover, simultaneously organized into a single horizontal and a
single vertical band of squares. Then
\begin{equation}
\label{eq:c11+:H:mu}
\card(\cSTAb_{\!1,1}(\cH_g,N))=
\cyl_{1,1}\left(\cH_g\right)\cdot\frac{N^d}{2d} +
o\left(N^{d}\right)\text{ as } N\to+\infty\,,
\end{equation}
where the constant $\cyl_{1,1}\left(\cH_g\right)$
satisfies the following relation:
\begin{equation}
\label{eq:c11+:as:c1:squared:over:Vol}
\cyl_{1,1}\left(\cH_g\right)
=
\frac{\left(\cyl_{1}(\cH_g)\right)^2}
{\Vol \cH_g}\,,
\end{equation}
and the constants $\cyl_{1}\left(\cH_g\right)$ and $d$ are the ones
from Theorem~\ref{th:c1+}.
\end{Theorem}

Theorem~\ref{th:c11} and Theorem~\ref{th:c11+} are proven
in~\cite[Corollary 4.25]{DGZZ-meander}.

\section{Dictionary of square-tiled surfaces}
\label{sec:flat}


In this Section we reduce the count of higher genus meanders to the
count of square-tiled surfaces and show that non-filling meanders
occur exceptionally rare when the number of intersections becomes
large.

Let $\cG$ be a graph defined as a union of a connected transverse
pair of multicurves. The graph $\cG$ is embedded into a surface $S$.
Consider a sufficiently small closed tubular neighborhood $G$ of
$\cG$ in $S$. We denote by $S'$ a closed surface obtained by pasting
topological discs to all boundary components of the surface with
boundary $G$. By definition, if the initial transverse pair of
multicurves is filling (i.e., if $\cG$ is a map), we get back the
original surface: in this case $S'$ is homeomorphic to $S$.
Otherwise, the surface $S'$ has strictly smaller genus. For example,
a non-filling transverse pair of simple closed curves on the surface
of genus six as in Figure~\ref{fig:non:filling} gives rise to a
surface $S'$ of genus two.

\begin{figure}[htb]
\includegraphics{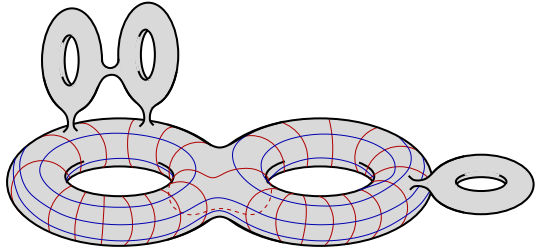}
\vspace{90bp} 
\caption{
\label{fig:non:filling}
Non-filling transverse pair of multicurves.
}
\end{figure}

By construction, $\cG$ is a map in $S'$. The vertices of $\cG$ are
intersections of the multicurves, so all vertices of $\cG$ have
valence $4$. Hence, all faces of the dual graph $\cG^\ast$ in $S'$
are $4$-gons. The edges of $\cG^\ast$ dual to horizontal edges of
$\cG$ will be called vertical, and those dual to the vertical edges
of $\cG$ will be called horizontal. By construction, any two opposite
edges of any face of $\cG^\ast$ are either both horizontal or both
vertical. Realizing the faces of $\cG^\ast$ as identical metric
squares we get a \textit{square-tiled surface} in the sense of
Section~\ref{s:Square:tiled:surfaces:and:MV:volumes}.

In the case when the transverse pair of multicurves is not filling,
we introduce an additional marking of the associated square-tiled
surface in order to record the information on the initial surface.
We have to mark disjoint collections
$\{V_{1,1},\dots, V_{1,j_1}\}\sqcup\dots\sqcup\{V_{k,1},\dots,
V_{1,j_k}\}$ of disjoint vertices of the tiling and genera
$g_1,\dots,g_k$ of associated surfaces with respectively
$j_1,\dots,j_k$ boundary components. The genus $g$ of the initial
surface and the genus $g'$ of the associated square-tiled surface are
related by
\begin{equation}
\label{eq:marking:equation:on:genus}
g=g'+(g_1+j_1-1)+\dots+(g_k+j_k-1)\,,
\end{equation}
where $g_i+j_i-1> 0$ for $i=1,\dots,k$. A
square-tiled surface endowed with a marking
$$
(\{V_{1,1},\dots, V_{1,j_1}\},g_1),\dots,(\{V_{k,1},\dots, V_{1,j_k}\},g_k)
$$
defines the original surface endowed with an ordered connected
transverse pair of multicurves uniquely up to a homeomorphism.

We can formalize the above constructions as the following statement.
\begin{Proposition}
\label{prop:pairs:multicurves:square:tiled:surfaces}
There is a natural one-to-one correspondence between filling
connected pairs of transverse multicurves on a surface of genus $g$
and square-tiled surfaces of genus $g$ (with non-labeled conical
points), where the square tiling is given by the graph $\cG^\ast$
dual  to the graph $\cG$ formed by the union of two multicurves.

This correspondence extends to the bijection between non-filling
pairs of transverse multicurves and marked square-tiled surfaces (with
non-labeled conical points), where the marking satisfies
Equation~\eqref{eq:marking:equation:on:genus}.

Restricting the correspondence to filling transverse connected pairs
of simple closed curves we get a bijection with the subset of
square-tiled surfaces of genus $g$ (with non-labeled conical points)
having a single horizontal and a single vertical band of squares.
\end{Proposition}

A square-tiled surface carries a meromorphic quadratic differential
$q$ with at most simple poles. This quadratic differential has the
form $(dz)^2$ in the natural coordinate on each square. Simple poles
of $q$ correspond to bigons of the complement $S-\cG$ (see
Definition~\ref{def:bigons}); zeroes of degree $m$ correspond to
$(4+2m)$-gons. Thus, restricting our consideration to nonorientable
connected transverse pairs of multicurves on a surface $S$ of genus
$g$, which form exactly $\nbigons$ boundary components of the
complement $S-\cG$ having two edges, we get a quadratic differential
in $\cQ_{g,n}$ in the case when the transverse pair of multicurves is
filling (i.e. when $\cG$ forms a \textit{map}) and in $\cQ_{g',n}$
with $g'<g$ in the case when the pair is not filling. Specifying the
numbers $\mu_1,\mu_2,\mu_3,\dots,\mu_m,\dots$ of boundary components
of the complement $S-\cG$ having respectively
$6,8,10,\dots,4+2m,\dots$ edges we get square-tiled surfaces in the
stratum $\cQ(\mu,(-1)^\nbigons)$.

Starting from a positively oriented transverse pair of multicurves
and applying the construction as above, we get an \textit{Abelian}
square-tiled surface endowed with an Abelian differential $\omega$
having the form $dz$ in the natural coordinate on each square. This
time the boundary components of the connected domains obtained by
removing the union of the transverse pair of multicurves from the
surface have $4,8,12,\dots$ edges. Specifying the numbers
$\mu_1,\mu_2,\mu_3,\dots,\mu_m,\dots$ of boundary components of the
complement $S-\cG$ having respectively $8,12,\dots,4+4m,\dots$ edges
we get square-tiled surfaces in the stratum $\cH(\mu)$. The
Proposition below is an analog of
Proposition~\ref{prop:pairs:multicurves:square:tiled:surfaces}.

\begin{Proposition}
\label{prop:pairs:multicurves:square:tiled:surfaces:oriented}
There is a natural one-to-one correspondence between filling oriented
transverse connected pairs of multicurves on a surface of genus $g$
and Abelian square-tiled surfaces of genus $g$ (with non-labeled
conical points), where the square tiling is given by the graph
$\cG^\ast$ dual  to the graph $\cG$ formed by the union of two
multicurves.

This correspondence extends to the bijection between non-filling
oriented transverse pairs of multicurves and marked Abelian
square-tiled surfaces (with non-labeled conical points), where the
marking satisfies Equation~\eqref{eq:marking:equation:on:genus}.

Restricting the correspondence to filling oriented transverse
connected pairs of simple closed curves we get a bijection with the
subset of Abelian square-tiled surfaces of genus $g$ (with
non-labeled conical points) having a single horizontal and a single
vertical band of squares.
\end{Proposition}

We are ready now to present our first counting result.

\begin{Proposition}
\label{prop:counting:non:orientable:pairs}
For any fixed $g$ and $\nbigons$ satisfying $2g+n\ge 4$
consider transverse connected pairs of
multicurves with at most $2N$ crossings on a surface $S$ of genus $g$,
such that the corresponding graph $\cG$ forms at most
$\nbigons$ two-edges boundary components of the complement $S-\cG$.

The total number of pairs as above which satisfy any of the following
properties:
\begin{enumerate}
\item the pair is not filling;
\item the pair is orientable (can take place only when $\nbigons=0$);
\item at least one of the boundary components of the complement $S-\cG$
has more than six edges;
\end{enumerate}
is of order $o(N^{6g-6+2\nbigons})$ as $N\to\infty$.

The number
$\MulticurvesNumber^{\textrm{filling}}_{g,\nbigons}(N)$ of filling pairs
as above which do not satisfy any of the properties
(1)--(3) has the following asymptotics:
   %
\begin{equation}
\label{eq:N:filling:pairs}
\MulticurvesNumber^{\textrm{filling}}_{g,\nbigons}(N)
=\frac{\Vol\cQ_{g,\nbigons}}{(4g-4+\nbigons)!\cdot\nbigons!\cdot(12g-12+4n)}
\cdot N^{6g-6+2\nbigons}
+o(N^{6g-6+2\nbigons})\quad \mbox{ as }N\to\infty\,.
\end{equation}
\end{Proposition}
\begin{proof}
Suppose that a pair of multicurves as above does not satisfy any of
the properties (1)--(3). Then by
Proposition~\eqref{prop:pairs:multicurves:square:tiled:surfaces} the
number $\MulticurvesNumber^{\textrm{filling}}_{g,\nbigons}(N)$ counts
square-tiled surfaces with \textit{non-labeled} zeroes
and poles in the stratum $\cQ(1^{4g-4+\nbigons},(-1)^\nbigons)$ of
meromorphic quadratic differentials. There are
$(4g-4+\nbigons)!\cdot\nbigons!$ ways to label $(4g-4+\nbigons)$
zeroes and $\nbigons$ poles, so Equation~\eqref{eq:N:filling:pairs}
follows from~\eqref{eq:VolQ:N:d}.

Suppose that a pair of multicurves is orientable (this implies that $\nbigons=0$).
Such a pair defines an Abelian square-tiled surface,
that represents a point in
one of the finite number of strata $\cH_{g'}$, where $g'\le
g$. The number of square-tiled surfaces tiled with at most $2N$
squares in any given stratum $\cH(\mu)$ grows as $const\cdot N^d$
as $N\to\infty$, where $d=\dim_{\mathbb{C}}\cH(\mu)$. Any stratum
of Abelian differentials in genus $g'\le g$ has dimension bounded
from above by the dimension $\dim_{\mathbb{C}}\cH_g=4g-3$ of $\cH_g$.
The inequality $2g+n\ge 4$ implies that $4g-3<6g-6+n$. This proves,
that the number of filling orientable pairs is negligible compared to
the number $\MulticurvesNumber^{\textrm{filling}}_{g,\nbigons}(N)$ of
nonorientable filling pairs computed above as $N\to\infty$.

Similarly, if the pair is filling, nonorientable, but at least one of
the faces has more than six edges, then the associated square-tiled
surface lives in one of the finite number of strata
$\cQ(\mu,(-1)^\nbigons)$ of meromorphic quadratic differentials in
$\cQ_{g,\nbigons}$, different from the principal stratum. The number
of square-tiled surfaces tiled with at most $2N$ squares in a stratum
$\cQ(\mu,(-1)^\nbigons)$ grows as $const\cdot N^d$ as
$N\to\infty$, where $d=\dim_{\mathbb{C}}\cQ(\mu,(-1)^\nbigons)$.
The dimension of any stratum in $\cQ_{g,\nbigons}$ different from the
principal stratum $\cQ(1^{4g-4+\nbigons},(-1)^\nbigons)$ is strictly
less than $\dim_{\mathbb{C}}\cQ_{g,\nbigons}=6g-6+2n$. This proves,
that the number of filling nonorientable pairs having at least one of
the faces with more than six edges is negligible compared to the
number $\MulticurvesNumber^{\textrm{filling}}_{g,\nbigons}(N)$ of
nonorientable filling pairs computed above.

It remains to prove that the number of non-filling pairs as above is
negligible. Equation~\eqref{eq:marking:equation:on:genus} implies
that there is a finite number of choices of parameters $k$ and
$g',g_1,j_1,\dots,g_k, j_k$. Thus, it is sufficient to prove the
statement under assumption that all these parameters are fixed, which
we impose from now on. The dimensional arguments as above allow to
restrict consideration to the situation when the resulting
square-tiled surface belongs to the principal stratum
$\cQ(1^{4g'-4+\nbigons},(-1)^\nbigons)$.

The total number $v$ of vertices of any square-tiled surface of genus
$g'$ with exactly $k$ squares satisfies the Euler characteristic
relation $k -2k + v = 2-2g'$. Thus, when the number of squares is at
most $2N$, the number of choices for any $V_{i,j}$ is at most $2N+2$.
Thus, the number of choices of all marked points has the order
$const_1\cdot N^J$, where $J=j_1+\dots+j_k$, as $N\to\infty$. The
number of square-tiled surfaces in the stratum
$\cQ(1^{4g'-4+\nbigons},(-1)^\nbigons)$ tiled with at most $2N$
squares grows as $const_2\cdot N^{6g'-6+2\nbigons}$ as $N\to\infty$.
This implies that the total number of marked square-tiled surfaces
tiled with at most $2N$ squares is bounded by $const_3\cdot
N^{6g'-6+2\nbigons+J}$ as $N\to\infty$.

Let $G=g_1+\dots+g_k$. Recall that the entries of
Equation~\eqref{eq:marking:equation:on:genus} satisfy the conditions
$g_i+j_i-1> 0$ for $i=1,\dots,k$. This implies that $G+J-k\ge k$.
Equation~\eqref{eq:marking:equation:on:genus} implies that
$$
6g-6+2\nbigons=
(6g'-6+2\nbigons+J)+G-k+5(G+J-k)\,.
$$
Since $G\ge 0$ and by assumption $k\ge 1$ we conclude that
\begin{multline*}
(6g'-6+2\nbigons+J)
=6g-6+2\nbigons - (G-k+5(G+J-k))
\\ \le
6g-6+2\nbigons - (-k+5k)
\le 6g-6+2\nbigons - 4\,.
\end{multline*}
Taking into consideration the asymptotic
relation~\eqref{eq:N:filling:pairs}, which we have already proved,
this completes the proof of
Proposition~\ref{prop:counting:non:orientable:pairs}.
\end{proof}

The statement below is completely analogous:

\begin{Proposition}
\label{prop:counting:orientable:pairs}
For any fixed $g\ge 1$ consider transverse connected oriented pairs of
multicurves with at most $N$ crossings on a surface $S$ of genus $g$.

The total number of pairs as above which satisfy any of the following
two properties:
\begin{enumerate}
\item the pair is not filling;
\item at least one of the boundary components of the complement $S-\cG$
has more than $8$ edges;
\end{enumerate}
is of order $o(N^{4g-3})$ as $N\to\infty$.

The number
$\MulticurvesNumber^{+,\textrm{filling}}_g(N)$ of filling pairs
as above which do not satisfy any of the properties
(1)--(2) has the following asymptotics:
\begin{equation}
\label{eq:N:oriented:filling:pairs}
\MulticurvesNumber^{+,\textrm{filling}}_g(N)
=\frac{\Vol\cH_g}{(2g-2)!\cdot(8g-6)}
\cdot N^{4g-3}
+
o(N^{4g-3})\quad \mbox{ as }N\to\infty\,.
\end{equation}
\end{Proposition}

\section{Proofs of the main result for fixed values of $g$ and $\nbigons$}
\label{s:proofs:fixed:g:and:n}

In this Section we derive part of the main results of
Section~\ref{sec:def} from the count presented in
Section~\ref{s:Square:tiled:surfaces:and:MV:volumes}.

\begin{proof}[Proof of relations~\eqref{eq:asymptotics:Mgp}
and~\eqref{eq:constant:Cgp} from Theorem~\ref{th:meander}]
By Proposition~\ref{prop:pairs:multicurves:square:tiled:surfaces}
nonorientable filling meanders on a surface of genus $g$ with at most
$2N$ intersections having exactly $\nbigons$ bigonal faces and no
faces with more than $6$ edges are in the natural one-to-one
correspondence with square-tiled surfaces in the stratum
$\cQ(1^{4g-4+\nbigons},(-1)^\nbigons)$ tiled with at most $2N$
squares, with non-labeled zeroes and poles, and having a single
horizontal and a single vertical cylinder. The number of such
square-tiled surfaces with \textit{labeled} zeroes and poles is given
by Formula~\eqref{eq:c11:Q:nu} from Theorem~\ref{th:c11}. Dividing
both sides of~\eqref{eq:c11:Q:nu} by the number
$(4g-4+\nbigons)!\cdot\nbigons!$ of different labelings we get the
number of unlabeled square-tiled surfaces as above, i.e. the number
$\MeandNumber^{\textrm{filling}}_{g,\nbigons}(N)$ of nonorientable
filling meanders with $\nbigons$ bigonal and with $4g-4+\nbigons$
hexagonal faces, and with no faces of more than $6$ edges
\[
\MeandNumber^{\textrm{filling}}_{g,\nbigons}(N)=C_{g,\nbigons} N^{6g-6+2\nbigons}
+ o(N^{6g-6+2\nbigons})\quad \mbox{ as }N\to\infty\,,
\]
with $C_{g,\nbigons}$ given by Equation~\eqref{eq:constant:Cgp}.
Proposition~\ref{prop:counting:non:orientable:pairs}
implies that all but negligible (as $N\to+\infty$) part of meanders
as above are nonorientable, filling, and have only bigonal,
quadrangular and hexagonal faces, or, equivalently,
$$
\MeandNumber_{g,\nbigons}(N)
=\MeandNumber^{\textrm{filling}}_{g,\nbigons}(N)
+
o(N^{6g-6+2\nbigons})\quad\text{ as }
N\to\infty\,.
$$
This completes the proof of~\eqref{eq:asymptotics:Mgp}.
Expression~\eqref{eq:cyl:11:definition} for
$\cyl_{1,1}(\cQ_{g,\nbigons})$ corresponds to
Equation~\eqref{eq:c11:as:c1:squared:over:Vol} from
Theorem~\ref{th:c11}.
\end{proof}

The proof of the part of Theorem~\ref{th:oriented_meander} which
concerns any fixed $g$ (i.e. existence polynomial
asymptotics~\eqref{eq:M:plus} and the fact that the coefficient of
the leading term is given by expression~\eqref{eq:Cg:plus}) is
completely analogous and is based on Theorem~\ref{th:c11+} and
Proposition~\ref{prop:counting:orientable:pairs}. The value of
$\cyl_{1,1}(\cH_g)$ is given by
Formula~\eqref{eq:c11:as:c1:squared:over:Vol}, the value of
$\cyl_1(\cH_g)$ is given by Formula~\eqref{eq:c1:Hg}.

\begin{proof}[Proof of Theorem~\ref{th:arc_system}]
Gluing the two boundary components in such a way that the endpoints
of a balanced arc system are matched, we get a connected transverse
pair of multicurves. The horizontal multicurve has a single connected
component: it is a simple closed curve represented by the original
boundary component, whereas the vertical multicurve may have several
connected components. All such transverse connected pairs of
multicurves correspond to square-tiled surfaces with unlabeled zeroes
and poles having a single horizontal band of squares. Those, which
represent meanders, correspond to square-tiled surfaces with
unlabeled zeroes and poles having a single horizontal and a single
vertical band of squares. Once again
Proposition~\ref{prop:pairs:multicurves:square:tiled:surfaces:oriented}
allows us to limit our consideration to only those pairs of
multicurves of each of the two types which are nonorientable, filling
and do not have faces with more than $6$ edges. This implies that
$$
\lim_{N\to\infty} \prob_{g,\nbigons}(N)
=\lim_{N\to\infty}
\frac{\Card(\cST^{\mathit{unlabeled}}_{1,1}(\cQ(1^{4g-4+\nbigons},(-1)^\nbigons),2N))}
{\Card(\cST^{\mathit{unlabeled}}_1(\cQ(1^{4g-4+\nbigons},(-1)^\nbigons),2N))}\,.
$$
Since, passing from the count of unlabeled square-tiled surfaces to
the count of labeled ones, we have to label the same number of zeroes
and poles for the square-tiled surfaces in the numerator and in the
denominator, we get the same limit for the ratio of the numbers of
analogous labeled square-tiled surfaces.
Combining Equations~\eqref{eq:c1:N:d},\eqref{eq:c11:Q:nu}
and~\eqref{eq:c11:as:c1:squared:over:Vol} we get
$$
\lim_{N\to\infty}
\frac{\Card(\cST^{\mathit{labeled}}_{1,1}(\cQ(1^{4g-4+\nbigons},(-1)^\nbigons),2N))}
{\Card(\cST^{\mathit{labeled}}_1(\cQ(1^{4g-4+\nbigons},(-1)^\nbigons),2N))}
=
\frac{\cyl_1(\cQ_{g,\nbigons})}
{\Vol \cQ_{g,\nbigons}}\,,
$$
which proves~\eqref{eq:Pgn:equals:p1gn} and~\eqref{eq:p1:definition}.
\end{proof}

The part of Theorem~\ref{th:oriented_arc_system} claiming existence
of the limit~\eqref{eq:oriented:limit} and providing
expression~\eqref{eq:p1:Hg:def} for its value is proved completely
analogously.

The large $n$ and large $g$ asymptotics of the related constants are
given in Corollary~\ref{cor:cyl11_quad_poles},
Corollary~\ref{cor:cyl11_quad_genus} and Corollary~\ref{cor:vol:ab}
of Section~\ref{sec:vol}.

\begin{proof}[Proof of Theorem~\ref{th:all:the:same}]
Having been translated into the language of square-tiled surface,
Theorem~\ref{th:all:the:same} becomes an implication of
Proposition~\ref{prop:counting:non:orientable:pairs} combined
with Corollaries~4.24 and~4.25 from~\cite{DGZZ-meander}.
For the sake of completeness we justify below that all topological
configurations of filling nonorientable systems of arcs mentioned
in Theorem~\ref{th:all:the:same} are realizable.

We start with the case when the surface of genus $g-1$ has two
boundary components. If $g=1$, the proof of existence of a filling
nonorientable system of arcs having $n_1$ minimal arcs at the first
boundary component and $n_2$ minimal arcs at the second boundary
component is trivial for any pair $n_1, n_2\in\N$. Thus, we can
suppose that $g\ge 2$.

\begin{Lemma}
\label{lm:two:components:n:0}
An orientable surface of any genus greater than or equal to one with
exactly two boundary components admits a nonorientable filling
balanced system of arcs such that all faces of the complement are
hexagons.
\end{Lemma}

We present an equivalent formulation of this Lemma suitable for
technique of Corollaries~4.24 and~4.25 from~\cite{DGZZ-meander}.

\begin{Lemma}
For any $g\ge 2$ the stratum $\cQ(1^{4g-4})$ admits a one-cylinder
separatrix diagram such that the singular leaf is connected.
\end{Lemma}
\begin{proof}
One can use the diagram from~\cite[Figure~13]{Zorich:representatives}
for $g=2$ and the diagram
from~\cite[Figure~14]{Zorich:representatives} for $g\ge 3$ letting in
both cases $p=0$ (where $p$ is defined in the Figures).
\end{proof}

We will need the following observation. Note that the original system
of arcs, constructed in Lemma~\ref{lm:two:components:n:0} is
nonorientable. It means that it contains at least one arc coming back
to the same boundary component. If there is an arc at one boundary
component, there should be at least one arc having both endpoints on
the other boundary component since the system is balanced.

Having constructed a system of arcs as in
Lemma~\ref{lm:two:components:n:0} we can complete it with arbitrary
number $n_1$ of minimal arcs at the first boundary component and with
arbitrary number $n_2$ of minimal arcs at the second boundary
component. If $n_1\neq n_2$, the resulting system of arcs is not
balanced, but since each of the boundary components has an arc with
both endpoints on this boundary component, taking several copies of
such an arc for the component with deficiency of endpoints we get
already a balanced systems. By construction it is nonorientable and
filling. If some of the faces have more than 6 sides, we can add
extra arcs to partition them to get only bigons, quadrilaterals and
hexagons. The proof of existence of balanced arc systems on a
connected surface with two boundary components mentioned in
Theorem~\ref{th:all:the:same} is completed. The proof of existence
of remaining arc systems mentioned in Theorem~\ref{th:all:the:same}
is trivial.
\end{proof}

\begin{Remark}
The technique of the above proofs applies without any further changes
to meandric systems. Meandric systems correspond to square-tiled
surfaces having a certain number of horizontal bands of squares and a
certain number of vertical bands of squares, see Section~4 of the
original paper~\cite{DGZZ-meander} for more details.
\end{Remark}

\begin{proof}[Proof of Lemma~\ref{lem:estim:meander}]
One can easily construct a meander with $2N$ crossings on a surface
of genus $g$ from a meander with $2N$ crossings on a sphere by
replacing a topological disk, forming a face of the graph $\cG$, by a
topological surface of genus $g$ having a single boundary component.
This gives an obvious lower bound $M_0^{=N}$ for $M_g^{=N}$. The
number $M_0^{=N}$ is greater than $C_N$, see~\cite{Lando:Zvonkin:92}
(the paper~\cite{Albert:Paterson}, actually, provides a sharper
bound).

For the upper bound, we first assume that the meander is filling, as
it was done above. Cutting along one of the two closed curves forming
the meander we get a filling balanced arc system of genus $g$ with
$2N$ endpoints on each boundary component.
We first consider the case,
when we get two connected components of genera $g_1$ and $g_2$, where
$g_1+g_2=g$. The dual graph $\cG^*$ of the arc system on each
component is a unicellular map with $N$ edges: the unique face
corresponds to the boundary of the cut component. The number
$\epsilon_{g}(N)$ of unicellular maps with $N$ edges on a genus $g$
surface is well-known (see~\cite[Theorem 2]{Harer:Zagier} or
\cite[Theorem 5 and Proposition 6]{Chapuy:Feray:Fusy}):
\[\epsilon_g(N)=C_N\cdot p_g(N)\,,\]
where $p_g$ is an explicit polynomial of degree $3g$. The number of
meanders in that case is then bounded by
$2N\sum_{g_1+g_2=g}\epsilon_{g_1}(N)\epsilon_{g_2}(N)\leq {C_N}^2
P_g(N)$ where $P_g(N)=2N\sum_{g_1+g_2=g} p_{g_1}(N)p_{g_2}(N)$ is of
degree $3g+1$ and the factor $2N$ accounts for the $2N$ possibly
different identifications.

In the second case we get a single connected surface with two
boundary components. Now the dual graph $\cG^*$ to the arc system is
a bicellular map of genus $g-1$ with $2N$ edges: the two faces of the
graph correspond to the two boundary components of the surface. The
number $\epsilon^{[2]}_{g-1}(N)$ of such bicellular maps is bounded
by
\[
\epsilon_{g-1}^{[2]}(2N)\leq \epsilon_{g}(2N+1)
=C_{2N+1}\cdot p_g(2N+1)\,,
\]
see~\cite[Corrolary 1]{Han:Reidys}. The number of
meanders in that case is bounded by $2N\cdot p_g(2N+1)\cdot
C_{2N+1}$, where the factor $2N$ accounts for the $2N$ possibly
different identifications.

By Stirling's approximation we have
$$
(C_N)^2\sim \frac{1}{\pi}\cdot N^{-3}\cdot 2^{4N}\,,\quad
C_{2N+1}\sim \sqrt{\frac{2}{\pi}}\cdot N^{-3/2}\cdot 2^{4N}\,.
$$
Recall that the polynomial $P_g(N)$ has degree $3g+1$ and the
polynomial $p_g(2N+1)$ has degree $3g$ in $N$. Hence, the quantity
$P_g(N)\cdot(C_N)^2$ is negligible compared to $2N\cdot
p_g(2N+1)\cdot C_{2N+1}$ as $N\to\infty$.

The above estimates imply that the contributions of non-filling
meanders are negligible compared to $2N\cdot p_g(2N+1)\cdot
C_{2N+1}$, since the corresponding  dual maps have lower genera.

The rough upper and lower bounds obtained above can be improved using
finer arguments. We do not do it here to avoid overloading of the
paper.
\end{proof}


\section{Asymptotic count for large values of $g$ and $\nbigons$}
\label{sec:vol}

\subsection{Formula for the Masur--Veech volume through
intersection numbers}

We recall here the formula from \cite{DGZZ-volumes} giving the
Masur--Veech volume $\Vol\cQ_{g,n}$ of the moduli spaces $\cQ_{g,n}$
of meromorphic quadratic differentials with $n$ simple poles on
Riemann surfaces of genus $g$.

A multicurve on a surface of genus $g$ with $n$ punctures cuts the
surface into several connected components, where each component has
certain genus, certain number of boundary components and certain
number of punctures. By a \textit{stable graph} we call the dual
graph to such a multicurve, decorated with the following information:
to each vertex of the graph we associate the genus of the
corresponding connected component, and for each puncture on that
component we add a half edge at the corresponding vertex. These
graphs encode the topological type of a multicurve. Using a similar
correspondence as the one described in Section~\ref{sec:flat}, one
can show that these graphs encode also the type of decomposition into
cylinders of a square-tiled surface.

We are particularly interested in stable graphs representing simple closed curves (or, equivalently, one-cylinder square-tiled surfaces). A stable graph $\Gamma_1(g,n)$, as on the left in Figure~\ref{fig:stable:graphs}, represents a non-separating simple closed curve on a surface of genus $g$ with $n$ punctures.
Considering stable graphs $\Gamma_1(g,n)$
we always assume
that $g\ge 1$ and that if $g=1$, then $n\ge 1$ without specifying it explicitly.
A stable graph $\Gamma_{g_1, n_1}^{g_2,n_2}$, as on the right in Figure~\ref{fig:stable:graphs}, represents a simple closed curve separating the surface into a subsurface of genus $g_1$ endowed with $n_1$ punctures and a complementary subsurface of genus $g_2$ endowed with $n_2$ punctures. Here $g=g_1+g_2$ and $n=n_1+n_2$.
Considering the graphs $\Gamma_{g_1, n_1}^{g_2,n_2}$
we will always assume
that $2g_i+n_i\ge 3$ for $i=1,2$, without specifying it explicitly.
We denote by $\mathcal{G}_{g,n}$ the set of all stable graphs corresponding to a surface of genus $g$ with $n$ punctures.
\medskip

\begin{figure}[htb]
\includegraphics{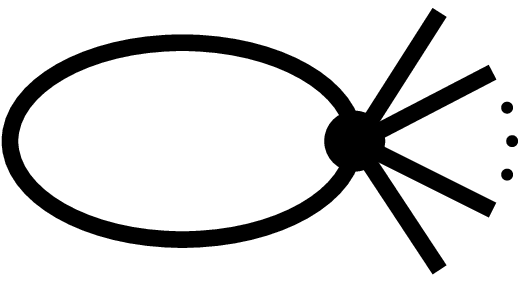}
\includegraphics{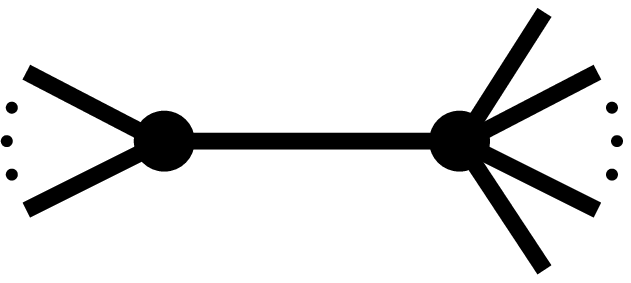}
\begin{picture}(-8,0)(0,0)
\put(-53,-20){$\bigggrB{50pt}n$}
\put(-110,-20){$g-1$}
\put(-110,-55){$\Gamma_1(g,n)$}
\put(-8,-20){$n_1\biggglB{30pt}$}
\put(38,-30){$g_1$}
\put(86,-30){$g_2$}
\put(118,-20){$\bigggrB{50pt}n_2$}
\put(55,-55){$\Gamma_{g_1, n_1}^{g_2,n_2}$}
\end{picture}
\vspace{55bp}
\caption{
\label{fig:stable:graphs}
Stable graphs representing a non-separating (on the left) and
separating (on the right) simple closed curves.}
\end{figure}

Let $g,n$ be non-negative integers with $2g+n \ge 3$. Let $b_1$,
\ldots, $b_n$ be formal variables. For a multi-index $\boldsymbol{d}
= (d_1, \ldots, d_n)$ we denote by $b^{2\boldsymbol{d}}$ the product
$b_1^{2d_1}\cdot\cdots\cdot b_n^{2d_n}$, by $|\boldsymbol{d}|$ the
sum $d_1 + \cdots + d_n$ and by $\boldsymbol{d}!$ the product $d_1!
\cdots d_n!$

Define a homogeneous polynomial $N_{g,n}(b_1,\dots,b_n)$ of degree
$6g-6 + 2n$ in the variables $b_1,\dots,b_n$ as
\begin{equation}
\label{eq:N:g:n}
N_{g,n}(b_1,\dots,b_n)=
\sum_{|\boldsymbol{d}|=3g-3+n}c_{\boldsymbol{d}} b^{2\boldsymbol{d}}\,,
\end{equation}
where
\begin{equation*}
c_{\boldsymbol{d}}=\frac{1}{2^{5g-6+2n}\, \boldsymbol{d}!}\,
\int_{\overline{\cM}_{g,n}} \psi_1^{d_1}\dots\psi_n^{d_n}.
\end{equation*}
Here $\psi_1$, \ldots, $\psi_n$ are the $\psi$-classes on the
Deligne--Mumford compactification $\overline{\cM}_{g,n}$.

We also use the following common notation for the intersection
numbers as above (often called \textit{Witten--Kontsevich
correlators}). Given an ordered partition $d_1+\dots+d_n=3g-3+n$ of
$3g - 3 + n$ into a sum of non-negative integers we define
\begin{equation*}
\langle \tau_{d_1} \dots \tau_{d_n}\rangle_g
:=\int_{\overline{\cM}_{g,n}} \psi_1^{d_1}\dots\psi_n^{d_n}\,.
\end{equation*}

Polynomials $N_{g,n}$ are implicitly present in Kontsevich's
proof~\cite{Kontsevich} of Witten's
conjecture~\cite{Witten}
and in the discretized model of the moduli space
of L.~Chekhov, see~\cite{Chekhov:93, Chekhov:97}.
They represent the top homogeneous
parts of Norbury's quasi-polynomials counting metric ribbon graphs
with edges of integer lengths~\cite{Norbury}. Up to a normalization constant $2^{2g-3+n}$,
the polynomial $N_{g,n}(b_1,\dots,b_n)$
coincides with the top homogeneous part of Mirzakhani's
volume polynomial $V_{g,n}(b_1,\dots,b_n)$ providing the
Weil--Petersson volume of the moduli space of bordered
Riemann
surfaces~\cite{Mirzakhani:simple:geodesics:and:volumes}.

Given a stable graph $\Gamma$ denote by $V(\Gamma)$ the set
of its vertices and by $E(\Gamma)$ the set of its edges. To
each stable graph $\Gamma\in\cG_{g,n}$ we associate the
following homogeneous polynomial $P_\Gamma$
of degree $6g-6+2n$. To
every edge $e\in E(\Gamma)$ we assign a formal variable
$b_e$. Given a vertex $v\in V(\Gamma)$ denote by $g_v$ the
integer number decorating $v$ and denote by $n_v$ the
valency of $v$, where the legs adjacent to $v$ are counted
towards the valency of $v$. Take a small neighborhood of
$v$ in $\Gamma$. We associate to each half-edge (``germ''
of edge) $e$ adjacent to $v$ the monomial $b_e$; we
associate $0$ to each leg. We denote by $\boldsymbol{b}_v$
the resulting collection of size $n_v$. If some edge $e$ is
a loop joining $v$ to itself, $b_e$ would be present in
$\boldsymbol{b}_v$ twice; if an edge $e$ joins $v$ to a
distinct vertex, $b_e$ would be present in
$\boldsymbol{b}_v$ once; all the other entries of
$\boldsymbol{b}_v$ correspond to legs; they are represented
by zeroes. To each vertex $v\in E(\Gamma)$ we associate the
polynomial $N_{g_v,n_v}(\boldsymbol{b}_v)$, where $N_{g,v}$
is defined in~\eqref{eq:N:g:n}. We associate to the stable
graph $\Gamma$ the polynomial obtained as the product
$\prod b_e$ over all edges $e\in E(\Gamma)$ multiplied by
the product $\prod N_{g_v,n_v}(\boldsymbol{b}_v)$ over all
$v\in V(\Gamma)$. We define $P_\Gamma$ as follows:
$$
P_\Gamma(\boldsymbol{b})
=
\frac{2^{6g-5+2n} \cdot (4g-4+n)!}{(6g-7+2n)!}\cdot
\frac{1}{2^{|V(\Gamma)|-1}} \cdot
\frac{1}{|\operatorname{Aut}(\Gamma)|}
\cdot
\prod_{e\in E(\Gamma)}b_e\cdot
\prod_{v\in V(\Gamma)}
N_{g_v,n_v}(\boldsymbol{b}_v)
\,.
$$

\begin{Example}
Using the rule described above we get the following polynomials
associated to the graphs
$\Gamma_1(g,n)$ and
$\Gamma_{g_1, n_1}^{g_2,n_2}$ from Figure~\ref{fig:stable:graphs}.
\begin{align}
\label{eq:P:Gamma1:g:n}
P_{\Gamma_1(g,n)}(b)
&=\frac{2^{6g-5+2n} \cdot (4g-4+n)!}{(6g-7+2n)!}
\cdot\frac{1}{2}
\cdot
b_1\cdot N_{g-1,n+2}(b_1,b_1,0,\dots,0)\,,
  \\
\label{eq:P:Gamma:g1:n1:g2:n2}
P_{\Gamma_{g_1,n_1}^{g_2,n_2}}(b_1)
&=\frac{2^{6g-5+2n} \cdot (4g-4+n)!}{(6g-7+2n)!}
\cdot\frac{1}{2|\operatorname{Aut}(\Gamma_{g_1,n_1}^{g_2,n_2})|}
\\ \notag & \times
b_1\cdot N_{g_1,n_1+1}(b_1,0,\dots,0)
\cdot N_{g_2,n_2+1}(b_1,0,\dots,0)
\,.
\end{align}
Here we used the fact that $|\operatorname{Aut}(\Gamma_1(g,n))|=2$
for any $g$ and $n$. We have
\begin{equation}
\label{eq:Aut}
|\operatorname{Aut}(\Gamma_{g_1,n_1}^{g_2,n_2})|=
\begin{cases} 2&\text{when we have both}\ g_1=g_2\text{ and }n_1=n_2\,;
\\
1&\text{otherwise}\,.
\end{cases}
\end{equation}
\end{Example}

We now define an operator $\cZ$ acting on polynomials.
It is defined on monomials as
\begin{equation}
\label{eq:cZ}
\cZ\ :\quad
\prod_{i=1}^{k} b_i^{m_i} \longmapsto
\prod_{i=1}^{k} \big(m_i!\cdot \zeta(m_i+1)\big)\,,
\end{equation}
and extended to arbitrary polynomials by linearity.
Everywhere in the current paper $\zeta$ is the Riemann zeta function
$$
\zeta(s) = \sum_{n \geq 1} \frac{1}{n^s}\,.
$$

We have proved in~\cite{DGZZ-volumes} the following statement.
\begin{NNTheorem}[{\cite[Theorem 1.5]{DGZZ-volumes}}]
   %
The Masur--Veech volume $\Vol\cQ_{g,n}$ of the moduli space of meromorphic quadratic differentials with $n$ simple poles on complex curves of genus $g$ has the following value:
\begin{equation}
\label{eq:Vol:Qgn}
\Vol \cQ_{g,n}
= \sum_{\Gamma \in \cG_{g,n}} \Vol(\Gamma)\,,
\end{equation}
where the contribution of an individual stable graph $\Gamma$ is equal to
\begin{equation}
\label{eq:volume:contribution:of:stable:graph}
\Vol(\Gamma)=\cZ(P_\Gamma)\,.
\end{equation}
\end{NNTheorem}

\subsection{Contribution of single-band square-tiled surfaces}
\label{ss:contribution:of:single:band}

In this section we analyze the contributions  $\Vol(\Gamma_1(g,n))$
and $\Vol(\Gamma_{g1,n1}^{g_2,n_2})$ of the graphs representing
square-tiled surfaces having a single horizontal cylinder to
the Masur--Veech volume $\Vol\cQ_{g,n}$ expressed as a sum in the
right-hand side of Equation~\eqref{eq:Vol:Qgn}. We start with
preparatory facts on Witten--Kontsevich correlators involved in the
polynomials $P_{\Gamma_1(g,n)}$ and $P_{\Gamma_{g1,n1}^{g_2,n_2}}$;
see~Equations~\eqref{eq:P:Gamma1:g:n}
and~\eqref{eq:P:Gamma:g1:n1:g2:n2} respectively.

\begin{Lemma}
\label{lem:string}
For $g\geq 1$, $n\geq 0$, $d_1\geq 0$, the intersection numbers satisfy the following equalities:
\begin{eqnarray}
\label{eq:1:cor}
 \langle\tau_0^n\tau_{3g+n-2}\rangle_{g} &=&\langle \tau_{3g-2}\rangle_{g} = \frac{1}{24^g\cdot g!}\,.
\\
 \langle\tau_0^{n+2}\tau_{n}\rangle_{0}&=& 1\,.
\\
\label{eq:tau3}
\qquad \langle \tau_0^n \tau_{d_1}\tau_{3g-1+n-d_1}\rangle_{g}&=&\sum_{i=\max(0,d_1-3g+1)}^{\min(d_1, n)} \binom{n}{i}\langle\tau_{d_1-i}\tau_{3g-1-d_1+i}\rangle_{g}\,.
\\
\label{eq:simplification:4}
\langle \tau_0^n\tau_{d_1}\tau_{n-1-d_1}\rangle_0 &=&\binom{n-1}{d_1}\,.
\end{eqnarray}
\end{Lemma}
\begin{proof}
Applying repeatedly the string equation
\[
\langle \tau_0\tau_{d_1}\dots \tau_{d_k}\rangle_g=\sum_{i=1}^{k}
\langle \tau_{d_1}\dots \tau_{d_i-1}\dots \tau_{d_k}\rangle_{g}
\]
we eliminate $\tau_0$ thus proving the left equality in~\eqref{eq:1:cor} and Equation~\eqref{eq:tau3}.
The right equality in~\eqref{eq:1:cor} is due to E.~Witten~\cite{Witten}. The remaining equalities
concern
genus $0$ correlators, for which we use the closed formula
$$
\langle \prod_{i=1}^n \tau_{d_i}\rangle_0=\frac{(n-3)!}{\prod_i d_i!}\,.
$$
also due to E.~Witten~\cite[p. 251, after Equation (2.26)]{Witten}.
\end{proof}

Values of 2-correlators $\langle\tau_{k}\tau_{3g-1-k}\rangle_{g}$ can
be obtained in a particularly efficient way through the following
recursive relations found in~\cite{Zograf:2-correlators}:
\begin{align}
\label{eq:tau:g:k:difference:1}
&(6j+1)\langle\tau_{3j}\tau_{3g-1-3j}\rangle_g
-(6j+1-6j)\langle\tau_{3j-1}\tau_{3g-3j}\rangle_g
=\frac{1}{24^g\cdot g!}\binom{g}{j}\left(1-\frac{2j}{g}\right)\,.
\\
\label{eq:tau:g:k:difference:2}
&(6j+3)\langle\tau_{3j+1}\tau_{3g-2-3j}\rangle_g
-(6j-1-6j)\langle\tau_{3j}\tau_{3g-1-3j}\rangle_g
=-2\cdot\frac{1}{24^g\cdot g!}\binom{g-1}{j}\,.
\\
\label{eq:tau:g:k:difference:3}
&(6j+5)\langle\tau_{3j+2}\tau_{3g-3-3j}\rangle_g
-(6j-3-6j)\langle\tau_{3j+1}\tau_{3g-2-3j}\rangle_g
=2\cdot\frac{1}{24^g\cdot g!}\binom{g-1}{j}\,.
\end{align}
Here we use the explicit formula~\eqref{eq:1:cor} for
$\langle\tau_{0}\tau_{3g-1}\rangle_g=\frac{1}{24^g\cdot g!}$ as the
base of the recursion. In genera $1$ and $2$ this gives
$\langle\tau_1\tau_1\rangle=\frac{1}{24}$ and
$\langle\tau_1\tau_4\rangle=\frac{1}{384}$,
$\langle\tau_2\tau_3\rangle=\frac{29}{5760}$.
The remaining correlators in $g=1,2$ are
obtained by the symmetry $\langle\tau_{3g-1-k}\tau_k\rangle_g
=\langle\tau_k\tau_{3g-1-k}\rangle_g$ for $k=0,\dots,3g-1$.



We denote by $c_1(\Gamma)=\Vol(\Gamma)$ the contribution of
square-tiled surfaces associated to a stable graph $\Gamma$ to the
Masur--Veech volume $\Vol\cQ_{g,n}$, see Equation~\eqref{eq:Vol:Qgn}.
Graphs $\Gamma_1(g,n)$ and $\Gamma_{g_1,n_1}^{g_2,n_2}$ as in
Figure~\ref{fig:stable:graphs} are the only stable graphs in
$\cG_{g,n}$ having a single \textit{edge} (which we distinguish from
$n$ \textit{legs}). In other words, these are the only graphs
representing $1$-cylinder square-tiled surfaces. Among all
$1$-cylinder square-tiled surfaces associated to these stable graphs,
we distinguish those for which the single horizontal cylinder
is composed from a single-band of squares and we can compute
separately the contribution $\cyl_1(\Gamma)$ of such square-tiled
surfaces to  the Masur--Veech volume $\Vol\cQ_{g,n}$.

\begin{Lemma}
\label{prop:cyl1}
The contributions of single-band square-tiled surfaces
to the volume of the principal strata $\cQ_{g,n}$ are
given by
\begin{align}
\label{eq:cyl:1:Gamma:init}
\cyl_1(\Gamma_1(g,n)) & =2^{g+1}\frac{(4g-4+n)!}{(3g-4+n)!}
\cdot
\sum_{d_1=0}^{3g-4+n}\binom{3g-4+n}{d_1}\langle \tau_0^n\tau_{d_1}\tau_{3g-4+n-d_1}\rangle_{g-1}
\\
\label{eq:cyl:1:Gamma:g1:g2:init}
\cyl_1(\Gamma_{g_1, n_1}^{g_2,n_2}) &
=\frac{2^{g+2}}{|\operatorname{Aut}(\Gamma_{g_1, n_1}^{g_2,n_2})|}
\cdot\frac{(4g-4+n)!}{(3g-4+n)!}\cdot\frac{1}{g!\cdot 24^g}
\cdot\binom{g}{g_1}\cdot\binom{3g-4+n}{ 3g_1-2+n_1},
\end{align}
where $g=g_1+g_2$ and $n=n_1+n_2$.

The total contribution
of single-band square-tiled surfaces
to the volume of $\cQ_{g,n}$ is given by
\begin{equation}
\label{cyl1Q}
\cyl_1(\cQ_{g,n})
=  \cyl_1(\Gamma_1(g,n))
+\frac{1}{2}\sum_{n_1=0}^{n}
\binom{n}{n_1}\sum_{g_1=0}^{g}
|\operatorname{Aut}(\Gamma_{g_1, n_1}^{g_2,n_2})|
\cdot\cyl_1(\Gamma_{g_1,n_1}^{g_2,n_2})\,.
\end{equation}
\end{Lemma}
\begin{proof}
The following relation generalizing~\eqref{eq:c:cyl:quadratic} is
valid for any stable graph $\Gamma$ with a single edge and for any
$g$ and $n$.
\begin{equation}
\label{eq:Vol:Gamma:c1:Gamma:cyl1:Gamma}
\Vol(\Gamma)=c_1(\Gamma)=\cyl_1(\Gamma)\cdot\zeta(6g-6+2n)\,.
\end{equation}
This relation is an immediate corollary of~\cite[Formula (1.14) and
Lemma 1.32]{DGZZ-volumes}. Thus, to prove the desired expressions, it
is sufficient to apply the relation $\Vol(\Gamma)=\cZ(P_\Gamma)$
given by~\eqref{eq:volume:contribution:of:stable:graph} to the two
stable graphs under consideration. The polynomials
$P_{\Gamma_1(g,n)}(b)$ and $P_{\Gamma_{g_1,n_1}^{g_2,n_2}}(b_1)$ are
given in Equation~\eqref{eq:P:Gamma1:g:n}
and~\eqref{eq:P:Gamma:g1:n1:g2:n2}, where, applying the general
definition~\eqref{eq:N:g:n} of the polynomials $N_{g,n}$ we obtain
\begin{align*}
N_{g-1, n+2}(b_1,b_2, 0, \dots, 0) &= \frac{1}{2^{5g-7+2n}}\sum_{\substack{d_1+d_2=3g-4+n\\ d_1\geq 0, \, d_2\geq 0}} \frac{\langle \tau_0^n\tau_{d_1}\tau_{d_2}\rangle_{g-1}}{d_1!d_2!}b_1^{2d_1}b_2^{2d_2}\,,
\\
N_{g_1, n_1+1}(b_1, 0, \dots, 0) &=  \frac{1}{2^{5g_1-4+2n_1}}\frac{\langle \tau_0^{n_1}\tau_{3g_1-2+n_1}\rangle_{g_1}}{(3g_1-2+n_1)!}b_1^{6g_1-4+2n_1}\,.
\end{align*}
Plugging the operator $\cZ$, defined by~\eqref{eq:cZ}, into the
formula~\eqref{eq:volume:contribution:of:stable:graph}   and
simplifying the results using
relations~\eqref{eq:1:cor}--\eqref{eq:simplification:4} from
Lemma~\ref{lem:string} we
obtain Equations~\eqref{eq:cyl:1:Gamma:init}
and~\eqref{eq:cyl:1:Gamma:g1:g2:init}.

Note that $\Gamma_{g_1, n_1}^{g_2,n_2}$ and $\Gamma_{g_2,
n_2}^{g_1,n_1}$ define the same stable graphs. Thus, the stable graph
$\Gamma_{g_1, n_1}^{g_2,n_2}$ is present in the sum~\eqref{cyl1Q}
exactly once if and only if both conditions $g_1=g_2$ and $n_1=n_2$
are satisfied. All other stable graphs of the form $\Gamma_{g_1,
n_1}^{g_2,n_2}$ are present in the sum~\eqref{cyl1Q} twice. Taking
into consideration Equation~\eqref{eq:Aut}, this
justifies~\eqref{cyl1Q}. Lemma~\ref{prop:cyl1} is proved.
\end{proof}


The combinatorial Proposition~\ref{prop:cyl1:simplified} below
simplifies expressions~\eqref{eq:cyl:1:Gamma:init}
and~\eqref{eq:cyl:1:Gamma:g1:g2:init} for $\cyl_1(\Gamma_1(g,n))$ and
$\cyl_1(\Gamma_{g_1, n_1}^{g_2,n_2})$ respectively.

\begin{Proposition}
\label{prop:cyl1:simplified}
Assume that $g\ge 1$, and that if $g=1$ then $n\ge 2$.
The contribution
to the Masur--Veech volume of the principal stratum $\cQ_{g,n}$ of meromorphic quadratic differentials
coming from single-band square-tiled surfaces corresponding
to the stable graph $\Gamma_1(g,n)$ has the following form:
\begin{align}
\label{cyl:1:Gamma:1:n}
\cyl_1(\Gamma_1(1,n)) & =4n\cdot \binom{2n-2}{n-1}\quad \text{for } g=1\,;
\\
\label{cyl:1:Gamma:1:new}
\cyl_1(\Gamma_1(g,n)) &
=2^{g+1}\binom{4g-4+n}{g}\cdot g!
\sum_{k=0}^{3g-4} \binom{3g-4+2n}{n+k} \langle\tau_k\tau_{3g-4-k}\rangle_{g-1}
\quad \text{for }g\ge 2\,.
\end{align}

The total contribution
to the Masur--Veech volume of the principal stratum $\cQ_{g,n}$ of meromorphic quadratic differentials
coming from single-band square-tiled surfaces corresponding
to all stable graphs $\Gamma_{g_1,n_1}^{g_2,n_2}$ has the following form:
   %
\begin{equation}
\label{eq:sum:cyl1:separating}
\frac{1}{2}\binom{n}{n_1}
\sum_{\substack{g_1+g_2=g\\n_1+n_2=n}}
|\operatorname{Aut}(\Gamma_{g_1, n_1}^{g_2,n_2})|
\cdot\cyl_1(\Gamma_{g_1,n_1}^{g_2,n_2})
=2^{g+1}
\binom{4g-4+n}{g}
\frac{1}{24^g}
\sum_{g_1=0}^{g}
\binom{g}{g_1}
\binom{3g-4+2n}{3g_1-2+n}\,.
\end{equation}

The total contribution of single-band square-tiled
surfaces to the volume of $\cQ_{g,n}$ is given by
\begin{align}
\cyl_1&(\cQ_{0,n})
= 2\binom{2n-4}{n-2}
\quad \text{for }n\ge 4\,;
\\
\cyl_1&(\cQ_{1,n})
=4n\cdot \binom{2n-2}{n-1}
+\frac{n}{3}\cdot \binom{2n-1}{n-2}
\quad \text{for }n\ge 2\,;
\\
\label{cyl1Q:simplified}
\cyl_1&(\cQ_{g,n})
=2^{g+1}\binom{4g-4+n}{g}
\cdot\Bigg(
g! \sum_{k=0}^{3g-4} \binom{3g-4+2n}{n+k}
\langle\tau_k\tau_{3g-4-k}\rangle_{g-1}
\\ \notag
&\qquad+\frac{1}{24^g}
\sum_{g_1=0}^{g}
\binom{g}{g_1}
\binom{3g-4+2n}{3g_1-2+n}
\Bigg)
\quad \text{for }g\ge 2\,.
\end{align}
\end{Proposition}
   %
   %
\begin{proof}
We develop formulas obtained in Lemma~\ref{prop:cyl1}.
Using the following combinatorial identity (see~\cite[(3.20)]{Gould}):
\begin{equation}
\label{eq:Gould:3:20}
\sum_{k=0}^n \binom{n}{k}\binom{x}{k+r}
=
\binom{n+x}{n+r}\,.
\end{equation}
and Equation \eqref{eq:simplification:4}, we get for $g=1$
$$
\sum_{d_1=0}^{3g-4+n} \binom{3g-4+n}{d_1}
\langle\tau_0^n\tau_{d_1}\tau_{3g-4+n-d_1}\rangle_{g-1}
=\sum_{d_1=0}^{n-1} \binom{n-1}{d_1}\binom{n-1}{d_1}
=\binom{2n-2}{n-1}\,,
$$
which simplifies the sum in~\eqref{eq:cyl:1:Gamma:init} in the case $g=1$.

When $g\ge 2$, using equation~\eqref{eq:tau3} and letting $k=d_1-i$, we get:
\begin{multline*}
\sum_{d_1=0}^{3g-4+n}\binom{3g-4+n}{d_1}
\langle \tau_0^n\tau_{d_1}\tau_{3g-4+n-d_1}\rangle_{g-1}
\\=
\sum_{d_1=0}^{3g-4+n}\sum_{i=\max(0,d_1-3g+4)}^{\min(d_1, n)} \binom{3g-4+n}{d_1}\binom{n}{i}
\langle\tau_{d_1-i}\tau_{3g-4-d_1+i}\rangle_{g-1}=
\\
=\sum_{k=0}^{3g-4}\langle\tau_k\tau_{3g-4-k}\rangle_{g-1}\sum_{i=0}^n \binom{n}{i}\binom{3g-4+n}{i+k}
=\sum_{k=0}^{3g-4}\langle\tau_k\tau_{3g-4-k}\rangle_{g-1}
\binom{2n+3g-4}{n+k}\,,
\end{multline*}
where we use identity~\eqref{eq:Gould:3:20} one more time to justify
the last equation.

Finally, we can simplify the second term in the sum~\eqref{cyl1Q}
from Lemma~\ref{prop:cyl1} simplifying
expression~\eqref{cyl:1:Gamma:1:new} for
$\cyl_1(\Gamma_{g_1,n_1}^{g_2,n_2})$ from this Proposition. We get
\begin{multline*}
\frac{1}{2}\sum_{n_1=0}^{n} \binom{n}{n_1}
\sum_{g_1=0}^{g}
|\operatorname{Aut}(\Gamma_{g_1, n_1}^{g_2,n_2})|
\cdot\cyl_1(\Gamma_{g_1,n_1}^{g_2,n_2})
=\\=
2^{g+1}\cdot\frac{(4g-4+n)!}{(3g-4+n)!}\cdot\frac{1}{g!\,24^g}
\sum_{n_1=0}^{n}
\sum_{g_1=0}^{g}
\binom{n}{n_1}\binom{g}{g_1}\binom{3g-4+n}{3g_1-2+n_1}\,.
\end{multline*}
Changing the order of summation and
applying identity~\eqref{eq:Gould:3:20}
we can simplify the latter sum as
$$
\sum_{g_1=0}^{g}
\binom{g}{g_1}
\sum_{n_1=0}^{n}
\binom{n}{n_1}\binom{3g-4+n}{3g_1-2+n_1}
=
\sum_{g_1=0}^{g}
\binom{g}{g_1}
\binom{3g-4+2n}{3g_1-2+n}\,,
$$
which justifies~\eqref{eq:sum:cyl1:separating}.
\end{proof}


Combining proposition~\ref{prop:cyl1:simplified} with recursive
formulas~\eqref{eq:tau:g:k:difference:1}--\eqref{eq:tau:g:k:difference:3}
for 2-correlators we get explicit expressions for
$cyl_{1}(\cQ_{g,n})$ in terms of $g$ and $n$. In the next two
sections we analyze $cyl_{1}(\cQ_{g,n})$ in two regimes: when $g$ is
fixed and $n\to\infty$ and in the regime when $n$ is fixed and
$g\to+\infty$.

\subsection{Asymptotic count for large values of $\nbigons$}
\label{ss:large:number:of:poles}

We now discuss asymptotics of the quantities $\Vol\cQ_{g,n}$,
$cyl_1(\mathcal Q_{g,n})$, $c_1(\cQ_{g,n})$, representing the
Masur--Veech volume, and the contributions to this volume coming from
single-band square-tiled surfaces, and from one-cylinder square-tiled
surfaces respectively. In this section we study the regime when the
genus $g$ is fixed while the number of poles $n$ tends to infinity.
This allows us to derive asymptotic of the quantities
$cyl_{1,1}(\cQ_{g,n})$ and $\prob_1(\cQ_{g,n})$ in the same regime
and, thus, prove Formula~\eqref{eq:Cgp:p:to:infty} from
Theorem~\ref{th:meander}.

We start with the following simple Lemma:

\begin{Lemma}
For any positive integer $a\ge 2$ and for any integers $b$
and $c$ one has the following asymptotics:
\begin{equation}
\label{eq:binomial}
\binom{an+b}{n+c}\sim
\frac{1}{\sqrt{2\pi n}}\cdot
\frac{a^{an+b+\frac{1}{2}}}{{(a-1)}^{(a-1)n+b+\frac{1}{2}}}
\quad\text{as }n\to+\infty
\,.
\end{equation}
\end{Lemma}
\begin{proof}
It is sufficient to prove the Lemma for particular case $c=0$
since then, given arbitrary $c$ we denote $n+c$ by $m$ and
apply the asymptotic formula for $\binom{am+(b-ac)}{m}$.

For $c=0$ we apply Stirling's formula to each of the three
factorials in $\binom{an+b}{n}=\frac{(an+b)!}{((a-1)n+b)!\cdot n!}$
and having simplified the resulting expression
we get~\eqref{eq:binomial}.
\end{proof}

\begin{Corollary}
\label{cor:cyl1:large:n}
For any fixed genus $g\geq 0$, we have the following asymptotics:
\begin{equation}
\label{eq:cyl1:large:n}
cyl_1(\mathcal Q_{g,n}) \sim c_1(\cQ_{g,n})
\sim  \frac{1}{\sqrt{\pi}}\cdot a_g\cdot n^{g-\frac{1}{2}}\cdot 4^n
\quad\text{as } n\to\infty\,,
\end{equation}
where $a_0=\cfrac{1}{8}$, $a_1=\cfrac{7}{6}$ and
\begin{equation}
\label{eq:a:g}
a_g=2^{2g-3}\cdot
\left(2^{2g}\sum_{k=0}^{3g-4}
\langle\tau_k\tau_{3g-4-k}\rangle_{g-1}
+\frac{1}{3^g g!}\right)
\text{ for }g\geq 2\,.
\end{equation}
\end{Corollary}
\begin{proof}
A computation for $g=0$ is, essentially, performed
in~\cite{DGZZ-meander}. Namely, by Formula~(1.5)
in~\cite{DGZZ-meander} one has
   %
$$
p_1(\cQ(1^{\nbigons-4},-1^\nbigons))=
\frac{\cyl_1(\cQ(1^{\nbigons-4},-1^\nbigons))}
{\Vol\cQ(1^{\nbigons-4},-1^\nbigons)}\,.
$$
Since $\cQ(1^{\nbigons-4},-1^\nbigons)$ is the unique stratum of top
dimension in the moduli space $\cQ_{0,\nbigons}$ we get equalities
$\cyl_1(\cQ(1^{\nbigons-4},-1^\nbigons))=\cyl_1(\cQ_{0,\nbigons})$
and $\Vol\cQ(1^{\nbigons-4},-1^\nbigons)=\Vol\cQ_{0,\nbigons}$.
Multiplying the asymptotic expression for
$p_1(\cQ(1^{\nbigons-4},-1^\nbigons))$, evaluated
in~\cite{DGZZ-meander} (see the expression just above Theorem~1.3),
by the exact value~\eqref{eq:Vol:Q:0:n} of $\Vol\cQ_{0,\nbigons}$,
obtained in~\cite{AEZ:genus:0}, we get the desired asymptotic
expression for $\cyl_1(\cQ_{0,\nbigons})$.

Assume that $g\geq 2$ (the computation for $g=1$ is similar, but simpler). We first compute the contribution coming from the stable graph $\Gamma_1(g,n)$, and show that for any fixed $g$ we have
\begin{equation}
\label{eq:Gamma1:large:n}
\Vol\Gamma_1(g,n)  \sim
\frac{2^{4g-3}}{\sqrt \pi}
\left(\sum_{k=0}^{3g-4}\langle\tau_k\tau_{3g-4-k}\rangle_{g-1}\right)
n^{g-\frac{1}{2}}\cdot 4^n
\mbox{ as } n \to\infty.
\end{equation}
In order to prove~\eqref{eq:Gamma1:large:n} we start by
applying~\eqref{eq:binomial}
to get the following asymptotics of the binomial coefficient
present in~\eqref{cyl:1:Gamma:1:new}:
\begin{equation}
\label{eq:asym_sum}
\binom{2n+3g-4}{n+k}
\sim \frac{2^{2n+3g-4}}{\sqrt{\pi n}}\quad\text{as } n
\to\infty\,.
\end{equation}
Note that for any fixed $g$, the asymptotic expression
of the binomial coefficient in~\eqref{eq:asym_sum}
does not depend on $k$ anymore for large values of $n$, and thus,
can be factored out of the sum in~\eqref{cyl:1:Gamma:1:new}.

For any fixed $g$ the ratio of factorials in the line
above~\eqref{cyl:1:Gamma:1:new} has the following asymptotics
for large values of $n$:
$$
\frac{(4g-4+n)!}{(3g-4+n)!}
=\underbrace{(n+3g-3)(n+3g-2)\cdots(n+4g-4)}_{g\text{ terms}}
\sim n^g \quad\mbox{as }n\to\infty\,.
$$
Recall that
$\Vol\Gamma_1(g,n)=\zeta(6g-6+2n)\cdot\cyl_1(\Gamma_1(g,n))$ and that
we have exponentially rapid convergence $\zeta(6g-6+2n)\rightarrow 1$
as $n\to\infty$. Combining the three asymptotic relations above we
conclude that formula~\eqref{cyl:1:Gamma:1:new} for
$\cyl_1(\Gamma_1(g,n))$ implies~\eqref{eq:Gamma1:large:n}.

Now we show that for any fixed $g$ the asymptotic volume
contribution coming from the remaining stable graphs has the
following form:
\begin{equation}
\label{eq:Gamma:g1n1:g2n2:large:n}
\frac{1}{2}\sum_{n_1=0}^{n} \binom{n}{n_1}
\sum_{g_1=0}^{g}
|\operatorname{Aut}(\Gamma_{g_1, n_1}^{g_2,n_2})|
\cdot\Vol\Gamma_{g_1,n_1}^{g_2,n_2}
\sim
\frac{2^{2g-3}}{\sqrt\pi\cdot 3^g\cdot g!}\cdot n^{g-\frac{1}{2}}\cdot 4^n
\quad\text{as } n \to\infty\,.
\end{equation}
In order to prove this, we start by applying~\eqref{eq:binomial} to
get the following asymptotics of the binomial coefficient present in
the second line of~\eqref{cyl1Q:simplified}:
$$
\binom{2n+3g-4}{n+3g_1-2}
\sim \frac{2^{2n+3g-4}}{\sqrt{\pi n}}\quad\text{as } n
\to\infty\,.
$$
We get the same expression as in~\eqref{eq:asym_sum}. This asymptotic
equivalence is uniform for any fixed $g$ and any $g_1$ in the range
$0\le g_1\le g$. Since it does not depend on $g_1$ for large values
of $n$ anymore, it  can be factored out of the sum
in~\eqref{cyl1Q:simplified}. The remaining sum can be now explicitly
computed:
$$
\sum_{g_1=0}^{g}\binom{g}{g_1}=2^g
\,.
$$
The rest of the computation is now
completely analogous to the case of $\Gamma_1(g,n)$
treated above.
\end{proof}

Having found large $n$ asymptotics of $\cyl_1(\cQ_{g,n})$ we recall
information on large number of poles asymptotics of $\Vol\cQ_{g,n}$.

An explicit formula for the Masur--Veech volume of any stratum of
meromorphic quadratic differentials with at most simple poles in
genus $0$ was conjectured by M.~Kontsevich and proven by
J.~Athreya--A.~Eskin--A.~Zorich in~\cite{AEZ:genus:0}.
In particular, one has
\begin{equation}
\label{eq:Vol:Q:0:n}
\Vol\cQ_{0,n}=4\left(\frac{\pi}{2}\right)^{n-3}
\quad\text{for } n=4,5,\dots\,.
\end{equation}

A simple closed formula for $\Vol\cQ_{1,n}$ was found
in~\cite[Corollary 1.5]{CMS:quad}:
\begin{equation}
\label{eq:Vol:Q:1:n}
\Vol\cQ_{1,n}=\pi^{2n}\cdot\frac{n!}{3(2n-1)!}
\big((2n-3)!!+(2n-2)!!\big)
\quad\text{for } n=2,3,\dots\,.
\end{equation}

An expression for $\Vol\cQ_{g,n}$ in terms of Hodge integrals was
recently discovered
D.~Chen--M.~M\"oller--A.~Sauvaget~\cite{CMS:quad}. Based on this
formula, M.~Kazarian in~\cite{Kazarian} and
D.~Yang--D.~Zagier--Y.~Zhang in~\cite{YZZ} independently proved
quadratic recursions for the volumes. These results combined with
Formula~\cite[(6)]{CMS:quad} allow to derive a close expression in
the style of~\eqref{eq:Vol:Q:1:n} for $\Vol\cQ_{g,n}$ for any small
value of $g$.

Thus, the asymptotics of the Masur--Veech volumes for a fixed genus
$g$ and large number of poles $n$ is now completely explicit. For
small values of $g$ Formula~\eqref{eq:Vol:Q:gn:large:n} below
(including the rational values of $\kappa_g$)
was predicted in~\cite[(5.12)]{ABC+}.

\begin{NNProposition}[Corollary~4 in~\cite{YZZ}]
For any fixed genus $g\geq 0$,
the following asymptotics is valid:
\begin{equation}
\label{eq:Vol:Q:gn:large:n}
\Vol\cQ_{g,n}
\sim \kappa_g n^{\frac{g}{2}}\left(\frac{\pi^2}{2}\right)^n
\quad\text{as }n\to \infty\,,
\end{equation}
where
\begin{equation}
\label{eq:kappa:g}
\kappa_g=
\frac{64\cdot \pi^{6g-\frac{11}{2}}}
{384^g\cdot \Gamma\left(\frac{5g-1}{2}\right)}
\cdot\tilde\kappa_g\,,
\end{equation}
$$
\tilde\kappa_0=-1\,,
\quad \tilde\kappa_1=2\,,
\quad \tilde\kappa_2=98\,,
\quad \tilde\kappa_3=19600\,,
$$
and where $\tilde\kappa_g$
is recursively defined by
$$
\tilde\kappa_g=50(g-1)^2\tilde\kappa_{g-1}
+\frac{1}{2}\sum_{h=2}^{g-2}\tilde\kappa_h\tilde\kappa_{g-h}
\quad\text{for }g\ge 4\,.
$$
\end{NNProposition}
Recall that for $g\in\mathbb{N}$ one has
$$
\Gamma\left(\frac{5g-1}{2}\right)
=\begin{cases}
\left(\cfrac{5g-3}{2}\right)!
&\text{for odd }g\\
\sqrt{\pi}\cdot\cfrac{(5g-3)!!}{2^\frac{5g-2}{2}}
&\text{for even }g
\end{cases}\,.
$$

\begin{Corollary}
\label{cor:cyl11_quad_poles}
For any fixed genus $g\geq 0$, the following asymptotic formulas are valid:
\begin{align}
\label{eq:cyl11:large:p}
cyl_{1,1}(\cQ_{g,n})& \sim  \frac{1}{\pi}\cdot
\frac{a_g^2}{\kappa_g}\cdot
n^{\frac{3g}{2}-1}\left(\frac{32}{\pi^2}\right)^n
&\quad\text{as } n\to \infty\,;
\\
\relcontgn & \sim
\frac{1}{\sqrt{\pi}}\cdot
\frac{a_g}{\kappa_g}\cdot
n^{\frac{g-1}{2}}\left(\frac{8}{\pi^2}\right)^n
&\quad\text{as } n\to \infty\,,
\end{align}
where $a_g$ and $\kappa_g$ are given by Equations~\eqref{eq:a:g}
and~\eqref{eq:kappa:g} respectively.
\end{Corollary}
\begin{proof}
Recall that $\cyl_{1,1}(\cQ_{g,\nbigons})
=\tfrac{\cyl_1(\cQ_{g,\nbigons})^2}{\Vol\cQ_{g,\nbigons}}$
and that $p_1(\cQ_{g,\nbigons})
=\frac{cyl_1(\cQ_{g,\nbigons})}{\Vol\cQ_{g,\nbigons}}$.
Plugging the asymptotic expressions~\eqref{eq:cyl1:large:n}
for $cyl_1(\mathcal Q_{g,n})$ and~\eqref{eq:Vol:Q:gn:large:n}
for $\Vol\cQ_{g,n}$ we get the desired relations.
\end{proof}

\begin{proof}[Proof of Formula~\eqref{eq:Cgp:p:to:infty}
from Theorem~\ref{th:meander}]
Using Stirling's formula for the factorials in the
denominator of the right-hand side expression
in~\eqref{eq:constant:Cgp}, we get
$$
(4g-4+\nbigons)!\,\nbigons!\,(12g-12+4\nbigons)
\sim \nbigons^{4g-4} (n!)^2 4n
\sim 8\pi\cdot n^{4g-2}\left(\frac{n}{e}\right)^{2n}
\ \text{as }n\to+\infty\,.
$$
Plugging the expression~\eqref{eq:cyl11:large:p} for the large $\nbigons$
asymptotics of $\cyl_{1,1}(\cQ_{g,n})$ into
Formula~\eqref{eq:constant:Cgp} for $C_{g,n}$ and simplifying the
fraction we obtain the desired asymptotics~\eqref{eq:Cgp:p:to:infty}.
\end{proof}

Having obtained asymptotic expressions~\eqref{eq:Gamma1:large:n} and~\eqref{eq:Gamma:g1n1:g2n2:large:n}
for large $n$ volume contributions of the stable graphs
as in Figure~\ref{fig:stable:graphs} we are ready to prove
Theorem~\ref{th:ratio:sep:nonsep:large:p}.

\begin{proof}[Proof of Theorem~\ref{th:ratio:sep:nonsep:large:p}]
By~\cite[Theorem 1.22]{DGZZ-volumes} for any $g\ge 1$ one has
\begin{equation}
\label{eq:cyl1:sep:cyl1:nonsep}
\frac{c_{g,n,\mathrm{sep}}}{c_{g,n,\mathrm{nonsep}}}
=\cfrac{\frac{1}{2}\sum_{n_1=0}^{n}
\binom{n}{n_1}\sum_{g_1=0}^{g}
|\operatorname{Aut}(\Gamma_{g_1, n_1}^{g_2, n_2})|
\Vol\Gamma_{g_1,n_1}^{g_2,n_2}}
{\Vol(\Gamma_1(g,n))}\,.
\end{equation}
Using the asymptotic expression~\eqref{eq:Gamma:g1n1:g2n2:large:n}
for the numerator of the ratio in the right-hand side of the above
equation and the asymptotic expression~\eqref{eq:Gamma1:large:n}
for the denominator of this ratio
we get~\eqref{eq:sep:over:nonsep:large:p}
in the general case $g\ge 2$.

It remains to consider the particular case $g=1$.
Using expression~\eqref{cyl:1:Gamma:1:n} for the
denominator of~\eqref{eq:cyl1:sep:cyl1:nonsep}
and evaluating the expression~\eqref{eq:sum:cyl1:separating}
for the numerator of~\eqref{eq:cyl1:sep:cyl1:nonsep}
in the particular case $g=1$ we get
\begin{multline*}
\frac{c_{1,n,\mathrm{sep}}}{c_{1,n,\mathrm{nonsep}}}
=
\frac{4n\cdot\frac{1}{24}\cdot 2\cdot\binom{2n-1}{n-2}}
{4n\cdot\binom{2n-2}{n-1}}
=\frac{1}{12}\cdot
\frac{(2n-1)!\,(n-1)!\,(n-1)!}{(2n-2)!\,(n-2)!\,(n+1)!}
\\=
\frac{1}{12}\cdot
\frac{(2n-1)(n-1)}{n(n+1)}
\sim\frac{1}{6}
\quad\text{as }n\to+\infty\,.
\end{multline*}
which completes the proof of~\eqref{eq:sep:over:nonsep:g:1:large:p}.
\end{proof}

\subsection{Large genus asymptotic count of meanders}
\label{ss:large:genus}

In this section we study asymptotics of the quantities
$c_{g,n,\mathrm{sep}}$, $c_{g,n,\mathrm{nonsep}}$,
$\cyl_1(\cQ_{g,n})$, and $\Vol\cQ_{g,n}$ in the regime, when $n$ is fixed and $g\to+\infty$.

Recall that $\zeta(m)\to 1$ as $m\to\infty$. Thus, Equation~\eqref{eq:Vol:Gamma:c1:Gamma:cyl1:Gamma}
implies that
\begin{equation}
\label{eq:cyl1:sim:Vol:Gamma}
\cyl_1(\Gamma_1(g,n))\sim\Vol(\Gamma_1(g,n))
\text{ and }
\cyl_1(\Gamma_{g_1,n_1}^{g_2,n_2})\sim\Vol(\Gamma_{g_1,n_1}^{g_2,n_2})
\text{ as } g\to+\infty
\end{equation}
uniformly in $n,g_1,g_2,n_1,n_2$.

\begin{Proposition}\label{prop:cyl1:sep:nonsep}
For any fixed $n\ge 0$ the following asymptotic relations are valid
\begin{align}
\label{eq:cgn:nonsep:g:to:infty}
\Vol(\Gamma_1(g,n))
&\sim
\sqrt{\frac{2}{3\pi g}}
\cdot\left(\frac{16}{3}\right)^n
\cdot\left(\frac{8}{3}\right)^{4g-4}
\text{as } g\to+\infty\,;
\\
\label{eq:cgn:sep:g:to:infty}
\frac{1}{2}\sum_{n_1=0}^{n}
\binom{n}{n_1}\sum_{g_1=0}^{g}\Vol(\Gamma_{g_1,n_1}^{g_2,n_2})
&\sim
\frac{2}{3\pi g}
\cdot\frac{1}{4^g}
\cdot\left(\frac{16}{3}\right)^n
\cdot\left(\frac{8}{3}\right)^{4g-4}
\text{as } g\to+\infty\,.
\end{align}
\end{Proposition}
In the particular case $n=0$,
relations~\eqref{eq:cgn:nonsep:g:to:infty}
and~\eqref{eq:cgn:sep:g:to:infty} were proved
in~\cite[(4.5)]{DGZZ-volumes}    and in~\cite[(4.15)]{DGZZ-volumes}
respectively. In the general case,
relation~\eqref{eq:cgn:nonsep:g:to:infty} was proved by A.~Aggarwal
in~\cite[(8.9)]{Agg:vol:quad}. For the sake of completeness we
present below a short proof of both relations.

\begin{proof}
Combining~\cite[Proposition~4.1]{DGZZ-volumes}  and~\cite[Formula~(4.2)]{DGZZ-volumes} we obtain the following large genus asymptotics for $2$-correlators
$$
\langle\tau_k\tau_{3g-1-k}\rangle_{g}
=\frac{1}{24^g\cdot g!}
\cdot\frac{(6g-1)!!}{(2k+1)!!(6g-1-2k)!!}
\left(1+O\left(\frac{1}{g}\right)\right)
\quad\text{as }g\to+\infty\,,
$$
where the error term $O\left(\frac{1}{g}\right)$
is uniform in $0\le k\le 3g-1$.
Passing from double factorials to factorials we rewrite the
above expression as
$$
\langle\tau_k\tau_{3g-1-k}\rangle_{g}
\sim
\frac{1}{24^g\cdot g!}
\cdot\frac{1}{6g}
\cdot\cfrac{\binom{6g}{2k+1}}{\binom{3g-1}{k}}\,,
$$
where the asymptotic equivalence is uniform in $0\le k\le 3g-1$. Plugging the resulting asymptotics for $2$-correlators into the sum
involved in Formula~\eqref{cyl:1:Gamma:1:new} we get
$$
\sum_{k=0}^{3g-1} \binom{3g-1+2n}{n+k}
\langle\tau_k\tau_{3g-1-k}\rangle_{g}
\sim
\frac{1}{24^g\cdot g!}
\cdot\frac{1}{6g}
\sum_{k=0}^{3g-1}
\cfrac{\binom{3g-1+2n}{n+k}\binom{6g}{2k+1}}{\binom{3g-1}{k}}\,.
$$
Applying asymptotic Formula~\eqref{eq:binomial:sum:asymptotics:general}
to the above sum we get
$$
\sum_{k=0}^{3g-1}
\cfrac{\binom{3g-1+2n}{n+k}\binom{6g}{2k+1}}{\binom{3g-1}{k}}
\sim 2^{6g+2n-1}
\quad\text{as }g\to+\infty\,.
$$
Applying~\eqref{eq:cyl1:sim:Vol:Gamma}, plugging the asymptotic
expression above into Formula~\eqref{cyl:1:Gamma:1:new} and replacing
the binomial $\binom{4g+n}{g}$ by the equivalent asymptotic expression
given by~\eqref{eq:binomial} we get
\begin{multline*}
\Vol(\Gamma_1(g+1,n)) =\cyl_1(\Gamma_1(g+1,n))
=2^{g+2}\frac{(4g+n)!}{(3g-1+n)!}
\sum_{k=0}^{3g-1} \binom{3g-1+2n}{n+k}
\langle\tau_k\tau_{3g-1-k}\rangle_g
\\
\sim
2^{g+2}\cdot(3g+n)\cdot\binom{4g+n}{g}
\cdot\frac{1}{24^g}
\cdot\frac{1}{6g}
\cdot 2^{6g+2n-1}
\sim
2^{g+1}
\cdot
\left(
\left(\frac{4}{3}\right)^{4g+n}
\cdot 3^{g-\tfrac{1}{2}}
\cdot\sqrt{\frac{2}{\pi g}}\right)
\cdot\frac{1}{3^g\cdot 2^{3g}}\cdot
\cdot 2^{6g+2n-1}
\\ =
\sqrt{\frac{2}{3\pi g}}
\cdot\left(\frac{8}{3}\right)^{4g}
\cdot\left(\frac{16}{3}\right)^n
\text{ as } g\to+\infty\,.
\end{multline*}
Adjusting the above expression
to genus $g$ instead of $g+1$
we complete the proof of~\eqref{eq:cgn:nonsep:g:to:infty}.

In order to prove~\eqref{eq:cgn:sep:g:to:infty}
we apply~\eqref{eq:cyl1:sim:Vol:Gamma}
and then use~\eqref{eq:sum:cyl1:separating}:
\begin{multline*}
\frac{1}{2}\sum_{n_1=0}^{n}
\binom{n}{n_1}\sum_{g_1=0}^{g}
|\operatorname{Aut}(\Gamma_{g_1, n_1}^{g_2,n_2})|
\cdot\Vol(\Gamma_{g_1,n_1}^{g_2,n_2})
=
\frac{1}{2}\sum_{n_1=0}^{n}
\binom{n}{n_1}\sum_{g_1=0}^{g}
|\operatorname{Aut}(\Gamma_{g_1, n_1}^{g_2,n_2})|
\cdot\cyl_1(\Gamma_{g_1,n_1}^{g_2,n_2})
\\=
2^{g+1}\cdot\binom{4g-4+n}{g}\cdot\frac{1}{24^g}
\sum_{g_1=0}^{g}
\binom{g}{g_1}
\binom{3g-4+2n}{3g_1-2+n}\,.
\end{multline*}
Applying Formula~\eqref{eq:binomial} to the binomial
coefficient $\binom{3g-4+2n}{3g_1-2+n}$ and asymptotic
equivalence~\eqref{eq:binomial:sum:asymptotics:general} to the sum of
binomial coefficients we get the following asymptotics for the above
expression:
\begin{multline*}
\frac{1}{2}\sum_{n_1=0}^{n}
\binom{n}{n_1}\sum_{g_1=0}^{g}
|\operatorname{Aut}(\Gamma_{g_1, n_1}^{g_2,n_2})|
\cdot\Vol(\Gamma_{g_1,n_1}^{g_2,n_2})
\sim
2^{g+1}\cdot
\frac{2^{8g-8+2n+\tfrac{1}{2}}}{3^{3g-4+n+\tfrac{1}{2}}}
\cdot\frac{1}{\sqrt{\pi g}}
\cdot\frac{1}{2^{3g}\cdot 3^g}
\cdot
\frac{\sqrt{2}}{\sqrt{\pi\cdot 4\cdot 3\cdot g}}
\cdot 2^{4g+2n-4}
\\=
\frac{1}{\pi g}\cdot\frac{2^{10g+4n-11}}{3^{4g+n-3}}
=
\frac{2}{3\pi g}
\cdot\frac{1}{4^g}
\cdot\left(\frac{8}{3}\right)^{4g-4}
\cdot\left(\frac{16}{3}\right)^n
\text{ as } g\to+\infty\,,
\end{multline*}
which completes the proof of~\eqref{eq:cgn:sep:g:to:infty}.
\end{proof}

Now everything is ready to prove Theorem~\ref{th:ratio:sep:nonsep:large:g}.

\begin{proof}[Proof of Theorem~\ref{th:ratio:sep:nonsep:large:g}]
Replacing the numerator and the denominator
of the fraction in the right-hand side
of~\eqref{eq:cyl1:sep:cyl1:nonsep},
with respectively~\eqref{eq:cgn:sep:g:to:infty}
and~\eqref{eq:cgn:nonsep:g:to:infty}
we obtain the desired asymptotics~\eqref{eq:sep:over:nonsep:large:g}.
\end{proof}

We proceed now with a recollection of necessary facts concerning the large genus asymptotics of the
Masur--Veech volume $\Vol\cQ_{g,n}$ for a fixed value of the parameter $n$.
The large genus asymptotic formula for the Masur--Veech volume
$\Vol\cQ_g$ of the moduli space of holomorphic quadratic differentials
was conjectured in~\cite{DGZZ-volumes}. A more ambitious conjecture
on the uniform large genus asymptotic formula for all strata
of meromorphic quadratic differentials was stated in~\cite{ADGZZ}. This general conjecture is still wide open.
However, the particular case of the principal strata,
or equivalently the large genus asymptotics of $\Vol\cQ_{g,n}$
for any fixed $n$ was spectacularly proved by A.~Aggarwal
in~\cite{Agg:vol:quad}. As part of the proof, he also computed the asymptotics of one-cylinder contribution to $\Vol\cQ_{g,n}$.

\begin{NNProposition}[Thm 1.7. \cite{Agg:vol:quad}]
For any fixed $n\geq 0$, the following asymptotics hold:
\begin{eqnarray}
\label{eq:Vol:Qgn:as:n:to:infty}
\Vol\cQ_{g,n} & \sim &
\frac{4}{\pi}\cdot\left(\frac{16}{3}\right)^n\cdot\left(\frac{8}{3}\right)^{4g-4}
\text{ as } g\to+\infty\,,
\\
\label{eq:cyl1gn:as:g:to:infty}
cyl_1(\cQ_{g,n}) & \sim &
\sqrt{\frac{2}{3\pi g}}
\cdot\left(\frac{16}{3}\right)^n\cdot\left(\frac{8}{3}\right)^{4g-4}
\text{ as } g\to+\infty\,.
\end{eqnarray}
\end{NNProposition}

The result of~\cite{Agg:vol:quad} is, actually, much stronger:
he proved that the asymptotics~\eqref{eq:Vol:Qgn:as:n:to:infty} holds uniformly for all $n$ such that $20n\leq \log(g)$.

In the particular case $n=0$,
relation~\eqref{eq:cyl1gn:as:g:to:infty} was proved
in~\cite{DGZZ-volumes} as a combination of~\cite[(4.5)]{DGZZ-volumes}
and~\cite[(4.15)]{DGZZ-volumes}. In the general case,
relation~\eqref{eq:cgn:nonsep:g:to:infty} was first proved by
A.~Aggarwal as a combination of~\cite[(8.9)]{Agg:vol:quad}
and~\cite[Proposition~1]{Agg:vol:quad}
or~\cite[Lemma~9.2]{Agg:vol:quad}.
Relation~\eqref{eq:cyl1gn:as:g:to:infty} can be also obtained as an
immediate corollary of~Proposition~\ref{prop:cyl1:sep:nonsep}.

\begin{Corollary}
\label{cor:cyl11_quad_genus}
For any fixed number of poles $n\geq 0$, the following asymptotics
relations are valid:
\begin{align}
\label{eq:cyl11:large:g}
cyl_{1,1}(\cQ_{g,n})
& \sim
\frac{1}{6g}\left(\frac{8}{3}\right)^{4g-4}\cdot
\left(\frac{16}{3}\right)^n
&\quad\text{as }g\to \infty\,;
\\
\label{eq:p1:large:g}
p_1(\cQ_{g,n}) & \sim
\frac{\sqrt{6\pi}}{12}\cdot \frac{1}{\sqrt{g}}
&\quad\text{as }g\to \infty\,.
\end{align}
\end{Corollary}
\begin{proof}
It is sufficient to plug the asymptotic
expressions~\eqref{eq:Vol:Qgn:as:n:to:infty}
and~\eqref{eq:cyl1gn:as:g:to:infty} into
definitions~\eqref{eq:cyl:11:definition}
and~\eqref{eq:p1:definition}.
\end{proof}

Note that~\eqref{eq:p1:large:g} coincides with~\eqref{eq:p1:n:to:infty:promis}, so
the proof of Theorem~\ref{th:arc_system} is completed.

\begin{proof}[Proof of Formula~\eqref{eq:Cgp:g:to:infty}
from Theorem~\ref{th:meander}]
Using Stirling's formula for the factorial $(4g-4+\nbigons)!$ in the
denominator of the right-hand side expression
in~\eqref{eq:constant:Cgp}, we get
$$
(4g-4+\nbigons)!\,\nbigons!\,(12g-12+4\nbigons)
\sim n! (4g)^{\nbigons-4}
\sqrt{8\pi g}\left(\frac{4g}{e}\right)^{4g}\cdot\, 12g
\ \text{ as }g\to+\infty\,.
$$
Plugging the expression~\eqref{eq:cyl11:large:g} for the large genus
asymptotics of $\cyl_{1,1}(\cQ_{g,n})$ into
Formula~\eqref{eq:constant:Cgp} for $C_{g,n}$ and simplifying the
fraction we obtain the desired asymptotics~\eqref{eq:Cgp:g:to:infty}.
\end{proof}

\subsection{Large genus asymptotic count of oriented meanders}
\label{ss:large:genus:Abelian}

The large genus volume asymptotics for Abelian differentials was
conjectured in~\cite{Eskin:Zorich}. The conjecture was proved by
D.~Chen--M.~M\"oller--D.~Zagier in~\cite{CMZ} for the principal
stratum $\cH(1^{2g-2})$, by A.~Sauvaget in~\cite{Sauvaget:minimal}
for the minimal stratum $\cH(2g-2)$ and finally by
A.~Aggarwal~\cite{Aggarwal:Volumes} for all strata. A.~Sauvaget
computed in~\cite{Sauvaget:asymptotic:expansion} the next terms in
the asymptotic expansion of the volumes.
D.~Chen--M.~M\"oller--A.~Sauvaget--D.~Zagier interpreted the volumes
in terms of intersection numbers in~\cite[Theorem 1.1]{CMSZ} and
showed that these volumes satisfy certain recursion relation,
see~\cite[Theorem 3.1]{CMSZ}. They have found an alternative proof of
the conjecture~\cite{Eskin:Zorich} on large genus volume asymptotics
for all strata, and even for all connected components of all strata
of Abelian differentials.

The contribution of the one-cylinder square-tiled surfaces
to the Masur--Veech volume of any connected component of any stratum
of Abelian differentials
is evaluated in our paper~\cite{DGZZ-Yoccoz}. The contribution of the square-tiled surfaces with a single cylinder of height $1$ is deduced easily by dividing by $\zeta(d)$, where $d$ is the dimension of the stratum, see~\eqref{eq:c:cyl:Abelian}
or~\cite[Remark~4.29]{DGZZ-meander} for more details.
We now recall the relevant results concerning the large genus
asymptotics of the quantities $\Vol\cH(1^{2g-2})=\Vol\cH_g$ and
$\cyl_1(\cH(1^{2g-2}))=\cyl_1(\cH_g)$ for the principal stratum and
deduce from them the asymptotics of $\cyl_{1,1}(\cH_g)$,
$p_1(\cH_g)$, and of $C_g^+$.

\begin{NNProposition}[{\cite[Theorem 19.3]{CMZ}}]
The following asymptotics holds:
\begin{equation}
\label{eq:asympt:vol}
\Vol\cH_g  =  \frac{1}{4^{g-2}}\left(1-\frac{\pi^2}{24 g}+O\left(\frac{1}{g^2}\right)\right)
\quad\text{as }g\to\infty\,.
\end{equation}
\end{NNProposition}

\begin{NNProposition}[{\cite[Corollaries~2.6~and~2.12]{DGZZ-Yoccoz}}]
The following relation hold:
   %
\begin{equation}
\label{eq:asympt:cyl1}
cyl_1 (\cH_g)  = \frac{1}{(2g-1)\cdot 2^{2g-3}}
\\= \frac{1}{g\cdot 4^{g-1}}\left(1+\frac{1}{2g}+O\left(\frac{1}{g^2}\right)\right)
\quad\text{as }g\to\infty\,.
\end{equation}
\end{NNProposition}

\begin{Corollary}
\label{cor:vol:ab}
The following asymptotics holds
\begin{align}
\label{eq:asympt:P1}
p_1(\cH_g) &=  \frac{1}{4g}\left(1+\frac{12+\pi^2}{24g}+O\left(\frac{1}{g^2}\right)\right)
\text{ as } g\to+\infty\,,
\\
\label{eq:asympt:cyl11}
cyl_{1,1}(\cH_g) &  =   \frac{1}{g^2\cdot 4^g}\left(1 +\frac{24+\pi^2}{24g}+O\left(\frac{1}{g^2}\right)\right)
\text{ as } g\to+\infty\,,
\\
\label{eq:asympt:Cg+}
C_g^+& =  \frac{1}{4\sqrt{\pi}}
\cdot\frac{1}{g^{\frac{3}{2}}}
\left(\frac{e}{4g}\right)^{2g}
\left(1+\frac{29+\pi^2}{24 g} + O\left(\frac{1}{g^2}\right)\right)
\text{ as } g\to+\infty\,.
\end{align}
\end{Corollary}
\begin{proof}
We have $p_1(\cH_g)=\frac{cyl_1(\cH_g)}{\Vol\cH_g}$
by~\eqref{eq:p1:Hg:def}.
We have $cyl_{1,1}(\cH(g))=\frac{cyl_{1}(\cH_g)^2}{\Vol\cH_g}$
by~\eqref{eq:cyl:11:definition}. Applying~\eqref{eq:asympt:vol}
and~\eqref{eq:asympt:cyl1} we get~\eqref{eq:asympt:P1} and~\eqref{eq:asympt:cyl11}.

Finally, $C_g^+= \frac{cyl_{1,1}(\cH_g)}{(2g-2)!(8g-6)}$
by~\eqref{eq:Cg:plus}.
For the asymptotic expansion~\eqref{eq:asympt:Cg+}, we use the asymptotic formula for the factorial:
\[n! =\sqrt{2\pi n}\left(\frac{n}{e}\right)^n \left(1+\frac{1}{12n}+O\left(\frac{1}{n^2}\right)\right) \mbox{ as }n\to\infty,
\]
so we get
\begin{multline*}
\frac{1}{(2g-2)!}
=\frac{2g(2g-1)}{(2g)!}
=(2g)^2\left(1-\frac{1}{2g}\right)\frac{1}{(2g)!}
\\=
\frac{1}{2\sqrt{\pi g}}\cdot\frac{e^{2g}}{(2g)^{2g-2}}
\left(1-\frac{13}{24g}
+O\left(\frac{1}{g^2}\right)\right)
\quad\text{as }g\to\infty\,.
\end{multline*}
Multiplying the expression above
by $\frac{1}{8g-6}$
and by~\eqref{eq:asympt:cyl11}
we get the desired relation~\eqref{eq:asympt:Cg+}.
\end{proof}
Note that the results of A.~Sauvaget~\cite{Sauvaget:asymptotic:expansion} allow to compute the asymptotic expansion of
$\Vol\cH_g$ of any order: for any integer $r\geq 1$,
he defines by an explicit recursion real coefficients $(c_s)_{s=0..r}$, such that
\[
\Vol\cH_g=2^{2g-2}\sum_{s=0}^{r} \frac{c_s}{g^s} + O\left(\frac{1}{g^{r+1}}\right)
\quad\text{as }g\to+\infty\,.
\]
Thus, using the close expression~\eqref{eq:asympt:cyl1}
for the quantity $cyl_1(\cH_g)$ which we have
evaluated in~\cite[Corollaries~2.6~and~2.12]{DGZZ-Yoccoz},
one can extend the asymptotic expansions~\eqref{eq:asympt:P1}--\eqref{eq:asympt:Cg+}
for $p_1(\cH_g)$, $cyl_{1,1}$ and $C_g^+$ up to any order
in $\frac{1}{g}$.

\appendix

\section{Meanders and arc systems of special combinatorial types}
\label{app:combi}

As before, having a transverse pair of multicurves on a surface $S$,
denote by $\cG$ the associated embedded graph obtained as the union
of multicurves. Boundary components of the complement $S\setminus\cG$
might have only even number of sides. The boundary components with
two sides are call \textit{bigons}, see Definition~\ref{def:bigons}.

All the techniques used in this paper apply to count of meanders and
arc systems in the following more restrictive setting. Fix a finite
subset $F$ of $\N^*$ and fix a map $\mu$ from $F$ to $\N^*$. We
introduce the following notation:

$$
|\mu|=\sum_{j\in F} j\cdot \mu(j)\,,\qquad
\ell(\mu)=\sum_{j\in F} \mu(j)\,,\qquad
\ell_{\mathit{odd}}(\mu)=\sum_{\substack{j\in F\\j\ \text{is odd}}} \mu(j)\,.
$$

\begin{Definition}
\label{def:type}
We say that a pair of transverse multicurves has \textit{type} $\mu$
if for every $j\in F$ the complement $S\setminus\cG$ has exactly
$\mu(j)$ boundary components with $2j+4$ sides and for any $j\in
\N^*\setminus F$ there are no boundary components with $2j+4$ sides.
   %
\end{Definition}

Note that a type $\mu$ defined above does not impose
restrictions neither on a number of quadrangular boundary components,
nor on a number of bigons.
   %

The moduli space of meromorphic quadratic differentials on a surface
of genus $g$ with exactly $\nbigons$ simple poles is naturally
stratified by strata $\cQ(j^{\mu(j)}, -1^\nbigons)$ (also denoted
$\cQ(\mu,(-1)^\nbigons)$ for brevity) of quadratic differentials with
prescribed orders of zeroes ($\mu(j)$ zeroes of order $j$ for
$j=1,2,\dots$) and $\nbigons$ simple poles (see e.g.
\cite{Zorich:flat:surfaces} for references), where $|\mu|=
4g-4+\nbigons$.

We call the following two collections of data
\textit{exceptional}:
\begin{equation}
\label{eq:exceptional}
\big\{g=2, n=0,\quad F=\{3,1\},\quad  \mu(3)=\mu(1)=1\big\}
\quad \text{ and } \quad
\big\{g=2, n=0,\quad  F=\{4\},\quad  \mu(4)=1\big\}\,.
\end{equation}
All other collections $\{g,n,F,\mu\}$ as above satisfying both
conditions $g+2n\ge 4$ and $|\mu|= 4g-4+\nbigons$ are called
\textit{non-exceptional}.

The strata $\cQ(3,1)$ and $\cQ(4)$
corresponding to exceptional collections $\{g,n,F,\mu\}$
are empty while the strata corresponding to non-exceptional collections
are not, see~\cite{Masur:Smillie}.

Similarly, the moduli space of Abelian differentials on a surface of
genus $g$ is naturally stratified by strata $\cH(j^{\mu(j)})$ (also
denoted $\cH(\mu)$ for brevity) of Abelian differentials with
prescribed orders of zeroes ($\mu(j)$ zeroes of order $j$ for
$j=1,2,\dots$) where $|\mu|= 2g-2$. Recall that Abelian differentials
do not exist in genus zero; the only stratum in genus $g=1$ is
$\cH(0)$; for $g\ge 2$ and any $\mu$ satisfying $|\mu|= 2g-2$ the
stratum $\cH(\mu)$ is not empty.

For any pair of nonnegative integers $g$ and $n$ satisfying $g+2n\ge
4$ there exist a nonorientable transverse pair of multicurves of type
$\mu$ with $\nbigons$ bigons on a surface of genus $g$ if and only if
there exists a nonnegative integer $g'\le g$ such that
$\{g',n,F,\mu\}$ is non-exceptional. The pair is filling if and only
if $g'=g$. Pairs of transverse multicurves on a surface of genus
$g=2$ with no bigons and such that $F=\{3,1\}$ do not exist. Pairs of
transverse multicurves on a surface of genus $g=2$ with no bigons
such that $F=\{4\}$ and $\mu(4)=1$ do exist, but they are necessarily
orientable. All these properties remain valid when both (or one of
the) multicurves are simple closed curves.

The correspondence between pairs of multicurves and square-tiled
surfaces holds when fixing the type, as detailed below. To lighten
the presentation we focus from now on filling pairs of multicurves.
The arguments of the proof of Theorem~\ref{th:meander} can be easily
adapted to this more restrictive setting to show that the
contribution of non filling pairs is negligible.

Recall that simple poles of a meromorphic quadratic differential $q$
associated to a square-tiled surface $\cG^\ast$  correspond to bigons
of $\cG$ and zeroes of order $\degofz\in\N$ of $q$ correspond to
$(2\degofz+4)$-gons.
Proposition~\ref{prop:pairs:multicurves:square:tiled:surfaces}
translates in this new setting as follows.
\begin{Proposition}
\label{prop:pairs:multicurves:square:tiled:surfaces:comb}
For any non-exceptional data $\{g,n,F,\mu\}$, filling transverse
connected pairs of multicurves of type $\mu$ with exactly $\nbigons$
bigons on a surface of genus $g$ are in a natural one-to-one
correspondence with square-tiled surfaces of genus $g$ (with
non-labeled conical points) in the stratum $\cQ(\mu, (-1)^\nbigons)$.
The square tiling is given by the dual graph $\mathcal{G}^\ast$ of
the graph $\mathcal{G}$ formed by the union of the two multicurves.

Considering only filling transverse connected pairs of simple closed
curves of type $\mu$ we get a bijection with the subset of
square-tiled surfaces (with non-labeled conical points) in $\cQ(\mu,
(-1)^\nbigons)$ having a single horizontal and a single vertical band
of squares.
\end{Proposition}

Equation~\eqref{eq:VolQ:N:d}, Theorems~\ref{th:c1}, \ref{th:c1+},
\ref{th:c11}, \ref{th:c11+} hold when replacing
\begin{center}
\begin{tabular}{rcl}
$\cQ_{g,\nbigons}$ & by & $\cQ(\mu, (-1)^\nbigons)$\\
$d=\dim_\C\cQ_{g,\nbigons}=6g-6+2\nbigons$ & by  & $d=\dim_\C\cQ(\mu,
(-1)^\nbigons)=2g-2+\ell(\mu)+\nbigons$\\
$\cH_g$ & by & $\cH(\mu)$\\
$d=\dim_\C\cH_g=4g-3$ & by & $d=\dim_\C\cH(\mu)=2g-1+\ell(\mu)$.
\end{tabular}
\end{center}
Using step by step the same arguments, we get the following results
for the count of meanders and arc systems in this setting.

\begin{Theorem}
\label{th:meander:combi}
For any non-exceptional data $\{g,n,F,\mu\}$,
the number $\MeandNumber_{g,\nbigons, \mu}(N)$ of (filling)
meanders of type $\mu$ and genus $g$ with at most
$2N$ crossings and $\nbigons$ bigons satisfies
the following asymptotics as $N\to \infty$:
\[
\MeandNumber_{g,\nbigons, \mu}(N)
=C_{g,\nbigons,\mu} N^{d} + o(N^{d})\,,
\]
where  $d=2g-2+\ell(\mu)+\nbigons$ and
\[
C_{g,\nbigons, \mu}
=\frac{cyl_{1,1}(\cQ(\mu, -1^\nbigons))}
{\nbigons!\prod\mu(j)!\cdot 2d}\,.
\]
   %
  %
\end{Theorem}

\begin{Remark}
We could chose a setting in which we impose to a filling transverse
pair of simple closed curves on a surface of genus $g$ have
\textit{exactly} $n$ bigons, but \textit{at least} $\mu'(1)$
hexagonal faces, \textit{at least} $\mu'(2)$ octagonal faces, etc,
assuming that $|\mu'|< 4g-4+n$. This lets certain freedom for the
number and types of the remaining nontrivial faces --- the ones with
at lest 6 edges. The dimensional consideration as in
Section~\ref{sec:flat} imply that the predominant configuration for a
meander satisfying these constraints is the one having the maximal
possible number $4g-4+n - |\mu'|$ of faces of degree 6 allowed by the
Euler characteristic constraints, and the rest of the faces (except
for the $\nbigons$ bigons) of degree $4$. In this setting meanders
with other collections of nontrivial faces are negligible in the
asymptotic count.
\end{Remark}

In the oriented case, similar results hold for any $g\geq 2$, any
finite subset $F$ of $\N^\ast$, and any map $\mu:F\to\N^*$ such that
$|\mu|=2g-2$. We denote by $\MeandNumber^+_{g,\mu}(N)$ the number of
oriented meanders of genus $g$ with at most $N$ crossings, and
exactly $\mu(j)$ faces of valency $4(j+1)$ for $j\in F$ and with no
faces of valency $4(j+1)$ for $j\in \N^*\setminus F$. We call such
meanders \textit{oriented meanders of type $\mu$}.

\begin{Theorem}\label{th:oriented_meander:combi}
For any genus $g\geq 2$, any finite subset $F$ of $\N^*$,
and any map $\mu:F\to\N^*$ such that $|\mu|=2g-2$, the number
$\MeandNumber_{g, \mu}^+(N)$ of oriented meanders of type $\mu$ and
genus $g$ with at most $N$ crossings satisfies the following
asymptotics:
\[
\MeandNumber_{g,\mu}^+(N)=C_{g, \mu}^+ N^{d} + o(N^{d})
\quad\text{as } N\to \infty\,,
\]
where $d=2g-1+\ell(\mu)$, and
\[C_{g, \mu}^+=\frac{cyl_{1,1}(\cH(\mu))}{\prod\mu(i)!\cdot 2d}\] is a
rational multiple of $\pi^{-2g}$.
%
\end{Theorem}

Similarly, for any non-exceptional data $\{g,n,F,\mu\}$ we define arc
systems of type $\mu$. Fix the upper bound $N$ for the number of
arcs. Denote by $\operatorname{AS}_{g,n,\mu}(N)$ the number of all
possible couples (balanced arc system of type $\mu$ of genus $g$ with
$n$ bigons with $\narcs\le N$ arcs; identification) considered up to
a natural equivalence. Denote by $\operatorname{MAS}_{g,n,\mu}(N)$
the number of those couples, which give rise to a meander. Define
\[
\prob_{g,\nbigons,\mu}(N)
=\frac{\operatorname{MAS}_{g,n}(N)}{\operatorname{AS}_{g,n,\mu}(N)}\,.
\]
Recall that by convention $g$ denotes the genus of the surface
obtained \textit{after} identification of the two boundary
components.

%
%
%
%
%

\begin{Theorem}
\label{th:arc_system:combi}
The proportion of arc systems of type $\mu$ and genus $g$ with
$\nbigons$ bigons giving rise to meanders among all such arc systems
satisfies
\[ \lim_{N\to\infty} P_{g,\nbigons, \mu}(N)
=p_1(\cQ( \mu, -1^\nbigons))\,,
\]
where
\[p_1(\mu, (-1)^\nbigons)
=\frac{cyl_{1}(\cQ(\mu,(-1)^\nbigons))}{\Vol_1\cQ(\mu, (-1)^\nbigons)}\,.
\]
\end{Theorem}

\begin{Remark}
It was proved in\cite{DGZZ-meander} that
$$
\cyl_{1,1}\left(\cQ(\mu,(-1)^\nbigons)\right)
=
\frac{\big(\cyl_{1}\left(\cQ(\mu,(-1)^\nbigons)\right)\big)^2}
{\Vol \cQ(\mu,(-1)^\nbigons)}\,,
$$
where $\cyl_{1}\left(\cQ(\mu,(-1)^\nbigons)\right)\in\mathbb{Q}$.
Recall that for every quadratic differential $q$ on a Riemann surface
$S$ of genus $g$ there exists a canonical double cover $p:\hat S\to
S$ such that $p^\ast q$ is a square of globally defined holomorphic
$1$-form. Denote by $\hat g$ the genus of the covering surface $\hat
S$ and by $g_{\mathit{eff}}$ the \textit{effective genus} defined as
$\hat g-g$. One has $2g_{\mathit{eff}} = 2+|\mu| +
\ell_{\mathit{odd}}(\mu)$, see, say~\cite{Eskin:Kontsevich:Zorich}.

Conjecturally, $\Vol_1\cQ(\mu, (-1)^\nbigons)$ is a rational multiple
of $\pi^{2g_{\mathit{eff}}}$. This conjecture is valid for all strata
in genus $0$ as follows from a close formula for the Masur--Veech
volume of any stratum in genus zero obtained in~\cite{AEZ:genus:0}.
It is also valid for all strata of dimension at most 12: their
volumes were explicitly computed in~\cite{Goujard:volumes} using the
approach of~\cite{Eskin:Okounkov:pillowcase}. Finally, in the case
when zeros have only odd degrees, the conjecture was recently proved
in~\cite{Koziarz:Nguyen} and also follows from results of D.~Chen,
M.~M\"oller and A.~Sauvaget. However, in the presence of zeroes of
even degrees, the Conjecture is still open.
\end{Remark}

Similarly we define oriented arc systems of type $\mu$
and define the proportion $P_{g,\mu}^+$ as previously.

\begin{Theorem}
\label{th:oriented_arc_system:combi}
The proportion of oriented arc systems of type $\mu$ and genus $g$
giving rise to oriented meanders among all such arc systems satisfies
\[
\lim_{N\to\infty} P_{g, \mu}^+(N)=p_1(\cH(\mu))\,,
\]
where \[p_1(\cH(\mu))=\frac{cyl_{1}(\cH(\mu))}{\Vol_1\cH(\mu)}\] is a
rational multiple of $\pi^{-2g}$.
Furthermore, we have
\[
\lim_{g\to+\infty} p_1(\cH(\mu))\cdot (2g+\ell(\mu))= 1
\]
uniformly for all partitions $\mu$ such that $|\mu|=2g-2$.

The latter limit is proved in~\cite[Corollary 2.12]{DGZZ-Yoccoz},
which uses the uniform large genus asymptotic formula for
$\Vol_1\cH(\mu)$ conjectured in~\cite{Eskin:Zorich} and proved
independently in~\cite{Aggarwal:Volumes} and~\cite{CMSZ}.
\end{Theorem}


\section{Sum of a rational function over binomial coefficients}
\label{app:sum:of:binomials}

\subsection{Sum of ratios of binomial coefficients}

We are interested in the asymptotics as $n \to \infty$ of sums of the
form
\[
\sum_{k} \frac{\binom{an+b}{ck+d}}{\binom{sn + t}{uk + v}}\,,
\]
where the sum is over the integers $k$ such that $0 \leq ck +d \leq
an+ b$ and $0 \leq uk+v \leq sn + t$. Here $(a, b, c, d, s, t, u, v)$
are integral parameters with $a,c,s,u$ positive. In
Section~\ref{ss:general:case} we consider asymptotics of more general
sums of similar kind. The asymptotics takes a particularly nice form
when $a/c = s/u$.

\begin{Theorem}
\label{thm:binomial:sum:asymptotics}
Let $(a, b, c, d, s, t, u, v)$ be integers such that
$a, c, s, u$ are strictly positive integers, $a/c = s/u$, and $a > s$.
Let $\alpha = a/c = s/u$. Assume that $\alpha> 1$.
The following asymptotics holds
\begin{equation}
\label{eq:binomial:sum:asymptotics}
\sum_{k} \frac{\binom{an+b}{ck+d}}{\binom{sn + t}{uk + v}}
\sim
2^{(a-s)n + (b-t)} \cdot \alpha \cdot \sqrt{\frac{\pi}{2 (a - s)} \frac{s}{a}} \cdot \sqrt{n}
\qquad\text{as }n \to \infty\,,
\end{equation}
where the summation is taken over all integers $k$ satisfying all of
the following conditions: $0 \leq ck +d \leq an+ b$ and $0 \leq uk+v
\leq sn + t$.
\end{Theorem}

The above result is a direct corollary of the Local Limit
Theorem~\ref{thm:local:limit} combined with
Theorem~\ref{thm:tail:estimate} providing tail estimates. We state
the Local Limit Theorem for the interpolation of binomial
coefficients in terms of the $\Gamma$-function. Namely, for real
numbers $0 \le x \le y$ we let
\[
\binom{y}{x} := \frac{\Gamma(y+1)}{\Gamma(x+1) \cdot \Gamma(y-x+1)}\,.
\]

\begin{Theorem}
\label{thm:local:limit}
Consider $(a,b,c,d,s,t,u,v)$ and $\alpha$ satisfying assumptions of
Theorem~\ref{thm:binomial:sum:asymptotics}.
Let $0 < \delta < 1/4$ and let $k(x,n) = \frac{\alpha n}{2} \left(1 +
\frac{x}{\sqrt{n}}\right)$, where
$- n^{1/4-\delta}\le x\le n^{1/4-\delta}$.

For any $\delta$ as above
the following asymptotic equivalence holds
\begin{equation}
\label{eq:local:limit}
\frac{\binom{an+b}{ck(x,n)+d}}{\binom{sn+t}{uk(x,n)+v}}
\sim
2^{(a-s)n+(b-t)} \exp\left(-\frac{(a-s)x^2}{2}\right) \sqrt{\frac{s}{a}}
\qquad\text{as }n \to \infty
\end{equation}
uniformly in $x \in [- n^{1/4-\delta}, n^{1/4-\delta}]$.

Furthermore, for any $0 < \epsilon < 1$ we have
\begin{equation}
\label{eq:Hoeffding:like}
\frac{\binom{an+b}{ck(x,n)+d}}{\binom{sn+t}{uk(x,n)+v}}
\le
2^{(a-s)n+(b-t)} \exp\left(-\frac{(a-s)x^2}{2}\right) \sqrt{\frac{s}{a}} + O\left( \frac{1}{n} \right).
\end{equation}
uniformly in $x \in [-(1-\epsilon) \sqrt{n}, (1-\epsilon) \sqrt{n}]$.
\end{Theorem}

\begin{Remark}
Note that the right-hand sides of the asymptotic
expressions~\eqref{eq:binomial:sum:asymptotics},
\eqref{eq:local:limit} and \eqref{eq:Hoeffding:like} do not depend on
$d$ and $v$.
\end{Remark}

Theorem~\ref{thm:local:limit} provides upper bounds for the
expression in the left-hand side of~\eqref{eq:Hoeffding:like} only
outside of the tails $x \in \big[-\sqrt{n}, -(1-\epsilon) \sqrt{n}\big)
\cup\big((1-\epsilon) \sqrt{n}, \sqrt{n}\big]$,
for which certain approximations in the proof
of Theorem~\ref{thm:local:limit} become invalid. However we can apply
a softer large deviation estimates for the tails to show that the
tail contribution to the sum~\eqref{eq:binomial:sum:asymptotics} is
of exponentially lower order.
\begin{Theorem}
\label{thm:tail:estimate}
Let $H(p) = - p \log(p) - (1-p) \log(1-p)$, where $0<p<1$.
For any $\epsilon \in (0,1]$ we have
\begin{equation}
\label{eq:tail:estimate}
\sum_{|k - \frac{\alpha n}{2} | \ge \frac{\alpha}{2} (1 - \epsilon) n}
\frac{\binom{an + b}{ck + d}}{\binom{sn + t}{uk + v}}
= O(\epsilon\cdot a\cdot n\cdot\exp(a \cdot n \cdot H(1 - \epsilon/2)))
\quad\text{as }n\to+\infty\,.
\end{equation}
\end{Theorem}
The function $H(p)$ can be extended by continuity to $p=0$ and $p=1$
as $H(0)=H(1)=0$.
\begin{Remark}
Theorem~\ref{thm:tail:estimate} provides just a rough large deviation
upper bound. We expect that a finer estimate with the exponent
$n (a - e) H(1 - \epsilon/2)$ in the right hand side should be valid.
However, since such a refinement is not needed for our purpose, we
did not seek for an optimal bound.
\end{Remark}

The proof of Theorem~\ref{thm:local:limit} follows closely the proof
of the de Moivre-Laplace theorem for binomial coefficients. Studying
a ratio of binomials rather than a single binomial does not introduce
much difficulty.

Before proceeding to the proofs of Theorems~\ref{thm:local:limit} and
Theorem~\ref{thm:tail:estimate}, we recall in
Lemmas~\ref{lem:binomial:asymptotics}, \ref{lem:H} and
\ref{lem:large:deviations} well-known facts about binomial
coefficients.

\begin{Lemma}
\label{lem:binomial:asymptotics}
We have
\begin{equation}
\label{eq:binomial:asymptotics}
\binom{n}{p n}
=
e^{n H(p)} \cdot \frac{1}{\sqrt{2 \ \pi \ p \ (1-p)\ n}}
\left(1 + O\left( \frac{1}{n} \right) \right)
\quad\text{as }n \to \infty
\end{equation}
uniformly in $p$ restricted to compact subsets of $(0, 1)$.
\end{Lemma}

\begin{Remark}
Actually, expression~\eqref{eq:binomial:asymptotics} can be
strengthened to the following explicit bounds
\begin{equation}
\label{eq:binomial:upper:bound}
1 - \frac{1 - p\ (1 - p)}{12 \cdot p \cdot (1-p)} \cdot \frac{1}{n}
<
\frac{\binom{n}{p n}}{e^{n H(p)} \cdot \frac{1}{\sqrt{2 \ \pi \ p \ (1-p)\ n}}}
< 1
\end{equation}
valid for any $p \in (0,1)$. We limit ourself to a weaker version
sufficient for our needs.
\end{Remark}

\begin{proof}[Proof of Lemma~\ref{lem:binomial:asymptotics}]
Since all of $n$, $pn$ and $(1-p)n$ tend to $+\infty$ we could apply
Stirling's asymptotic formula to the three factorials
(or, more generally, to the three $\Gamma$-functions) in
$\binom{n}{pn} = \frac{n!}{(pn)! ((1-p)n)!}$. We get
\[
\binom{n}{np} =
\frac{\left( \frac{n}{e} \right)^n \sqrt{2 \cdot \pi \cdot n}}%
     {\left( \frac{pn}{e} \right)^{pn} \sqrt{2 \cdot \pi \cdot p \cdot n}
      \ \left(\frac{(1-p)n}{e} \right)^{(1-p)n} \sqrt{2 \cdot \pi \cdot (1-p) \cdot n}
    }
\ \left( 1 + O\left( \frac{1}{n} \right) \right).
\]
The right hand side in the above equation simplifies
as~\eqref{eq:binomial:asymptotics}.
\end{proof}

\begin{Lemma}
\label{lem:H}
Let $H(p) = - p \log p - (1-p) \log (1-p)$ be as in Theorem~\ref{thm:tail:estimate}.
Then, for any $x \in (-1,1)$ we have
\[
H\left( \frac{1}{2} + \frac{x}{2}\right)
=
\log(2)
- \sum_{n \ge 1} \frac{x^{2n}}{2n(2n-1)}.
\]
In particular, for any $x \in (-1,1)$ we have
\[
H\left( \frac{1}{2} + \frac{x}{2}\right)
\leq \log(2) - \frac{x^2}{2}
\]
and for small $x$
\[
H\left( \frac{1}{2} + \frac{x}{2}\right)
= \log(2) - \frac{x^2}{2} + O(x^4).
\]
\end{Lemma}

\begin{proof}
The function $H$ is analytic on $[0,1]$. Centered at $p=1/2$,
the radius of convergence is $1/2$ and we get the formula.
\end{proof}

Finally, in the proof of Theorem~\ref{thm:tail:estimate} we will use
the following version of the large deviations for binomials.

\begin{Lemma}[{\cite[Theorem 1]{ArratiaGordon}}]
\label{lem:large:deviations}
For any $s \in ]1/2,1[$ we have
\[
\sum_{k \ge s n} \binom{n}{k} \le e^{n H(s)}.
\]
\end{Lemma}
(In notation of~\cite{ArratiaGordon} one has to let $p=1/2$ and to
multiply both sides of an analogous relation by $2^n$.) The
paper\cite{ArratiaGordon} also~provides a finer asymptotic
equivalence.

We are now ready to proceed to the proofs of
Theorem~\ref{thm:tail:estimate} and Theorem~\ref{thm:local:limit}.

\begin{proof}[Proof of Theorem~\ref{thm:tail:estimate}]
The denominator $\binom{sn + t}{uk + v}$
in the left-hand side of~\eqref{eq:tail:estimate}
is at least $1$, so each
term in the sum in the left hand side of~\eqref{eq:tail:estimate} is
bounded from above by the numerator $\binom{an+b}{ck+d}$.

We now bound the numerators using Lemma~\ref{lem:large:deviations}.
For any $\epsilon', \epsilon$ satisfying $0<\epsilon<\epsilon'<1$ we have
\begin{multline*}
\sum_{k  \ge \alpha (1 - \epsilon/2) n}
\binom{an + b}{ck + d}
=
\sum_{ck  \ge (1 - \epsilon/2) n}
\binom{an + b}{ck + d}
\lesssim
\sum_{(c k + d) \ge (1 - \epsilon'/2) (a n + b)} \binom{an+b}{ck+d}
\\
\le
e^{(an+b) H(1 - \epsilon'/2)}
=
O(\exp(a \cdot n \cdot H(1 - \epsilon'/2)))
\ \text{ as }g\to\infty.
\end{multline*}
The case $k \le (\epsilon/2)\cdot \alpha\cdot n$
is symmetric.
\end{proof}

\begin{proof}[Proof of Theorem~\ref{thm:local:limit}]
Recall that by assumption $\alpha = a/c = s/u \ge 1$. Our parameter $k$
satisfies $0 \le c k + d \le an + b$ and $0 \le uk + v \le sn + t$.
In other words
\[
\max \left( -\frac{d}{c}, - \frac{v}{u} \right)
\le k \le
\alpha n + \min \left( \frac{b-d}{c}, \frac{t-v}{u} \right).
\]
We let $k = \frac{\alpha n}{2} \left(1 + \frac{x}{\sqrt{n}} \right)$
with $x \in (-\sqrt{n}, \sqrt{n})$.

By Lemma~\ref{lem:binomial:asymptotics} we have
\begin{equation}
\label{eq:first:approx}
\frac{\binom{an+b}{ck+d}}{\binom{sn+t}{uk+v}}
\sim
\exp \left((an+b) H(f_1(x,n)) - (sn+t) H(f_2(x,n))\right)
\cdot \sqrt{R(x,n)}\,,
\end{equation}
where $R(x,n) = \frac{f_2(x,n) \cdot (1-f_2(x,n)) \cdot
(sn + t)}{f_1(x,n) \cdot (1 - f_1(x,n)) \cdot (an + b)}$,
$f_1(x,n) = \frac{ck + d}{an + b}$ and
$f_2(x,n) = \frac{uk + v}{sn + t}$.

Let us first analyze $f_1(x,n)$ and $f_2(x,n)$.
Equalities $\alpha = a/c = s/u$ allows to rewrite these functions as
\[
f_1(x,n) = \frac{\frac{1}{2} + \frac{x}{2\sqrt{n}} + \frac{d}{a} \frac{1}{n}}{1 + \frac{b}{a} \frac{1}{n}}
\quad \text{and} \quad
f_2(x,n) = \frac{\frac{1}{2} + \frac{x}{2\sqrt{n}} + \frac{v}{s} \frac{1}{n}}{1 + \frac{t}{s} \frac{1}{n}}.
\]
Since $x/\sqrt{n} = O(1)$, uniformly in $x \in [-\sqrt{n}, \sqrt{n}]$ we have that $f_1$ and $f_2$
are of the same order
\[
f_1(x,n) = \frac{1}{2} + \frac{x}{2 \sqrt{n}} + O\left( \frac{1}{n} \right)
\quad \text{and} \quad
f_2(x,n) = \frac{1}{2} + \frac{x}{2 \sqrt{n}} + O\left( \frac{1}{n} \right).
\]
By Lemma~\ref{lem:H}, we get
   %
$$
(an+b) H(f_1(x,n)) - (sn+t) H(f_2(x,n))
= ((a-s)n + (b-t)) \left( H\left(\frac{1}{2} + \frac{x}{2 \sqrt{n}} \right) + O \left(\frac{1}{n^2} \right) \right)
$$
uniformly for $x$ inside $[-(1-\epsilon) \sqrt{n}, (1-\epsilon) \sqrt{n}]$.
In particular,
   %
$$
(an+b) H(f_1(x,n)) - (sn+t) H(f_2(x,n))
\le ((a-s)n + (b-t)) \left( \log(2) - \frac{x^2}{2 n} \right) + O \left(\frac{1}{n} \right)
$$
   %
uniformly for $x$ inside $[-(1-\epsilon) \sqrt{n}, (1-\epsilon) \sqrt{n}]$.

We now analyze the behavior close to $x=0$. Let us fix $0 < \delta <
1/2$ small and consider $k = \frac{\alpha n}{2} \left(1 +
\frac{x}{\sqrt{n}} \right)$ with $x \in [- n^{1/4-\delta},
n^{1/4-\delta}]$. All the $O(\cdot)$-estimates below are independent of
$x$ in this interval but do depend on the choice of $\delta$. We
obtain
   %
$$
(an+b) \ H(f_1(x,n))
= (an+b) \left(\log(2) - \frac{x^2}{2 \cdot n} + O\left( \frac{x^4}{n^2} \right)\right) \\
= (a n + b) \log(2) - \frac{a x^2}{2} + O\left(n^{-4\delta}\right).
$$
   %
The same analysis holds for $f_2(x,n)$ and we obtain
   %
$$
(an + b) H(f_1(x,n)) - (sn + t) H(f_2(x,n))
=
((a - s) n + (b - t)) \log(2) - \frac{(a - s) x^2}{2} + O\left(n^{-4\delta}\right).
$$

For the remaining term we have
\[
R(x,n) = \frac{s}{a} + \frac{at - bs}{a^2} \frac{1}{n} + O\left(\frac{x}{n^{3/2}}\right)
= \frac{s}{a} + O\left( \frac{1}{n} \right).
\]
uniformly for $x\in[-\sqrt{n}, \sqrt{n}]$
and we obtain~\eqref{eq:local:limit} and~\eqref{eq:Hoeffding:like}.
\end{proof}

We are ready to deduce Theorem~\ref{thm:binomial:sum:asymptotics}
from Theorems~\ref{thm:tail:estimate} and~\ref{thm:local:limit}.
\begin{proof}[Proof of Theorem~\ref{thm:binomial:sum:asymptotics}]
Theorems~\ref{thm:local:limit} and~\ref{thm:tail:estimate}
imply that the main contribution
to the sum~\eqref{eq:binomial:sum:asymptotics}
comes from the
terms with $k = k(x) = \frac{\alpha n}{2} \left(1 +
\frac{x}{\sqrt{n}}\right)$, where
$x \in [- n^{1/4-\delta}, n^{1/4-\delta}]$. Here one can choose any
$\delta$ satisfying $0<\delta<1/4$.

As $k$ varies in the integers, the values
of $x$ takes successive values spaced by $\frac{2}{\alpha \sqrt{n}}$,
and hence
\[
\sum_k \frac{\binom{an+b}{ck+d}}{\binom{sn+t}{ek+f}}
\sim
2^{(a-s)n+(b-t)} \cdot \sqrt{\frac{s}{a}} \cdot \frac{\alpha \cdot \sqrt{n}}{2} \cdot \int_{-\infty}^{+\infty} \exp\left( -\frac{(a-s)x^2}{2} \right) dx.
\]
The value of the integral is $\sqrt{\pi \cdot \frac{2}{a-s}}$ and we find~\eqref{eq:binomial:sum:asymptotics}.
\end{proof}

\subsection{General case}
\label{ss:general:case}

Theorem~\ref{thm:binomial:sum:asymptotics}
admits the following straightforward generalization.

\begin{Theorem}
\label{thm:binomial:sum:asymptotics:general}
Let $(a_i, b_i, c_i, d_i)$, where $i=1,\dots,l$,
and $(s_j, t_j, u_j, v_j)$, where $j=1,\dots,m$, be collections of integers.
Denote
$a=a_1+\dots + a_l$, $b=b_1+\dots+b_l$,
$s=s_1+\dots + s_m$, $t=t_1+\dots+t_m$,
$A=a_1\cdots a_l$, $S=s_1\cdots s_m$.

Suppose that all $a_i, c_i, s_j, u_j$ are strictly positive. Suppose
that $\frac{a_i}{c_i}=\frac{s_j}{u_j}=\alpha\ge 1$ for $i=1,\dots, l$
and for $j=1,\dots, m$. Suppose that $a > s$. Then
    %
\begin{equation}
\label{eq:binomial:sum:asymptotics:general}
\sum_{k}
\frac{
\binom{a_1 n+b_1}{c_1 k+d_1}
\cdots
\binom{a_l n+b_l}{c_l k+d_l}
}{
\binom{s_1 n+t_1}{u_1 k+v_1}
\cdots
\binom{s_m n+t_m}{u_1 k+v_m}
}
\sim
\alpha \cdot \left(\frac{\pi}{2}\right)^\frac{m-l+1}{2}
\cdot
\sqrt{
\frac{1}{(a - s)} \frac{S}{A}}
\cdot n^\frac{m-l+1}{2}
\cdot 2^{(a-s)n + (b-t)}\quad\text{as }n\to+\infty\,,
\end{equation}
    %
where summation is performed over all integers $k$ which satisfy all of the following conditions:
$0 \leq c_ik +d_i \leq a_i n+ b_i$ for all $i=1,\dots,l$
and $0 \leq u_j k + v_j \leq s_j n + t_j$ for all
$j=1,\dots, m$.

In the case when $m=0$, we let $s=t=0$ and $S=1$.
\end{Theorem}

\begin{proof}
The proof follows the proof of Theorem~\ref{thm:binomial:sum:asymptotics}
line-by-line. The only slight difference
in the asymptotic expression comes from the form of the factor
$R(x,n)$ in an analog of expression~\eqref{eq:first:approx}. Namely,
now we have
\begin{multline}
\label{eq:first:approx:general}
\frac{
\binom{a_1 n+b_1}{c_1 k+d_1}
\cdots
\binom{a_l n+b_l}{c_l k+d_l}
}{
\binom{s_1 n+t_1}{u_1 k+v_1}
\cdots
\binom{s_m n+t_m}{u_1 k+v_m}
}
\sim
\exp \Big(
(a_1 n+b_1) H(f_{1,1}(x,n))
+\cdots +
(a_l n+b_l) H(f_{1,l}(x,n))\Big)
\\ \times
\exp \Big(
-(s_1n+t_1) H(f_{2,1}(x,n))
-\cdots-
(s_m n+t_m) H(f_{2,m}(x,n))
\Big)
\cdot
\sqrt{R(x,n)}\,,
\end{multline}
where
\begin{align*}
  f_{1,i}(x,n) &= \frac{c_i k + d_i}{a_i n + b_i}\,,\quad i=1,\dots,l\,; \\
  f_{2,j}(x,n) &= \frac{u_j k + v_j}{s_j n + t_j} \,,\quad j=1,\dots,m\,; \\
  R(x,n) &= \frac{
\prod_{j=1}^m f_{2,j}(x,n) \cdot (1-f_{2,j}(x,n)) \cdot (s_j n + t_j)
}{
\prod_{j=1}^m f_{1,i}(x,n) \cdot (1 - f_{1,i}(x,n)) \cdot (a_i n + b_i)
}\,.
\end{align*}
Restricting $R(x,n)$ to $x\in[-(1-\epsilon)\sqrt{n}, (1-\epsilon)\sqrt{n}]$ we get
\begin{multline*}
  R(x,n)
= \frac{
\prod_{j=1}^m \big(2\pi\cdot
f_{2,j}(x,n) \cdot (1-f_{2,j}(x,n)) \cdot (s_j n + t_j)
\big)}{
\prod_{j=1}^l \big(2\pi\cdot f_{1,i}(x,n) \cdot (1 - f_{1,i}(x,n)) \cdot (a_i n + b_i)\big)
}
\\ =
(2\pi)^{m-l}\cdot
\frac{
\prod_{j=1}^m \left(\frac{1}{4}-\frac{x^2}{4n}+O\left(\frac{1}{n}\right)\right)
}{
\prod_{j=1}^l \left(\frac{\pi}{2}-\frac{\pi\cdot x^2}{2n}+O\left(\frac{1}{n}\right)\right)
}
\cdot
\frac{
\prod_{j=1}^m (s_j n + t_j)
}{
\prod_{i=1}^l (a_i n + b_i)
}
\\=
\left(\frac{\pi}{2}-\frac{\pi x^2}{2n}\right)^{m-l}
\cdot\frac{S}{A}\cdot n^{m-l}
\cdot\left(1+O\left(\frac{1}{n}\right)\right)
\,,
\end{multline*}
where $A=\prod_{i=1}^l a_i$ and
$S=\prod_{j=1}^m s_j$.
This expression gives rise to the factor
$\sqrt{
\left(\frac{\pi}{2}\right)^{m-l}
\cdot\frac{S}{A}\cdot n^{m-l}}$
generalizing the factor $\sqrt{\frac{s}{a}}$
which we get
in the particular case $m=l=1$ represented by formula~\eqref{eq:binomial:sum:asymptotics} in Theorem~\ref{thm:binomial:sum:asymptotics}.
\end{proof}

\begin{Example}
Formula~\eqref{eq:binomial:sum:asymptotics:general}
provides
an alternative proof of the asymptotics
$$
\sum_{k=1}^{n-1} \binom{n}{k}\binom{3n-4}{3k-2}
\sim \frac{1}{\sqrt{6\pi n}}\cdot 2^{4n-4}
\quad\text{as }n\to+\infty
$$
from Lemma~4.6 in~\cite{DGZZ-volumes}.
\end{Example}

\begin{Example}
The Dixon sum $S_n(p,x)$ is defined as
$$
S_n(p,x):=\sum_{k=0}^n \binom{n}{k}^p x^k,\quad n=1,2,\dots\,,
$$
see~\cite{Dixon}. Only few exact values of $S_n(p,x)$
are known, see~\cite{Ismail}. For any fixed $p\in\mathbb{N}$ formula~\eqref{eq:binomial:sum:asymptotics:general}
gives the following asymptotic expressions for
$S_n(p,1)$:
\begin{equation}
S_n(p,1)
\sim
\frac{1}{\sqrt{p}}
\cdot\left(\frac{2}{\pi}\right)^\frac{p-1}{2}
\cdot\frac{1}{n^\frac{p-1}{2}}
\cdot 2^{p n}\quad\text{as }n\to+\infty\,.
\end{equation}
\end{Example}

\subsection{Application to 2-correlators}
\label{ss:asymptotics:2:correlators}

Table~\ref{tab:sum:of:2:correlators} provides the exact values of the
sums of $2$-correlators for small genera $g$ while
Proposition~\ref{pr:sums:2:corr} below describes
the large genus asymptotic behavior of this sum.
\begin{table}[hbt]
$$
\begin{array}{|c|c|c|c|c|c|c|}
\hline
1&2&3&4&5&6&7
\\ \hline &&&&&&
\\[-\halfbls]
\frac{1}{8}
&\frac{49}{2880}
&\frac{1181}{725760}
&\frac{467}{3870720}
&\frac{33631}{4598415360}
&\frac{322873}{860823355392}
&\frac{205001}{12297476505600}
\\ [-\halfbls] &&&&&&
\\ \hline
\end{array}
$$
\caption{
\label{tab:sum:of:2:correlators}
Sums
$\sum\limits_{k=0}^{3g-1}\langle\tau_k\tau_{3g-1-k}\rangle_g$
of two-correlators for $g=1,\dots,7$.}
\end{table}
\begin{Proposition}
\label{pr:sums:2:corr}
The following asymptotic formulas hold:
\begin{equation}
\label{eq:asymptotics:for:the:sum:of:2:correlators}
\sum_{k=0}^{3g-1}\langle\tau_k\tau_{3g-1-k}\rangle_{g}
\sim
\frac{\sqrt{3}}{3}\cdot\left(\frac{2}{3}\right)^g\cdot\frac{1}{(2g+1)!!}
\sim
\frac{1}{2\sqrt{6}}
\cdot\frac{1}{g}
\cdot\left(\frac{e}{3g}\right)^g
\ \text{ as }g\to+\infty\,.
\end{equation}
\end{Proposition}

\begin{proof}
Consider the following normalization of the $2$-correlators
$\langle\tau_k\tau_{3g-1-k}\rangle_g$
introduced in~\cite{Zograf:2-correlators}:
\begin{equation*}
a_{g,k}
=\frac{(2k+1)!!\cdot(6g-1-2k)!!}{(6g-1)!!}
\cdot 24^g\cdot g!
\cdot \langle\tau_k\tau_{3g-1-k}\rangle_g\,.
\end{equation*}
The left-hand side of~\eqref{eq:asymptotics:for:the:sum:of:2:correlators}
can be rewritten in this notation as
\begin{equation}
\label{eq:sum:through:agk}
\sum_{k=0}^{3g-1}\langle\tau_k\tau_{3g-1-k}\rangle_{g}
=
\frac{(6g-1)!!}{24^g\cdot g!}
\cdot\sum_{k=0}^{3g-1}
\frac{a_{g,k}}{(2k+1)!!\cdot(6g-1-2k)!!}\,.
\end{equation}
By~\cite[Proposition 4.1]{DGZZ-volumes}
for all $g\in\mathbb{N}$ and for all integer $k$ in the range
$\{2,3,\dots,3g-3\}$ the following bounds are valid:
\begin{equation*}
1-\frac{2}{6g-1}=a_{g,1} =a_{g,3g-2}
< a_{g,k}
< a_{g,0}=a_{g,3g-1}=1\,.
\end{equation*}
These bounds combined with~\eqref{eq:sum:through:agk} imply
  %
\begin{equation}
\label{eq:sum:through:double:factorials}
\sum_{k=0}^{3g-1}\langle\tau_k\tau_{3g-1-k}\rangle_{g}
=
\frac{(6g-1)!!}{24^g\cdot g!}
\sum_{k=0}^{3g-1}
\frac{1}{(2k+1)!!\cdot(6g-1-2k)!!}
\cdot\big(1+o(1)\big)\ \text{ as }g\to+\infty\,.
\end{equation}

Passing from double factorials to factorials, collecting powers of $2$ and $3$,
and passing to binomial
coefficients we can rewrite the expression in the right-hand side of the
above relation as
\begin{multline*}
\frac{(6g-1)!!}{24^g\cdot g!}
\sum_{k=0}^{3g-1}
\frac{1}{(2k+1)!!\cdot(6g-1-2k)!!}
=
\frac{1}{24^g\cdot g!}\cdot
\frac{(6g)!}{2^{3g}\cdot(3g)!}
\sum_{k=0}^{3g-1}
\frac{2^k\cdot k!}{(2k+1)!}\cdot\frac{2^{3g-1-k}\cdot(3g-1-k)!}{(6g-1-2k)!}
\\=
\frac{1}{3^g}\cdot
\frac{1}{2^{3g+1}}\cdot
\frac{(6g)!}{g!\cdot(3g)!}
\cdot\frac{(3g-1)!}{(6g)!}
\sum_{k=0}^{3g-1}
\frac{k!\cdot(3g-1-k)!}{(3g-1)!}\cdot\frac{(6g)!}{(2k+1)!\cdot(6g-1-2k)!}
\\=
\label{eq:2:cor:intermed}
\frac{1}{3^g}\cdot
\frac{1}{2^{3g+1}}\cdot
\frac{1}{g!\cdot 3g}
\sum_{k=0}^{3g-1}
\frac{\binom{6g}{2k+1}}{\binom{3g-1}{k}}
\,.
\end{multline*}
From Theorem~\ref{thm:binomial:sum:asymptotics} with $a=6$, $b=0$, $c=2$, $d=1$, $s=3$, $t=-1$, $u=1$ and $v=0$ we obtain
\begin{equation}
\label{eq:sum:of:ratios:of:binomials}
\sum_{k=0}^{3g-1}
\frac{\binom{6g}{2k+1}}{\binom{3g-1}{k}}
\sim
2^{3g + 1} \cdot 3 \cdot \sqrt{\frac{\pi}{6} \cdot \frac{1}{2}} \cdot \sqrt{g}
\sim
\frac{(g!)^2}{(2g)!} \cdot
2^{5g} \cdot \sqrt{3}
\end{equation}
where the second equivalence is obtained by Stirling's formula.

Combining the above equalities we can
rewrite~\eqref{eq:sum:through:double:factorials} as
$$
\sum_{k=0}^{3g-1}\langle\tau_k\tau_{3g-1-k}\rangle_{g}
\sim
\frac{1}{3^g}\cdot
\frac{1}{2^{3g+1}}\cdot
\left(\frac{1}{\sqrt{2\pi g}}\cdot\left(\frac{e}{g}\right)^g\right)
\cdot
\frac{1}{3g}
\cdot
\left(2^{3g + 1} \cdot 3 \cdot \sqrt{\frac{\pi}{6} \cdot \frac{1}{2}} \cdot \sqrt{g}\right)
\sim
\left(\frac{e}{3g}\right)^g\cdot\frac{1}{2\sqrt{6}}
\cdot\frac{1}{g}
$$
and also as
\begin{multline*}
\sum_{k=0}^{3g-1}\langle\tau_k\tau_{3g-1-k}\rangle_{g}
=
\frac{1}{3^g}\cdot
\frac{1}{2^{3g+1}}\cdot
\frac{1}{g!\cdot 3g}\cdot
\frac{\sqrt{3}\cdot 2^{5g}\cdot(g!)^2}{(2g)!}
\cdot\big(1+o(1)\big)
\\=
\frac{\sqrt{3}}{3}\cdot\left(\frac{2}{3}\right)^g\cdot
\frac{1}{2g}\cdot
\frac{2^{g}\cdot g!}{(2g)!}\cdot
\frac{2g}{2g+1}
\big(1+o(1)\big)
=
\frac{\sqrt{3}}{3}\cdot\left(\frac{2}{3}\right)^g\cdot\frac{1}{(2g+1)!!}
\cdot\big(1+o(1)\big)
\ \text{ as }g\to+\infty\,.
\end{multline*}
\end{proof}


\end{document}